\documentclass{amsart}
\usepackage{mathptmx}
\usepackage{amsmath,amsthm}
\usepackage{amscd,tikz}
\usepackage{amsfonts,enumerate}
\usepackage{amssymb}
\usepackage{graphicx,cite,adjustbox}
\usepackage{mathrsfs}
\usepackage[all,cmtip,2cell]{xy}

\UseTwocells
\xyoption{poly}

\newcommand{\lel}{\mathrel{\leqslant_\ell}}
\newcommand{\ler}{\mathrel{\leqslant_{r}}}

\newtheorem{thm}{Theorem}[section]
\newtheorem{lem}[thm]{Lemma}
\newtheorem{pro}[thm]{Proposition}

\theoremstyle{definition}
\newtheorem{dfn}{Definition}[section]
\theoremstyle{remark}
\newtheorem{rmk}{Remark}[section]

\DeclareMathOperator{\im}{im}

\title[A Tale of Two Categories: Inductive groupoids and Cross-connections]{A Tale of Two Categories:\\
Inductive groupoids and Cross-connections}
\author{P. A. Azeef Muhammed}
\address{Department of Mathematics and Natural Sciences, Prince Mohammad Bin Fahd University, Al Khobar 31952, Kingdom of Saudi Arabia.}
\address{Institute of Natural Sciences and Mathematics,
Ural Federal University, Lenina 51, 620000 Ekaterinburg, Russia}
\email{azeefp@gmail.com}
\author{Mikhail V. Volkov}
\address{Institute of Natural Sciences and Mathematics,
Ural Federal University, Lenina 51, 620000 Ekaterinburg, Russia}
\email{m.v.volkov@urfu.ru}

\begin{document}

\begin{abstract}
A groupoid is a small category in which all morphisms are isomorphisms. An inductive groupoid is a specialised groupoid whose object set is a regular biordered set and the morphisms admit a partial order. A normal category is a specialised small category whose object set is a strict preorder and the morphisms admit a factorisation property. A pair of `related' normal categories constitutes a cross-connection. Both inductive groupoids and cross-connections were identified by Nambooripad as categorical models of regular semigroups. We explore the inter-relationship between these seemingly different categorical structures and prove a direct category equivalence between the category of inductive groupoids and the category of cross-connections.
\end{abstract}

\maketitle

\section{Introduction}

In the 1880s, Sophus Lie introduced pseudogroups\footnote{Lie termed them `infinite continuous groups' as opposed to `finite continuous groups', i.e., Lie groups in modern terminology.} as generalisations of Lie groups appropriate in the context of his work in geometries of infinite dimension; see \cite[Chapter~1]{lawson}. After algebra had managed to exempt the idea of a group from its geometric cradle and developed the abstract concept of a group, the quest for the abstract structure represented by pseudogroups began. This quest led to two main solutions in the 1950s. The first solution proposed independently by Wagner \cite{wagner} and Preston \cite{invpres} involved the introduction of inverse semigroups. The second solution provided by Ehresmann \cite{ehrs,ehrs1} used the categorical idea of a groupoid. Later, Schein \cite{schein,schein1} connected Ehresmann's work on differential geometry with Wagner's ideas on inverse semigroups to provide a structure theorem for inverse semigroups using groupoids.

Incidentally, inverse semigroups, born around the 1950s, were a special instance of a much more general concept invented earlier and for a completely different purpose. In the 1930s, von Neumann \cite{rr,cg1} introduced and made a deep study of regular rings in his ground-breaking work on continuous geometry. A ring $R$ is said to be (von Neumann) regular if for every $a\in R$, there exists $b\in R$ such that $aba=a$. Observe that the regularity of a ring is in fact a property of its multiplicative semigroup so that it is fair to say that von Neumann introduced regular semigroups as well, even though he did not use the latter term. Beyond the class of multiplicative semigroups of regular rings, examples of regular semigroups are plentiful and include many natural and important objects such as the semigroup of all transformations of a set. Certain species of regular semigroups were studied as early as 1940; see \cite{clifinv,rees}, and by the 1970s, studying regular semigroups became a hot topic in the blossoming area of the algebraic theory of semigroups.

Regular semigroups arguably form the most general class of semigroups which admits a notion of inverse elements that naturally extends the corresponding notion for groups. Namely, two elements $a,b$ of a semigroup $S$ are said to be inverses of each other if $aba=a$ and $bab=b$. It is well known and easy to see that a semigroup is regular if and only if each of its elements has an inverse. Inverse semigroups can be defined as semigroups in which every element has a unique inverse. In a given inverse semigroup $S$, the map that assigns to each element $a$ its unique inverse $a^{-1}$ can be seen to constitute an involutary anti-automorphism of $S$, similarly to the case of groups. Hence, the usual right-left duality reduces to a mere isomorphism: if $S^{op}$ denotes the dual\footnote{Recall that $S^{op}$ is defined on the same set as $S$ but the multiplication $\circ$ on $S^{op}$ is defined by $a\circ b:=b\cdot a$, where $\cdot$ stands for the multiplication in $S$.} of $S$, then $S$ is isomorphic to $S^{op}$ under the map $a\mapsto a^{-1}$, whenever $S$ is inverse. Thus, an inverse semigroup seen from the left looks the same as from the right. As we move to general regular semigroups, this inbuilt symmetry is lost, and the right-left duality may become highly nontrivial. In particular, providing structure theorems for regular semigroups using categories requires inventing categorical structures that would be less symmetric than groupoids on the one hand but still possess some intrinsic duality on the other hand, quite a challenging task!

There are two successful approaches to this task, both due to Nambooripad. They rely on the idea of replacing a `too symmetric' object by a couple of interconnected `unilateral' objects. We observe in passing that this idea has got other interesting incarnations in categorical algebra. As examples, consider Loday's approach to Leibniz algebras \cite{loday} or the recent notion of a constellation studied by Gould, Hollings, and Stokes \cite{gouldrestr,Gouldstokes2017,Stokes2017}\footnote{Leibniz algebras are non-anticommutative versions of Lie algebras. A standard way to produce a Lie algebra is to define Lie bracket $[x,y]:=xy-yx$ on an associative algebra $A$; the bracket operation is clearly anticommutative, that is, $[x,y]=-[y,x]$. Loday `splits' the multiplication in $A$ into two by considering two associative operations: the `right' product $x\triangleright y$ and the `left' product $x\triangleleft y$ which are interconnected by certain axioms. The axioms ensure that the bracket $[x,y]:=x\triangleright y-y\triangleleft x$ satisfies the Leibniz law $\bigl[[x,y],z\bigr]=\bigl[[x,z],y\bigr]+\bigl[x,[y,z]\bigr]$ while the anticommutativity may fail. In a similar way, constellations are partial algebras that are one-sided generalisations of categories.}.

Nambooripad's first approach \cite{mem} takes as its starting point the structure of idempotents in a semigroup. Recall that an element $e$ of a semigroup $S$ is called an idempotent if $e^2=e$. On the set $E(S)$ of all idempotents of $S$, one can define the relation $\leqslant $ letting $e\leqslant  f$ for $e,f\in E(S)$ if and only if $e=ef=fe$. It is easy to see that $\leqslant $ is a partial order on $E(S)$. In inverse semigroups idempotents commute (in fact, it is this property that specifies inverse semigroups within the class of regular semigroups) whence $(E(S),\leqslant)$ is a semilattice. This semilattice plays a crucial role in the structure theory of inverse semigroups. Nambooripad \cite{mem} considered general semigroups, for which he `split' the partial order $\leqslant $ into two interconnected preorders. This led him to the notion of a biordered set as the abstract model of the idempotents of an arbitrary semigroup. Nambooripad isolated a property that characterises biordered sets of regular semigroups and developed the notion of an inductive groupoid on the base of this characterisation. This way, the category $\mathbf{IG}$ of inductive groupoids arose, the first of the two categories being the objects of the present paper. Using this category, Nambooripad generalised Schein's work to regular semigroups and proved a category equivalence between the category of regular semigroups and the category $\mathbf{IG}$.

Later, Nambooripad \cite{cross}, building on an alternative approach initiated by Hall \cite{hall} and Grillet \cite{grilcross}, introduced the notion of a normal category as the abstract categorical model of principal one-sided ideals of a regular semigroup. Each regular semigroup $S$ gives rise to two normal categories: one that models the principal left ideals of $S$ and another one that corresponds to the principal right ideals. In the treatise~\cite{cross}, which was --- quoting Meakin and Rajan \cite{jmar} --- ``somewhat reminiscent of von Neumann's foundational work on regular rings'', Nambooripad devised a structure called a cross-connection which captured the non-trivial interrelation between these two  normal categories. Cross-connections also form a category, denoted by $\mathbf{CC}$, which constitutes the second main object of the present paper. Nambooripad mastered the category $\mathbf{CC}$ as an alternate technique to describe regular semigroups; namely, he proved that the category of regular semigroups is equivalent to $\mathbf{CC}$.

The two approaches of Nambooripad seemed unrelated if not orthogonal to each other. However, the present authors \cite{indcxn1} established an equivalence between inductive groupoids and cross-connections that `bypasses' regular semigroups in the sense that the equivalence between $\mathbf{IG}$ and $\mathbf{CC}$ exposited in \cite{indcxn1} was not a mere composition of the aforementioned categorical equivalences found in~\cite{mem,cross}. Still, the equivalence from \cite{indcxn1} remained in the realm of regular semigroups: what we did is that we explored the inter-relationship between the idempotent structure and ideal structure in an arbitrary regular semigroup to establish how one can be retrieved from the other. In the present paper, we make one further step. Namely, we discuss inductive groupoids and cross-connections in a purely categorical setting and build upon this a direct category equivalence between the categories $\mathbf{IG}$ and $\mathbf{CC}$, thus divorcing ourselves completely from a semigroup setting.

The rest of the paper is divided into five sections. In Section \ref{secpre}, we briefly discuss some preliminaries needed for the sequel; in particular, we introduce inductive groupoids and cross-connections. In the next section, we construct the inductive groupoid associated with a given cross-connection and build a functor $\mathbb{I}\colon\mathbf{CC}\to \mathbf{IG}$. In Section \ref{cxnind}, we construct a cross-connection from an inductive groupoid and the corresponding functor $\mathbb{C}\colon \mathbf{IG}\to \mathbf{CC}$. In Section~\ref{sec:equiv}, we verify that the functor $\mathbb{I}\mathbb{C}$ is naturally isomorphic to the functor $1_\mathbf{CC}$ and the functor $\mathbb{C}\mathbb{I}$ is naturally isomorphic to the functor $1_\mathbf{IG}$, thus establishing the category equivalence between $\mathbf{IG}$ and $\mathbf{CC}$. The final section re-discusses the results and outlines some possible developments.

\section{Preliminaries}\label{secpre}
We assume the reader's acquaintance with basic notions of category theory \cite{mac}. As mentioned, the ideas discussed in the paper arose in the realm of regular semigroups; however, here we deal with them in the realm of categories only. So, although a semigroup background may be helpful, it is not a prerequisite for understanding the constructions in the paper. The reader interested in a more detailed presentation of the genesis of the concepts of an inductive groupoid and a cross-connection and their role within semigroup theory may find rather a self-contained account of this material in~\cite[Sections~1 and~2]{indcxn1}.

Our basic notational conventions are the following. For a small category $\mathcal C$, its set of objects is denoted by $v\mathcal C$ and its set of morphisms is denoted by $\mathcal{C}$ itself. For $c,d\in v\mathcal{C}$, the set all morphisms from $c$ to $d$ is denoted by $\mathcal{C}(c,d)$. We compose functions and morphisms from left to right so that in expressions like $fg$ or $\gamma\ast\delta$ etc., the left factor applies first.

\subsection{Biordered sets}
Let $E$ be a set with a partial binary operation denoted by juxtaposition. Let $D_E\subseteq E\times E$ stand for the domain of the partial operation. Define two binary relations \raisebox{-3.5pt}{\begin{tikzpicture}\draw[>-, thick]   (0,0)node[anchor=east]{} -- (.5,0)node[anchor=west]{};\draw[->, thick]   (1.3,0)node[anchor=east]{ and } -- (1.8,0)node[anchor=west]{};\end{tikzpicture}} on the set $E$ as follows:
\begin{center}
\begin{tikzpicture}
\draw[>-, thick]   (0,0)node[anchor=east]{$e$} -- (.5,0)node[anchor=west]{$f$};
\draw (3,0) node{$\iff (e,f)\in D_E$ and $ef=e$;};
\draw[->, thick]   (6,0)node[anchor=east]{$e$} -- (6.5,0)node[anchor=west]{$f$};
\draw (9.1,0)node{$\iff (f,e)\in D_E$ and $fe=e$.};
\end{tikzpicture}
\end{center}
We use the notation\raisebox{-3.5pt}{\begin{tikzpicture}\draw[>-<, thick]   (0,0)node[anchor=east]{} -- (.5,0)node[anchor=west]{};\draw[<->, thick]   (1.3,0)node[anchor=east]{ and } -- (1.8,0)node[anchor=west]{};\end{tikzpicture}}for the `symmetric versions' of respectively\begin{tikzpicture}\draw[>-, thick]   (0,0)node[anchor=east]{} -- (.5,0)node[anchor=west]{};\end{tikzpicture}and \raisebox{-3.5pt}{\begin{tikzpicture}\draw[->, thick]   (1.3,0)node[anchor=east]{} -- (1.8,0)node[anchor=west]{;};\end{tikzpicture}} that is, \raisebox{-5.5pt}{\begin{tikzpicture}\draw[>-<, thick]   (0,0)node[anchor=east]{} -- (.5,0)node[anchor=west]{ := }; \draw[>-, thick]   (1.2,0)node[anchor=east]{} -- (1.7,0)node[anchor=west]{$\cap$};  \draw[>-, thick]   (2.4,0)node[anchor=east]{(} -- (2.9,0)node[anchor=west]{$)^{-1}$};\end{tikzpicture}} and \raisebox{-5.5pt}{\begin{tikzpicture}\draw[<->, thick]   (0,0)node[anchor=east]{} -- (.5,0)node[anchor=west]{ := }; \draw[->, thick]   (1.2,0)node[anchor=east]{} -- (1.7,0)node[anchor=west]{$\cap$};  \draw[->, thick]   (2.4,0)node[anchor=east]{(} -- (2.9,0)node[anchor=west]{$)^{-1}$.};\end{tikzpicture}} Also, we let \raisebox{-5.5pt}{\begin{tikzpicture}\draw[>->, thick]   (0,0)node[anchor=east]{} -- (.5,0)node[anchor=west]{ := }; \draw[>-, thick]   (1.2,0)node[anchor=east]{} -- (1.7,0)node[anchor=west]{$\cap$};  \draw[->, thick]   (2.2,0)node[anchor=east]{} -- (2.7,0)node[anchor=west]{};\end{tikzpicture}}.

Recall that a \emph{preorder} is a reflexive and transitive binary relation. The partial algebra $E$ as above is said to be a \emph{biordered set} if the following axioms hold for all $e,f,g \in E$:

\begin{enumerate}
\item [(B1)] both\raisebox{-7pt}{%
\begin{tikzpicture}
\draw[>-, thick](.1,0)node[anchor=east]{} -- (.6,0)node[anchor=west]{};
\draw[->, thick](1.4,0)node[anchor=east]{ and } -- (1.9,0)node[anchor=west]{}; 		
\draw(3.7,0)node{ are preorders, and $D_E=$};
\draw[>-, thick](5.5,0)node[anchor=east]{} -- (6,0)node[anchor=west]{$\cup$};
\draw[->, thick](6.5,0)node[anchor=east]{} -- (7,0)node[anchor=west]{$\cup$};
\draw[>-, thick](7.7,0)node[anchor=east]{(} -- (8.2,0)node[anchor=west]{$\cup$};
\draw[->, thick](8.7,0)node[anchor=east]{} -- (9.2,0)node[anchor=west]{$)^{-1};$};
\end{tikzpicture}
        }
\vspace*{.1cm}

\item [(B2)] if\raisebox{-5pt}{%
			\begin{tikzpicture}
			\node at (1.8,0) (e){$e$};
			\node at (4.2,0) (f){$f$;};
			\node at (3,0) (fe){$fe$};
			\draw[>-, thick](0,0)node[anchor=east]{$e$} -- (.5,0)node[anchor=west]{$f$, then};
			\draw[>-<, thick]   (e) -- (fe);
			\draw[>->, thick]   (fe) -- (f);
			\end{tikzpicture}
  }		
	if\raisebox{-5pt}{%
		\begin{tikzpicture}
		\node at (1.8,0) (e){$e$};
		\node at (4.2,0) (f){$f$;};
		\node at (3,0) (ef){$ef$};
		\draw[->, thick](0,0)node[anchor=east]{$e$} -- (.5,0)node[anchor=west]{$f$, then};
		\draw[<->, thick]   (e) -- (ef);
		\draw[>->, thick]   (ef) -- (f);
		\end{tikzpicture}
		}
\vspace*{.1cm}

		\item [(B3)]if\raisebox{-5pt}{%
			\begin{tikzpicture}
			\node at (1.25,1) (e){$e$};
			\node at (.5,0) (f){$f$};
			\node at (2,0) (g){$g$,};
			\draw[->, thick]   (f) -- (e);
			\draw[->, thick]   (g) -- (e);
			\draw[>-, thick]   (f) -- (g);
            \end{tikzpicture}}
			then\raisebox{-5pt}{%
		    \begin{tikzpicture}
			\node at (.5,0) (fe){$fe$};
            \node at (2,0) (ge){$ge$ \ and};
            \draw[>-, thick] (fe) -- (ge);
			\end{tikzpicture}}
$(gf)e=(ge)(fe)$;\\
if \raisebox{-5pt}{%
			\begin{tikzpicture}
			\node at (1.25,1) (e){$e$};
			\node at (.5,0) (f){$f$};
			\node at (2,0) (g){$g$,};
			\draw[>-, thick]   (f) -- (e);
			\draw[>-, thick]   (g) -- (e);
			\draw[->, thick]   (f) -- (g);
            \end{tikzpicture}}
			then\raisebox{-5pt}{%
		    \begin{tikzpicture}
			\node at (.5,0) (fe){$ef$};
            \node at (2,0) (ge){$eg$ \ and};
            \draw[->, thick] (fe) -- (ge);
			\end{tikzpicture}}
$e(fg)=(ef)(eg)$;
\vspace*{.2cm}
	
	\item [(B4)]if\raisebox{-6.pt}{%
			\begin{tikzpicture}
			\node at (.6,0) (e){$e$};
			\node at (1.6,0) (f){$f$};
			\node at (2.6,0) (g){$g$,};
			\draw[>-, thick]   (e) -- (f);
			\draw[>-, thick]   (f) -- (g);
			\node at (4.2,0){ then $f(ge)=fe$ ;};
			\end{tikzpicture}}
			if\raisebox{-6.pt}{%
			\begin{tikzpicture}
            \node at (0.6,0) (e){$e$};
			\node at (1.6,0) (f){$f$};
			\node at (2.6,0) (g){$g$,};
			\draw[->, thick]   (e) -- (f);
			\draw[->, thick]   (f) -- (g);
			\node at (4.2,0){then $(eg)f=ef$;};
			\end{tikzpicture}}
\vspace{.1cm}

\item [(B5)] if$\begin{array}{c}
 		\begin{tikzpicture}
		\node at (.5,0) (f){$f$};
		\node at (1.4,0) (e){$e$};
		\node at (2.3,0) (g){$g$};
		\draw[>-, thick]   (f) -- (e);
		\draw[-<, thick]   (e) -- (g);
		\end{tikzpicture}\\
	    \begin{tikzpicture}
 		\node at (.3,0) (ef){and $ef$};
        \node at (1.8,0)(eg){$eg$,};
        \draw[->, thick]  (ef) -- (eg);
		\end{tikzpicture}
\end{array}$
		there exists $f'\in E$ such that
\raisebox{-5pt}{	
		\begin{tikzpicture}
		\node at (7.5,0) (f1){$f'$};
		\node at (9.1,0) (g1){$g$};
		\node at (8.3,1) (e1){$e$};
		\draw[->, thick]   (f1) -- (g1);
		\draw[>-, thick]   (f1) -- (e1);
		\draw[>-, thick]   (g1) -- (e1);
		\node at (10.4,0) {and $ef'= ef$;};
				\end{tikzpicture}
	}
if$\begin{array}{c}
 		\begin{tikzpicture}
		\node at (.5,0) (f){$f$};
		\node at (1.4,0) (e){$e$};
		\node at (2.3,0) (g){$g$};
		\draw[->, thick]   (f) -- (e);
		\draw[<-, thick]   (e) -- (g);
		\end{tikzpicture}\\
	    \begin{tikzpicture}
 		\node at (.3,0) (ef){and $fe$};
        \node at (1.8,0)(eg){$ge$,};
        \draw[>-, thick]  (ef) -- (eg);
		\end{tikzpicture}
\end{array}$
		there exists $f'\in E$ such that
\raisebox{-5pt}{	
		\begin{tikzpicture}
		\node at (7.5,0) (f1){$f'$};
		\node at (9.1,0) (g1){$g$};
		\node at (8.3,1) (e1){$e$};
		\draw[>-, thick]   (f1) -- (g1);
		\draw[->, thick]   (f1) -- (e1);
		\draw[->, thick]   (g1) -- (e1);
		\node at (10.4,0) {and $f'e= fe$.};
				\end{tikzpicture}
	}
\end{enumerate}

\vspace*{.5cm}
Further, for any two elements $e,f\in E$, their \emph{sandwich set} $\mathcal{S}(e,f)$ is defined as follows:
\begin{center}
$h\in \mathcal{S}(e,f) \iff $\raisebox{-5pt}{%
\begin{tikzpicture}
\node at (1.5,0) (e){$e$};
\node at (2.3,0) (h){$h$};
\node at (3.1,0) (f){$f$};
\draw[-<, thick]   (e) -- (h);
\draw[->, thick]   (h) -- (f);
\end{tikzpicture}}
and for every $g\in E$ such that\raisebox{-5pt}{%
\begin{tikzpicture}
\node at (4.5,0) (e1){$e$};
\node at (5.3,0) (g1){$g$};
\node at (6.1,0) (f1){$f$,};
\draw[-<, thick]   (e1) -- (g1);
\draw[->, thick]   (g1) -- (f1);
\end{tikzpicture}}
one has\\[2pt]
 \begin{tikzpicture}
\draw[>-, thick](7.6,0)node[anchor=east]{$gf$} -- (8.1,0)node[anchor=west]{$hf$};
\draw[->, thick](9.9,0)node[anchor=east]{and \  $eg$} -- (10.3,0)node[anchor=west]{$eh$.};
\end{tikzpicture}
\end{center}
A biordered set $E$ is {\em regular} if for every $e,f \in E$, the sandwich set $\mathcal{S}(e,f)$ is non-empty.

Given two biordered sets $E$ and $E'$ with the domains of partial operations $D_E$ and $D_{E'}$ respectively, we define a \emph{bimorphism} as a mapping $\theta\colon E\to E'$ satisfying:
\begin{enumerate}
	\item [(BM1)] $(e,f)\in D_E \implies (e\theta,f\theta)\in D_{E'}$;
	\item [(BM2)] $(ef)\theta = (e\theta)(f\theta)$.
\end{enumerate}

Observe that a bimorphism $\theta$ necessarily preserves arrows: for example, if\raisebox{-6pt}{%
			\begin{tikzpicture}
			\draw[>-, thick](0,0)node[anchor=east]{$e$} -- (.4,0)node[anchor=west]{$f$,};
			\end{tikzpicture}}
			then\raisebox{-5pt}{%
			\begin{tikzpicture}
			\node at (1.8,0) (et){$e\theta$};
			\node at (2.9,0) (ft){$f\theta$,};
			\draw[>-, thick]   (et) -- (ft);
			\end{tikzpicture}}  and so on.
If $E$ is a regular biordered set, then a bimorphism $\theta\colon E \to E'$ is called a \emph{regular bimorphism} if it satisfies:
\begin{enumerate}
	\item [(RBM)] $(\mathcal{S}(e,f))\theta\subseteq \mathcal{S}'(e\theta,f\theta).$
\end{enumerate}

The above used arrow notation for the preorders in biordered sets was introduced by Easdown \cite{eas}. It allows one to present the axioms of a biordered set in a concise way and may be helpful for fresh readers. In the sequel, we shall refer to the partial binary operation of a biordered set as the \emph{basic product} and we mostly use the following alternative notation to denote the preorders (for ease of writing and to save the arrow $\to$ for maps and morphisms):

\vspace{.2cm}
\begin{tikzpicture}
\draw[>-, thick]   (0,0)node[anchor=east]{$e\lel f \iff e$} -- (.5,0)node[anchor=west]{$f \iff e\:f=e$};
\draw (3.5,0) node{ and };
\draw[->, thick]   (6.5,0)node[anchor=east]{$e\ler  f \iff e$} -- (7,0)node[anchor=west]{$f \iff f\:e=e$.};
\end{tikzpicture}
\vspace{.2cm}

In a given biordered set $E$ with preorders $\lel$ and $\ler$, we can easily see that the relations  $\mathscr{L} :=\ \lel\:  \cap\: (\lel )^{-1}$\raisebox{-3.50pt}{\begin{tikzpicture}\draw[>-<, thick]   (0,0)node[anchor=east]{=} -- (.5,0)node[anchor=west]{};\end{tikzpicture}} and $\mathscr{R} :=\ \ler \: \cap  \:(\ler )^{-1}$\raisebox{-3.50pt}{\begin{tikzpicture}\draw[<->, thick]   (0,0)node[anchor=east]{=} -- (.5,0)node[anchor=west]{};\end{tikzpicture}} are equivalence relations while the relation $\leqslant: =\ \lel\:  \cap \:\ler $\raisebox{-3.50pt}{\begin{tikzpicture}\draw[>->, thick]   (0,0)node[anchor=east]{=} -- (.5,0)node[anchor=west]{};\end{tikzpicture}} is a partial order.

Although the axioms of biordered sets are complicated and may appear slightly artificial, biordered sets arise quite naturally in several mathematical contexts. If a semigroup $S$ has idempotents, the set $E(S)$ of all idempotents of $S$ can be seen to form a biordered set whose partial operation is a certain restriction of the multiplication of $S$. In \cite{mem}, the biordered sets of the form $E(S)$ where $S$ is a regular semigroup were characterised by Nambooripad as regular biordered sets. Later, Easdown \cite{eas} showed that given a biordered set $E$, we can always construct a semigroup $S$ such that $E(S)$ and $E$ are isomorphic as biordered sets. Beyond semigroups, Putcha  \cite{putcha} showed that pairs of opposite parabolic subgroups of a finite group of Lie type form a biordered set.

In this paper, we do not need the explicit use of the biorder axioms except in a few proofs; nevertheless we have included the full set of axioms for the sake of completeness.

\subsection{Ordered groupoids}
The notion of an ordered groupoid was introduced by Ehresmann \cite{ehrs,ehrs1} in the context of his work on pseudogroups. Ordered groupoids are essentially groupoids such that their morphisms admit a partial order compatible with the composition. Recall that our convention is to compose the morphisms from left to right.

\begin{dfn}\label{og}
Let $\mathcal{G}$ be a groupoid and denote by $\mathbf{d}\colon\mathcal{G}\to v\mathcal{G}$ and $\mathbf{r}\colon\mathcal{G}\to v\mathcal{G}$ its domain and codomain maps, respectively. Let $\leq$ be a partial order on $\mathcal{G}$. Then $(\mathcal{G},\leq)$ is called an \emph{ordered groupoid} if the following hold for all $e,f \in v\mathcal{G}$ and all $x,y,u,v\in\mathcal{G}$.
	\begin{enumerate}
		\item [(OG1)] If $u\leq x$, $v\leq y$ and $\mathbf{r}(u)=\mathbf{d}(v)$, $\mathbf{r}(x)=\mathbf{d}(y)$, then $uv \leq xy$.
		\item [(OG2)] If $x\leq y$, then $x^{-1}\leq y^{-1}$.
		\item [(OG3)] If $1_e\leq 1_{\mathbf{d}(x)}$, then there exists a unique morphism $e{\downharpoonleft}x\in\mathcal{G}$ (called the \emph{restriction} of $x$ to $e$) such that $e{\downharpoonleft} x\leq x$ and $\mathbf{d}(e{\downharpoonleft} x) = e$.
		\item [(OG3$^*$)] If $1_f\leq 1_{\mathbf{r}(x)}$, then there exists a unique morphism $x{\downharpoonright}f\in\mathcal{G} $ (called the \emph{corestriction} of $x$ to $f$) such that $x{\downharpoonright} f \leq x$ and $\mathbf{r}(x{\downharpoonright} f) = f$.
	\end{enumerate}
\end{dfn}

Observe that in an ordered groupoid $(\mathcal{G},\leq)$, the restriction of $\leq$ to the identity morphisms in $\mathcal{G}$ induces a partial order on the set $v\mathcal{G}$ of the objects of the groupoid.
An order preserving functor $F$ between two ordered groupoids is said to be a $v$-\emph{isomorphism} if its object map $vF$ is an order isomorphism.

\subsection{Inductive groupoids}
We are approaching the definition of the category of inductive groupoids, the first of the two main objects of this paper. Roughly speaking, an inductive groupoid is an ordered groupoid whose object set is a regular biordered set and containing certain distinguished morphisms which are induced by alternating sequences of $\mathrel{\mathscr{R}}$- and $\mathrel{\mathscr{L}}$-related elements of the biordered set. To formalise this rough idea, we start with associating an ordered groupoid with any given (not necessarily regular) biordered set $E$.

An \emph{$E$-path} is a sequence $(e_1,e_2,\dots,e_n)$ of elements of $E$ such that $e_i\mathrel{(\mathscr{R}\cup\mathscr{L})}e_{i+1}$ for all $i=1,\dots,n-1$. An element $e_i$ in such an $E$-path is called \emph{inessential} if $e_{i-1}\mathrel{\mathscr{R}}e_i \mathrel{\mathscr{R}}e_{i+1}$ or $e_{i-1}\mathrel{\mathscr{L}}e_i\mathrel{\mathscr{L}}e_{i+1}$. Two $E$-paths that share the same first and last elements are said to be \emph{essentially the same} if each of them can be obtained from the other by a sequence of adding or removing inessential elements. Clearly, this defines an equivalence relation on the set of all $E$-paths. The equivalence class of an $E$-path relative to this relation is referred to as an \emph{$E$-chain}. In the sequel, we take the liberty of using expressions like `let $\mathfrak{c}=(e_1,e_2,\dots,e_n)$ be an $E$-chain', meaning, of course, the equivalence class of all $E$-paths that are essentially the same as $(e_1,e_2,\dots,e_n)$.

The set $\mathcal{G}(E)$ of all $E$-chains can be thought of as a groupoid if we consider each $E$-chain $\mathfrak{c}=(e_1,e_2,\dots,e_n)$ as a morphism with domain $e_1$ and codomain $e_n$. The composition of two $E$-chains, say, $\mathfrak{c}$ as above and $\mathfrak{c}'=(f_1,f_2,\dots,f_m)$, is defined if and only if $e_n=f_1$, and if so, then the product $\mathfrak{c}\mathfrak{c}'$ is set to be equal to the $E$-chain containing the $E$-path
\[
(e_1,e_2,\dots,e_n=f_1,f_2,\dots,f_m).
\]
The inverse of $\mathfrak{c}$ is the $E$-chain $(e_n,e_{n-1},\dots,e_1)$.

Further, we introduce a binary relation $\leq_E$ on the set $\mathcal{G}(E)$. Let $\mathfrak{c}=(e_1,e_2,\dots,e_n)$ and $\mathfrak{c'}=(f_1,f_2,\dots,f_m)$ be two $E$-chains. Suppose that $e_1\leqslant f_1$ in $E$ and define the sequence $h_1,h_2,\dots,h_m$ inductively by letting $h_1:=e_1$ and $h_i:=(f_ih_{i-1})f_i$ for each $i=2,3,\cdots,m$. Using the assumption $e_1\leqslant f_1$ and the axioms of a biordered set, it is easy to check that all $h_2,\dots,h_m$ are indeed well-defined elements of $E$, and moreover, $(h_1,h_2,\dots,h_m)$ forms an $E$-path. Now we let $\mathfrak{c}\leq_E \mathfrak{c'}$ if and only if $e_1\leqslant f_1$  and the $E$-paths  $(e_1,e_2,\dots,e_n)$ and  $(h_1,h_2,\dots,h_m)$ are essentially the same.
Alternatively,  $\mathfrak{c}\leq_E \mathfrak{c'}$ if and only if for any $E$-path $(f_1,f_2,\dots,f_m)$ in the $E$-chain $\mathfrak{c'}$, there exists an $E$-path $(h_1,h_2,\dots,h_m)$ in the $E$-chain $\mathfrak{c}$ such that $h_i\leqslant f_i$ for each $i=1,2,\dots,m$. It can be shown (see \cite[Proposition 3.3]{mem}) that $\leq_E$ is a partial order on $\mathcal{G}(E)$ and the pair $(\mathcal{G}(E),\leq_E)$ constitutes an ordered groupoid. The partial order $\leq_E$ may be seen as an extension of the natural partial order $\leqslant$ of the biordered set $E$ to the set $\mathcal{G}(E)$. 

Given a biordered set $E$, a $2\times 2$ matrix $\bigl[\begin{smallmatrix} e&f\\ g&h \end{smallmatrix}\bigr]$ of elements of $E$ such that
\[
e\mathrel{\mathscr{R}} f\mathrel{\mathscr{L}}h\mathrel{\mathscr{R}} g\mathrel{\mathscr{L}}e
\]
or, in Easdown's arrow notation,
\begin{center}
\begin{tikzpicture}
\node at (0,0) (e){$e$};
\node at (1.5,0) (f){$f$};
\node at (0,-1.5) (g){$g$};
\node at (1.5,-1.5,0) (h){$h$};
\draw[<->, thick]   (e) -- (f);
\draw[>-<, thick] (e) -- (g);
\draw[<->, thick]   (g) -- (h);
\draw[>-<, thick] (f) -- (h);
\end{tikzpicture}
\end{center}
stands for the $E$-path $(e,f,h,g,e)$. We refer to the $E$-chain corresponding to this $E$-path as an \emph{$E$-square} and allow ourselves expressions like `$\bigl[ \begin{smallmatrix} e&f\\ g&h \end{smallmatrix} \bigr]$ forms an $E$-square'. 

Given $e,g,h$ in a biordered set $E$ such that $g,h \lel e$ and $g\mathrel{\mathscr{R}} h$, one can easily deduce from axioms (B2) and (B3) that
$\bigl[ \begin{smallmatrix} g&h\\ eg&eh \end{smallmatrix} \bigr]$
forms an $E$-square. Such an $E$-square is called \emph{row-singular}. Dually, we define a \emph{column-singular} $E$-square, and an $E$-square is said to be \emph{singular} if it is either row-singular or column-singular.

Given an ordered groupoid $\mathcal{G}$ and a $v$-isomorphism $\epsilon\colon \mathcal{G}(E)\to\mathcal{G}$, an $E$-square $\bigl[ \begin{smallmatrix} e&f\\ g&h \end{smallmatrix} \bigr]$
is said to be $\epsilon$-commutative if
\begin{equation}
\label{eq:commut}
\epsilon(e,f)\epsilon(f,h) =\epsilon(e,g)\epsilon(g,h).
\end{equation}
Observe that to simplify notation in~\eqref{eq:commut}, we denoted the image of the $E$-chain $(e,f)$ under $\epsilon$ by $\epsilon(e,f)$ rather than $\epsilon((e,f))$ and did so also for the other $E$-chains that occur in~\eqref{eq:commut}. This convention, of leaving out unnecessary braces when there is no ambiguity regarding the expression, shall be followed in the sequel.

\begin{dfn}\label{dfnig}
	Let $E$ be a regular biordered set, let $\mathcal{G}$ be an ordered groupoid, and let $\epsilon\colon \mathcal{G}(E)\to \mathcal{G}$ be a $v$-isomorphism, called an \emph{evaluation functor}. We say that $(\mathcal{G},\epsilon)$ forms an \emph{inductive groupoid} if the following axioms and their duals hold.
	\begin{enumerate}
		\item[(IG1)] Let $x\in \mathcal{G}$ and for $i=1,2$, let $e_i$, $f_i \in E$ be such that $e_1\ler  e_2$, $\epsilon (e_i) \leq \mathbf{d}(x)$ and $\epsilon (f_i)= \mathbf{r}(\epsilon (e_i){\downharpoonleft} x)$. Then $f_1\ler  f_2$, and
		$$\epsilon (e_1,e_1e_2)(\epsilon (e_1e_2){\downharpoonleft} x) = (\epsilon (e_1){\downharpoonleft} x)\epsilon (f_1,f_1f_2).$$
		\item[(IG2)] All singular $E$-squares are $\epsilon $-commutative.
	\end{enumerate}
\end{dfn}

Let $(\mathcal{G},\epsilon)$ and $(\mathcal{G}',\epsilon')$ be two inductive groupoids with biordered sets $E$ and $E'$ respectively. Suppose that  $F\colon \mathcal{G}\to\mathcal{G}'$ is an   order preserving functor such that
its object map $vF\colon E\to E'$ is a regular bimorphism of biordered sets. Then $vF$ induces a unique order preserving functor $\mathcal{G}(vF)$ between the ordered groupoids of $E$-chains $\mathcal{G}(E)$ and $\mathcal{G}(E')$ defined as follows; see \cite[Proposition 3.3]{mem}. For every $E$-chain $\mathfrak{c}=(e_1,e_2,\dots,e_n)$ in $\mathcal{G}(E)$,
$$\mathcal{G}(vF)(\mathfrak{c}):= (vF(e_1),vF(e_2),\dots,vF(e_n)).$$
The order preserving functor $F$ is said to be \emph{inductive} if the following diagram commutes.
\begin{equation*}
\xymatrixcolsep{6pc}\xymatrixrowsep{4pc}\xymatrix
{
	\mathcal{G}(E) \ar[r]^{\mathcal{G}(vF)} \ar[d]_{\epsilon }
	& \mathcal{G}(E') \ar[d]^{\epsilon '} \\
	\mathcal{G} \ar[r]^{F} & \mathcal{G}'
}
\end{equation*}

\begin{pro}[\!{\cite[Remark 3.1]{mem}}]\label{procat1}
Inductive groupoids with inductive functors as morphisms form a category.	
\end{pro}

We denote the category of inductive groupoids with inductive functors by $\mathbf{IG}$.

\subsection{Normal categories}
Now, we proceed to introduce the second main object of this paper: the category of cross-connections. Again, the construction requires several steps.  We begin by discussing normal categories. 

Recall that a preorder $\mathcal{P}$ is said to be \emph{strict} if identity morphisms are the only isomorphisms in $\mathcal{P}$. Observe that a small preorder is strict if and only if it is induced by a partially ordered set.

Let $\mathcal{C}$ be a small category and $\mathcal{P}$ a subcategory of $\mathcal{C}$ such that $\mathcal{P}$ is a strict preorder category with $v\mathcal{P} = v\mathcal{C}$. The pair $(\mathcal{C},\mathcal{P})$ is called a \emph{category with subobjects} if, first, every $f\in\mathcal{P}$ is a monomorphism in $\mathcal{C}$ and, second, if $f,g\in \mathcal{P}$ and $h \in \mathcal{C}$ are such that $f=hg$, then $h\in \mathcal{P}$.

In a category $(\mathcal{C},\mathcal{P})$ with subobjects, the morphisms in $ \mathcal{P}$ are called \emph{inclusions}. We write $c'\subseteq c$ if there is an inclusion $c'\to c$, and we denote this inclusion by $j(c',c)$. An inclusion $j(c',c)$ \emph{splits} if there exists a morphism $q\colon c\to c' \in \mathcal{C}$ such that $j(c',c)q =1_{c'}$. In this situation, the morphism $q$ is called a \emph{retraction}.

A \emph{normal factorization} of a morphism $f\colon c\to d$ in $\mathcal{C}$ is a factorization of the form $f=quj$ where $q\colon c\to c'$ is a retraction, $u\colon c'\to d'$ is an isomorphism, and $j=j(d',d)$ is an inclusion, where $c',d' \in v\mathcal{C}$ are such that $c' \subseteq c$, $ d'\subseteq d$. Figure \ref{fignf} represents the normal factorisation property.
	
	\begin{figure}[ht]
		\centering
		\xymatrixcolsep{1.5pc}\xymatrixrowsep{1.5pc}
		\xymatrix{
			&&&&&c \ar@{->}[rrr]^{f} \ar@{-->}[dd]_{q}&&&d\\
			\\
			&&&&&c' \ar@{->}@/^1pc/[rrr]^{u} &&&d'\ar@{-->}[uu]_{j}
		}
		\caption{Normal factorisation of a morphism $f$}\label{fignf}
	\end{figure}

The morphism $qu$ is called the \emph{epimorphic component} of the morphism $f$ and is denoted by $f^\circ$. It can be seen that $f^\circ$ is uniquely determined by $f$. The codomain of $f^\circ$ is called the \emph{image} of the morphism $f$ and is denoted by $\im f$. The following properties of epimorphic components may prove crucial in the sequel:
\begin{pro}[\!{\cite[Corollary II.4]{cross}}]\label{proepi}
Let $\mathcal{C}$ be a category with normal factorisation property where inclusions split.
\begin{enumerate}
    \item If $f$ and $g$ are composable such that the inclusion of $f$ is $j_f= j(\im f,\mathbf{r}(f))$, then 
    $$(fg)^\circ=f^\circ (j_f g)^\circ.$$
    \item If $f$ is an epimorphism, then $f^\circ=f$.
\end{enumerate}
\end{pro}

\begin{dfn}\label{dfnnc}
Let $\mathcal{C}$ be a category with subobjects and $d\in v\mathcal{C}$. A map $\gamma\colon v\mathcal{C}\to\mathcal{C}$ is called a \emph{cone}\footnote{Notice that what we call a cone here was called a \emph{normal cone} in \cite{cross,indcxn1}.} from the \emph{base} $v\mathcal{C}$ to the \emph{apex} $d$ if:
\begin{itemize}
\itemindent=10pt
{
\item[(Ncone1)] $\gamma(c)$ is a morphism from $c$ to $d$ for each $c\in v\mathcal{C}$;
\item[(Ncone2)] if $c\subseteq c'$, then $j(c,c')\gamma(c') = \gamma(c)$;
\item[(Ncone3)] there exists $c\in v\mathcal{C}$ such that $\gamma(c)\colon c\to d$ is an isomorphism.
}
\end{itemize}
The apex of the cone $\gamma$ shall be denoted by $c_{\gamma}$ in the sequel. 
\end{dfn}

A cone $\gamma$ is said to be \emph{idempotent} if $\gamma(c_\gamma)= 1_{c_\gamma}$. It is easy to verify that for any cone $\gamma$ and any epimorphism $f\colon c_\gamma\to d$, the map $v\mathcal{C}\to\mathcal{C}$ defined by $a\mapsto\gamma(a)f$ is a cone with apex $d$. This cone is denoted by $\gamma*f$.

\begin{dfn}
A category $\mathcal{C}$ with subobjects is called a \emph{normal category} if the following holds.
	\begin{enumerate}
		\item [(NC1)]Any morphism in $\mathcal{C}$ has a normal factorization.
		\item [(NC2)]Every inclusion in $\mathcal{C}$ splits.
		\item [(NC3)]For each $c \in v\mathcal{C} $, there is an idempotent cone with apex $c$.
	\end{enumerate}
\end{dfn}
Natural examples of a normal category include the powerset category (subsets of a set with functions as morphisms) \cite{tx}, the subspace category  (subspaces of a vector space with linear transformations as morphisms) \cite{tlx}, etc. 

\subsection{Normal dual}
The \emph{normal dual} $N^\ast\mathcal C$ of a normal category $\mathcal{C}$ is a full subcategory of the category $\mathcal{C}^\ast$ of all functors from $\mathcal C$ to the category $\bf{Set}$. The objects of $N^\ast\mathcal C$ are certain functors and the morphisms are natural transformations between them. Namely, for each cone $\gamma$ in $\mathcal{C}$, we define a functor (called an \emph{H-functor}  and denoted by $H({\gamma};-)$) from $\mathcal{C}$ to $\mathbf{Set}$ as follows. For each $c\in v\mathcal{C}$ and for each $g\in \mathcal{C}(c,d)$,
		\begin{align*}
		H({\gamma};{c})&\text{ is the set } \{\gamma\ast f^\circ : f \in \mathcal{C}(c_{\gamma},c)\}\ \text{ and }\\
		H({\gamma};{g})&\text{ is the map } H({\gamma};{c})\to H({\gamma};{d})\ \text{ given by }\gamma\ast f^\circ \mapsto \gamma\ast (fg)^\circ.
		\end{align*}
	
We define the \emph{M-set} of a cone $\gamma$ as
\begin{equation*}
M\gamma: = \{ c \in \mathcal{C}:\gamma(c)\text{ is an isomorphism} \}.
\end{equation*}
It can be shown that if two $H$-functors $H({\gamma};-)$ and $H({\gamma'};-)$ are equal, so are the $M$-sets of the cones $\gamma$ and $\gamma'$. Hence we can define the $M$-set of an $H$-functor as $MH(\gamma;-): = M\gamma$.

It can be seen that $H$-functors are {representable functors} such that for a cone $\gamma$, there is a natural isomorphism $\eta_\gamma\colon  H(\gamma;-) \to \mathcal{C}(c_\gamma,-)$. Here $\mathcal{C}(c_\gamma,-)$ is the hom-functor determined by $c_\gamma$.

It can be shown that the normal dual of a normal category is, in fact, normal (see \cite[Section III.4.2]{cross}). The proof is quite non-trivial since it involves characterising the morphisms in the normal dual, which are natural transformations. In some special cases, the normal dual can be transparently described; for instance, the normal dual of the subspace category has been described via the annihilator category \cite{tlx}, wherein the algebraic duality coincides with the cross-connection duality. In general, however, such simple descriptions do not appear to be possible. 

\subsection{Cross-connections}
An \emph{ideal} of a normal category $\mathcal{C}$ generated by its object $c$ is the full subcategory of $\mathcal{C}$, denoted $\langle c\rangle$, whose objects are given by
$$v\langle c \rangle :=\{d\in v\mathcal{C}: d\subseteq c\}.$$

A functor $F$ between two normal categories $\mathcal{C}$ and $\mathcal{D}$ is said to be a \emph{local isomorphism} if $F$ is inclusion preserving, fully faithful and for each $c\in v\mathcal{C}$, the restriction $F_{|\langle c \rangle}$ of $F$ to the ideal $\langle c\rangle$ is an isomorphism of $\langle c \rangle$ onto the ideal $\langle F(c) \rangle$.
\begin{dfn} \label{ccxn}
	Let $\mathcal{C}$ and $\mathcal{D}$ be normal categories. A \emph{cross-connection} from $\mathcal{D}$ to $\mathcal{C}$ is a triplet $(\mathcal{D},\mathcal{C};{\Gamma})$ (often denoted by just $\Gamma$), where $\Gamma\colon  \mathcal{D} \to N^\ast\mathcal{C}$ is a local isomorphism such that for every $c \in v\mathcal{C}$, there is some $d \in v\mathcal{D}$ such that $c \in M\Gamma(d)$.
\end{dfn}

Observe that in the above definition, $M\Gamma(d)$ is the $M$-set of the $H$-functor $\Gamma(d)$. Given a cross-connection $\Gamma$ between two normal categories $\mathcal{C}$ and $\mathcal{D}$, we define the set $E_\Gamma$ as:
\begin{equation}\label{eqb1}
E_\Gamma := \{ (c,d) \in v\mathcal{C}\times v\mathcal{D} \text{ such that } c\in M\Gamma(d) \}.
\end{equation}

For a cross-connection $\Gamma$ from $\mathcal{D}$ to $\mathcal{C}$, it can be shown that there is a unique \emph{dual cross-connection} $\Delta=(\mathcal{C},\mathcal{D};{\Delta})$ from $\mathcal{C}$ to $\mathcal{D}$ such that $(c,d) \in E_\Gamma$ if and only $(d,c) \in E_\Delta$. Then, for $(c,d)\in E_\Gamma$, there is a unique idempotent cone $\epsilon$ in $\mathcal{C}$ such that $c_\epsilon=c$ and $H(\epsilon;-)=\Gamma(d)$; this cone is denoted by $\gamma(c,d)$, in the sequel. Similarly $\delta(c,d)$ denotes a unique idempotent cone in $\mathcal{D}$ such that $(d,c)\in E_\Delta$.

Given a cross-connection $(\mathcal{D},\mathcal{C};{\Gamma})$ with dual $\Delta$, $(c,d),(c',d') \in E_\Gamma$, $f\in \mathcal{C}(c,c')$ and $f^*\in \mathcal{D}(d',d)$, the morphism $f^*$ is called the \emph{transpose} of $f$ if $f$ and $f^*$ make the following diagram commute:
\begin{equation*}\label{trans}
\xymatrixcolsep{3pc}\xymatrixrowsep{4pc}\xymatrix
{
	c\ar[d]_{f}& \Delta(c) \ar[rr]^{\eta_{\delta(c,d)}} \ar[d]_{\Delta(f)}
	& & \mathcal{D}(d,-) \ar[d]^{\mathcal{D}(f^*,-)}& d \\
	c'& \Delta(c') \ar[rr]^{\eta_{\delta(c',d')}} & & \mathcal{D}(d',-)&d'\ar[u]_{f^*}
}
\end{equation*}
It is worth noting that  cross-connection transposes enjoy several properties of usual matrix transposes. For instance, $f^{**}=f$ and $(fg)^*=g^*f^*$.
\begin{dfn}\label{morcxn}
Let $(\mathcal{D},\mathcal{C};{\Gamma})$ and $(\mathcal{D}',\mathcal{C}';{\Gamma}')$ be two cross-connections. A \emph{morphism of cross-connections} $m\colon \Gamma\to \Gamma'$ is a pair $m=(F_m,G_m)$ of inclusion preserving functors $F_m\colon \mathcal{C}\to \mathcal{C}'$ and $G_m\colon \mathcal{D}\to \mathcal{D}'$ satisfying the following axioms:
\begin{enumerate}
		\item [(M1)] if $(c,d)\in E_\Gamma$, then $(F_m(c),G_m(d)) \in E_{\Gamma'}$ and for all $c'\in v\mathcal{C}$,
		\[F_m(\gamma(c,d)(c'))=\gamma(F_m(c),G_m(d))(F_m(c'));\]
		\item [(M2)] if $(c,d), (c',d') \in E_\Gamma$ and $f^*\colon d'\to d$ is the transpose of $f\colon c\to c'$, then $G_m(f^*)$ is the transpose of $F_m(f)$.
	\end{enumerate}
\end{dfn}

\begin{pro}[\!{\cite[Section V.2.1]{cross}}]\label{procat2}
The cross-connections with cross-connection morphisms form a category.
\end{pro}

We denote by $\mathbf{CC}$ the category of cross-connections with cross-connection morphisms.

\section{Inductive groupoid of a cross-connection}\label{indcxn}

Recall that the aim of the present paper is to establish a category equivalence between the categories $\mathbf{IG}$ and $\mathbf{CC}$. In this section, given a cross-connection $\Gamma= (\mathcal{D},\mathcal{C};{\Gamma})$, we identify the inductive groupoid $(\mathcal{G}_\Gamma,\epsilon_\Gamma)$ associated with the cross-connection $\Gamma$. Further, we prove that this correspondence is also functorial.

\subsection{Biordered set of a cross-connection}
First, observe that for an element $(c,d)\in E_\Gamma$, we can uniquely associate with it the pair of idempotent cones $(\gamma(c,d),\delta(c,d))$.  By suitably defining  the basic products and preorders  \cite{cross}, the set $E_\Gamma$ can be realised as the regular biordered set associated with the cross-connection $\Gamma$.

Define two preorders $\lel $ and $\ler $ on $ E_\Gamma$ as follows. For any two elements $(c,d)$ and $(c',d')$ in $E_\Gamma$,
\begin{equation}\label{eqb2}
(c,d)\lel (c',d')\iff c\subseteq c'\text{  and  }(c,d)\ler (c',d')\iff d\subseteq d'.
\end{equation}
Also, we define basic products in $E_\Gamma$ as:
\begin{equation}\label{eqb3}
(c,d)(c',d')=\begin{cases}(c,d), & \text{ if } c\subseteq c';\\
(c',{\im\delta(c,d)(d') } ), & \text{ if } c'\subseteq c;\\
(c',d'), & \text{ if } d'\subseteq d;\\
(\im\gamma(c',d')(c),d),& \text{ if } d\subseteq d'.
\end{cases}
\end{equation}

Then $E_\Gamma$ as defined in (\ref{eqb1}) forms a regular biordered set with preorders and basic products as defined in (\ref{eqb2}) and (\ref{eqb3}) respectively. This biordered set $E_\Gamma$ shall serve as the set of objects $v\mathcal{G}_\Gamma$ of our required inductive groupoid $\mathcal{G}_\Gamma$.

\subsection{Ordered groupoids of a cross-connection}
Let $(c,d),(c',d') \in E_\Gamma$. Consider an isomorphism $f\colon c\to c'$. Then it has an inverse $f^{-1}\colon c'\to c$. Now let $g\colon d \to d'$ be the transpose of $f^{-1}$ relative to $\Gamma$, i.e., $g:=(f^{-1})^*$. Then by \cite[Corollary IV.23]{cross} and properties of transposes, the morphism $g$ will be an isomorphism so that $(f,g)$ will be a pair of isomorphisms in $\mathcal{C}\times\mathcal{D}$. So, a morphism in the inductive groupoid $\mathcal{G}_\Gamma$ from $ (c,d)$ to $(c',d')$ is defined as a pair $(f,g)\colon (c,d)\to(c',d')$. This is well-defined by the uniqueness of inverses and transposes for a fixed pair of objects in $E_\Gamma$.

\begin{lem}
$\mathcal{G}_\Gamma$ is a groupoid.
\end{lem}
\begin{proof}
First, observe that given two morphisms $(f,g)$ from $(c_1,d_1)$ to $(c_2,d_2)$ and $(f',g')$ from $(c_2,d_2)$ to $(c_3,d_3)$, then by the composition in $\mathcal{C}$ and $\mathcal{D}$, we have $ff'\colon c_1 \to c_3$ and $gg'\colon d_1 \to d_3$. So $(ff')^{-1}\colon c_3 \to c_1$ and by \cite[Corollary IV.22]{cross},
$$((ff')^{-1})^*=(f'^{-1}f^{-1})^*= (f^{-1})^*(f'^{-1})^* =gg'.$$
So, the composition is well-defined and associative. The morphism $(1_c,1_d)$ is the identity morphism at $(c,d) \in E_\Gamma$. Since $(f,g)\in \mathcal{G}_\Gamma$ is a pair of isomorphisms in $\mathcal{C}\times\mathcal{D}$, $(f,g)^{-1}= (f^{-1},g^{-1})$ in $\mathcal{G}_\Gamma$. Hence $\mathcal{G}_\Gamma$ is a groupoid.
\end{proof}

Now given a morphism $(f,g)$ from $(c,d)$ to $(c',d')$ and a morphism $(f_1,g_1)$ from $(c_1,d_1)$ to $(c'_1,d'_1)$, define a relation $\leq_\Gamma$ on $\mathcal{G}_\Gamma$ as follows:
\begin{align*}
(f,g) \leq_\Gamma (f_1,g_1) \iff & (c,d) \subseteq (c_1,d_1),\ 
(c',d') \subseteq (c'_1,d'_1),\\ &\text{and } (f,g)= ((j(c,c_1)f_1)^\circ,(j(d,d_1)g_1)^\circ)
\end{align*}
where $j(c,c_1)$ is the inclusion from $c$ to $c_1$ and $h^\circ$ stands for the epimorphic component of a morphism $h$ in the normal category $\mathcal{C}$.
\begin{lem}
The relation $\leq_\Gamma$ is a partial order on $\mathcal{G}_\Gamma$.
\end{lem}
\begin{proof}
Clearly $\leq_\Gamma$ is reflexive.

Let $(f,g) \leq_\Gamma (f_1,g_1)$ and $(f_1,g_1) \leq_\Gamma (f,g)$. Then $(c,d) = (c_1,d_1)$ and $(c',d') = (c'_1,d'_1)$. So $j(c,c_1)=1_{c_1}$. Also since $f_1$ is an isomorphism, $(f_1)^\circ=f_1$. Hence we have
$$(f,g)= ((j(c,c_1)f_1)^\circ,(j(d,d_1)g_1)^\circ)= ((f_1)^\circ,(g_1)^\circ)=(f_1,g_1).$$
So, $\leq_\Gamma$ is is anti-symmetric.

Let $(f,g) \leq_\Gamma (f_1,g_1)$ and $(f_1,g_1) \leq_\Gamma (f_2,g_2)$ where $(f_2,g_2)$ is a morphism from $(c_2,d_2)$ to $(c'_2,d'_2)$. Then clearly $(c,d) \subseteq (c_2,d_2), (c',d') \subseteq (c'_2,d'_2) $. Also,

\begin{equation*}
\begin{split}
(f,g)&= ((j(c,c_1)f_1)^\circ,(j(d,d_1)g_1)^\circ) \\
&=((j(c,c_1)(j(c_1,c_2)f_2)^\circ)^\circ,(j(d,d_1)(j(d_1,d_2)g_2)^\circ)^\circ)\\
&=(1_{c}((j(c,c_1)j(c_1,c_2)f_2)^\circ)^\circ,1_{d}((j(d,d_1)j(d_1,d_2)g_2)^\circ)^\circ) \ \text{ (by \cite[Corollary II.4]{cross})}\\
&=(((j(c,c_2)f_2)^\circ)^\circ,((j(d,d_2)g_2)^\circ)^\circ)\\
&=((j(c,c_2)f_2)^\circ,(j(d,d_2)g_2)^\circ).\\
\end{split}
\end{equation*}

So $(f,g) \leq_\Gamma (f_2,g_2)$, and thus $\leq_\Gamma$ is transitive. Hence $\leq_\Gamma$ is a partial order on $\mathcal{G}_\Gamma$.
\end{proof}

Observe that the partial order $\leq_\Gamma$ restricted to the identities of $\mathcal{G}_\Gamma$ reduces to the natural partial order $\leqslant $ on the biordered set $E_{\Gamma}$, and may be written as follows.
$$(1_c,1_{d}) \leq_\Gamma (1_{c_1},1_{d_1}) \iff (c,d)\subseteq({c_1},{d_1}).$$

Now define restrictions and corestrictions on $\mathcal{G}_\Gamma$ as follows. Take a morphism  $(f,g)\in\mathcal{G}_\Gamma$ from $(c_1,d_1)$ to $(c_2,d_2)$. If $(c'_1,d'_1)\subseteq(c_1,d_1)$, then $j(c'_1,c_1)$ is an inclusion in $\mathcal{C}$ from $c'_1$ to $c_1$, and $j(d'_1,d_1)$ is an inclusion in $\mathcal{D}$ from $d'_1$ to $d_1$. So,
$(j(c'_1,c_1)f,j(d'_1,d_1)g)$
is a pair of monomorphisms in $\mathcal{C} \times \mathcal{D}$ from $(c'_1,d'_1)$ to $(c_2,d_2)$. Consider their epimorphic components, say $$(f',g'):=((j(c'_1,c_1)f)^\circ,(j(d'_1,d_1)g)^\circ).$$

Observe that $f'$ and $g'$ are isomorphisms, and $(j(d'_1,d_1)g)^\circ= (((j(c'_1,c_1)f)^\circ)^{-1})^*$. So, the \emph{restriction} of $(f,g)$ to $(c'_1,d'_1)$ in $\mathcal{G}_\Gamma$ is defined as the pair $(f',g')$.

Similarly, if $(c'_2,d'_2)\subseteq(c_2,d_2)$, then $j(c'_2,c_2)$ is an inclusion in $\mathcal{C}$ from $c'_2$ to $c_2$, and $j(d'_2,d_2)$ is an inclusion in $\mathcal{D}$ from $d'_2$ to $d_2$. So,
$(j(c'_2,c_2)f^{-1},j(d'_2,d_2)g^{-1})$
is a pair of monomorphisms in $\mathcal{C} \times \mathcal{D}$ from $(c'_2,d'_2)$ to $(c_1,d_1)$. Consider their epimorphic components $((j(c'_2,c_2)f^{-1})^\circ,(j(d'_2,d_2)g^{-1})^\circ)$. These are isomorphisms with domain $(c'_2,d'_2)$. Now take their inverses, say
$$(f'',g''):=\left(((j(c'_2,c_2)f^{-1})^\circ)^{-1},((j(d'_2,d_2)g^{-1})^\circ)^{-1}\right).$$
Then, the \emph{corestriction} of $(f,g)$ to $(c'_2,d'_2)$ in $\mathcal{G}_\Gamma$ is defined as the pair of isomorphisms $(f'',g'')$ with codomain $(c'_2,d'_2)$.

\begin{thm}\label{thmog}
$(\mathcal{G}_\Gamma,\leq_\Gamma)$ is an ordered groupoid with restrictions and corestrictions defined as above.
\end{thm}
\begin{proof}
First, let $(f,g)$ be a morphism in $\mathcal{G}_\Gamma$ from $(c_1,d_1)$ to $(c_2,d_2)$, $(f_1,g_1)$ a morphism from $(c_2,d_2)$ to $(c_3,d_3)$, $(f',g')$ a morphism from $(c'_1,d'_1)$ to $(c'_2,d'_2)$ and $(f'_1,g'_1)$ a morphism  from $(c'_2,d'_2)$ to $(c'_3,d'_3)$ such that $(f',g')\leq_\Gamma(f,g)$ and $(f'_1,g'_1)\leq_\Gamma(f_1,g_1)$.

Then $(ff_1,gg_1)$ and $(f'f'_1,g'g'_1)$ are morphisms in $\mathcal{G}_\Gamma$ from $(c_1,d_1)$ to $(c_3,d_3)$ and  $(c'_1,d'_1)$ to $(c'_3,d'_3)$ respectively, such that $(c'_1,d'_1)\subseteq(c_1,d_1)$ and $(c'_3,d'_3)\subseteq(c_3,d_3)$.

We have
\begin{equation*}
\begin{split}
(f'f'_1,g'g'_1)&= ((j(c'_1,c_1)f)^\circ(j(c'_2,c_2)f_1)^\circ,(j(d'_1,d_1)g)^\circ(j(d'_2,d_2)g_1)^\circ)\\
&= ((j(c'_1,c_1)ff_1)^\circ,(j(d'_1,d_1)gg_1)^\circ) \quad \text{(using \cite[Corollary II.4]{cross})}.\\
\end{split}
\end{equation*}
Hence $(ff_1,gg_1) \leq_\Gamma (f'f'_1,g'g'_1)$, and (OG1) is satisfied.

Now if $(f,g)$ is a morphism in $\mathcal{G}_\Gamma$ from $(c_1,d_1)$ to $(c_2,d_2)$ and $(f',g')$ a morphism from $(c'_1,d'_1)$ to $(c'_2,d'_2)$ such that $(f',g')\leq_\Gamma(f,g)$ , then we need to show $(f',g')^{-1} \leq_\Gamma (f,g)^{-1}$. Clearly, $(c'_1,d'_1)\subseteq(c_1,d_1)$ and $(c'_2,d'_2)\subseteq(c_2,d_2)$.

Observe that using Proposition \ref{proepi},
$$(j(c'_2,c_2)f^{-1})^\circ(j(c'_1,c_1)f)^\circ= (j(c'_2,c_2)f^{-1}f)^\circ = (j(c'_2,c_2))^\circ =1_{c'_2},$$ and
$$(j(c'_1,c_1)f)^\circ(j(c'_2,c_2)f^{-1})^\circ= (j(c'_1,c_1)ff^{-1})^\circ = (j(c'_1,c_1))^\circ =1_{c'_1}.$$
Hence $((j(c'_1,c_1)f)^\circ)^{-1}=(j(c'_2,c_2)f^{-1})^\circ$, and similarly we show that $((j(d'_1,d_1)g)^\circ)^{-1}=(j(d'_2,d_2)g^{-1})^\circ$. So,
\begin{equation*}
\begin{split}
(f',g')^{-1}&= (f'^{-1},g'^{-1})\\
&= (((j(c'_1,c_1)f)^\circ)^{-1},((j(d'_1,d_1)g)^\circ)^{-1}) \\
&= ((j(c'_2,c_2)f^{-1})^\circ,(j(d'_2,d_2)g^{-1})^\circ).
\end{split}
\end{equation*}
Hence $(f',g')^{-1} \leq_\Gamma (f,g)^{-1}$, and (OG2) is satisfied.

Finally, if $(f,g)$ is a morphism in $\mathcal{G}_\Gamma$ from $(c_1,d_1)$ to $(c_2,d_2)$, and $(c'_1,d'_1)\subseteq(c_1,d_1)$, then we define the restriction of $(f,g)$ to $(c'_1,d'_1)$ as $((j(c'_1,c_1)f)^\circ,(j(d'_1,d_1)g)^\circ)$. So $((j(c'_1,c_1)f)^\circ,(j(d'_1,d_1)g)^\circ)\leq_\Gamma(f,g)$ and $\mathbf{d}((j(c'_1,c_1)f)^\circ,(j(d'_1,d_1)g)^\circ)=  (c'_1,d'_1)$, and (OG3) is satisfied. Similarly, we can verify (OG3*).

Hence $(\mathcal{G}_\Gamma,\leq_\Gamma)$ is an ordered groupoid.
\end{proof}

Now, since $E_{\Gamma}$ is a regular biordered set, we can build an ordered groupoid $\mathcal{G}(E_{\Gamma})$ of the $E$-chains of $E_{\Gamma}$. But to that end, we need to discuss how we can compose two cones in a normal category.

Recall that for a cone $\gamma$ in the category $\mathcal{C}$ and an epimorphism $f\colon c_\gamma \to d$, the map $\gamma*f \colon  a \mapsto\gamma(a)f$ from $v\mathcal{C}$ to $\mathcal{C}$ is a cone with apex $d$. Hence, given two cones $\gamma$ and $\sigma $, we can compose them as follows:
\begin{equation} \label{eqnsg1}
\gamma \cdot \sigma = \gamma*(\sigma(c_\gamma))^\circ
\end{equation}
where $(\sigma(c_\gamma))^\circ$ is the epimorphic component of the morphism $\sigma(c_\gamma)$.

Now, we define a partial order on the set $\mathcal{G}(E_{\Gamma})$ of the $E$-chains of $E_{\Gamma}$. First, for an $E$-chain $\mathfrak{c}=((c_1,d_1),(c_2,d_2),\dots,(c_n,d_n))$ and for $(h,k)\in E_\Gamma$ with $(h,k)\leqslant (c_1,d_1)$, let
$$(h,k)\cdot \mathfrak{c} := ((h_0,k_0),(h_1,k_1),(h_2,k_2),\dots,(h_n,k_n))$$
where $(h_0,k_0):=(h,k)$ and for $i=1,\dots,n$, the pairs $(h_i,k_i)$ are such that
\begin{subequations}
\begin{align}
\gamma(h_i,k_i)&:= \gamma(c_i,d_i)\gamma(h_{i-1},k_{i-1})\gamma(c_i,d_i)\ \text{ and }\label{eq:po1}\\
\delta(h_i,k_i)&:= \delta(c_i,d_i)\delta(h_{i-1},k_{i-1})\delta(c_i,d_i),\label{eq:po2}
\end{align}
\end{subequations}
and the cones in the right hand sides of \eqref{eq:po1} and \eqref{eq:po2} are composed as in \eqref{eqnsg1}.

Then for $E$-chains $\mathfrak{c}$ as above and  $\mathfrak{c}'=((c'_1,d'_1),(c'_2,d'_2),\dots,(c'_m,d'_m))$, define
\begin{equation*}
\mathfrak{c}\leq_E\mathfrak{c}'\iff (c_1,d_1)\subseteq(c'_1,d'_1)\
\text{ and }\mathfrak{c}= (c_1,d_1)\cdot \mathfrak{c}'.
\end{equation*}
Now, it may be verified that $(\mathcal{G}(E_{\Gamma}),\leq_E)$ forms an ordered groupoid with restriction $(h,k){\downharpoonleft}\mathfrak{c}=(h,k)\cdot \mathfrak{c}$.

Further, we define an evaluation map $\epsilon_\Gamma \colon\mathcal{G}(E_{\Gamma}) \to \mathcal{G}_\Gamma$ as follows. In the sequel, for convenience, we often denote the idempotent cones $\gamma(c_i,d_i)$ and $\delta(c_i,d_i)$ by $\gamma_i$ and $\delta_i$, respectively.
The object function of the evaluation functor is $v\epsilon_\Gamma =1_{E_{\Gamma}}$, and for an arbitrary $E$-chain $\mathfrak{c}=((c_1,d_1),(c_2,d_2),\dotsc,(c_n,d_n))$  in $\mathcal{G}(E_{\Gamma})$, we let
$$\epsilon_\Gamma(\mathfrak{c}) := ((\gamma_1\gamma_2\dotsc\gamma_n)(c_1),(\delta_1\delta_2\dotsc\delta_n)(d_1)).$$
Before proceeding, we need to verify the following lemma.
\begin{lem}
If $\mathfrak{c}=((c_1,d_1),(c_2,d_2),\dotsc,(c_n,d_n))\in\mathcal{G}(E_{\Gamma})$, then $\epsilon_\Gamma (\mathfrak{c}) \in \mathcal{G}_\Gamma$.
\end{lem}
\begin{proof}
First, observe that since $\gamma_i$ and $\delta_i$ are idempotent cones, we have $\gamma_i(c_i)=1_{c_i}$ and $\delta_i(d_i)=1_{d_i}$. Since $\mathfrak{c}$ is an $E$-chain, either $(c_{i-1},d_{i-1})\mathrel{\mathscr{L}}(c_{i},d_{i})$ or $(c_{i-1},d_{i-1}) \mathrel{\mathscr{R}} (c_{i},d_{i})$. If $(c_{i-1},d_{i-1}) \mathrel{\mathscr{L}} (c_{i},d_{i})$, then by (\ref{eqb2}), $c_{i-1}=c_i$ and so $\gamma_i(c_{i-1})=\gamma_i(c_{i})=1_{c_i}$. Otherwise, if $(c_{i-1},d_{i-1}) \mathrel{\mathscr{R}} (c_{i},d_{i})$, then by \cite[Proposition III.7]{cross},  $\gamma_i(c_{i-1})$ is an isomorphism. In either case, we see that $\gamma_i(c_{i-1})$ is an isomorphism. Similarly, we can verify that $\delta_i(d_{i-1})$ is an isomorphism. So,
\begin{equation*}
\begin{split}
(\gamma_1\gamma_2\dotsc\gamma_n)(c_1)&= \gamma_1(c_1)\ast((\gamma_2\dotsc\gamma_n)(c_1))^\circ \quad (\text{by cone multiplication: see } (\ref{eqnsg1}))\\ &=\gamma_1(c_1)\ast(\gamma_2(c_1)\ast(\gamma_3\dotsc\gamma_n)(c_2))^\circ\quad (\texttt{--"--})\\
&=\gamma_1(c_1)(\gamma_2(c_1))^\circ\ast((\gamma_3\dotsc\gamma_n)(c_2))^\circ\quad (\text{using Proposition \ref{proepi} })\\
&=\gamma_1(c_1)\gamma_2(c_1)\ast((\gamma_3\dotsc\gamma_n)(c_2))^\circ\quad  (\text{since $\gamma_2(c_{1})$ is an isomorphism})\\
&=\gamma_1(c_1)\:\gamma_2(c_1)\:\gamma_3(c_2)\:\dots\:\gamma_n(c_{n-1})\quad (\text{repeated use of previous arguments})\\
&= \gamma_2(c_1)\:\gamma_3(c_2)\:\dots\:\gamma_n(c_{n-1})\quad  (\text{since }\gamma_1(c_{1})=1_{c_1}).
\end{split}
\end{equation*}

Since each $\gamma_i(c_{i-1})$ is an isomorphism, we conclude that the morphism $(\gamma_1\gamma_2\dotsc\gamma_n)(c_1)$ is an isomorphism in $\mathcal{C}$. Similarly
$$(\delta_1\delta_2\dotsc\delta_n)(d_1)= \delta_2(d_1)\:\dots\:\delta_n(d_{n-1}),$$
and so $(\delta_1\delta_2\dotsc\delta_n)(d_1)$ is an  isomorphism in $\mathcal{D}$.

Now to show that $\epsilon_\Gamma (\mathfrak{c}) \in \mathcal{G}_\Gamma$, we need to show that
$$(((\gamma_1\gamma_2\dotsc\gamma_n)(c_1))^{-1})^*=(\delta_1\delta_2\dotsc\delta_n)(d_1).$$
But since $((fg)^{-1})^*=(f^{-1})^*(g^{-1})^*$, $\gamma(c_1,d_1)(c_1)=1_{c_1}$ and $\delta(c_1,d_1)(d_1)=1_{d_1}$, it suffices to show that
$$((\gamma_n(c_{n-1}))^{-1})^*=\delta_n(d_{n-1}) \text{ for } n=2,\dots,n.$$
Using \cite[Page 94, Equation (25) version of Lemma IV.18]{cross}, we have
$$(\gamma_{n-1}(c_{n}))^*=\delta_n(d_{n-1}).$$
So it remains to verify that
$$\gamma_{n-1}(c_{n})=(\gamma_n(c_{n-1}))^{-1}.$$

Observe that since $\mathfrak{c}$ is an $E$-chain, $\gamma_{n-1}$ and $\gamma_n$ are idempotent cones such that $\gamma_{n-1}\mathrel{\mathscr{R}}\gamma_n$ or $\gamma_{n-1}\mathrel{\mathscr{L}}\gamma_n$. If $\gamma_{n-1}\mathrel{\mathscr{L}}\gamma_n$, then $c_{n-1}=c_{n}$, so that
$$\gamma_{n-1}(c_{n}) = 1_{c_n} = (\gamma_n(c_{n-1}))^{-1}.$$

Otherwise, if $\gamma_{n-1}\mathrel{\mathscr{R}}\gamma_n$, then $d_{n-1}=d_n$ and
$$\gamma_{n-1}\gamma_n=\gamma_n\text{ and }\gamma_{n}\gamma_{n-1}=\gamma_{n-1}.$$
That is,
$$\gamma_{n-1}\gamma_n(c_{n-1})=\gamma_n\text{ and }\gamma_{n}\gamma_{n-1}(c_n)=\gamma_{n-1}.$$
Equating the first equation at the apex $c_n$ and the second equation at the apex $c_{n-1}$, we get
$$\gamma_{n-1}(c_n)\gamma_n(c_{n-1})=1_{c_n}\text{ and }\gamma_{n}(c_{n-1})\gamma_{n-1}(c_n)=1_{c_{n-1}}.$$
That is, $\gamma_{n-1}(c_{n})=(\gamma_n(c_{n-1}))^{-1}$; hence the lemma.
\end{proof}

\begin{pro}
$\epsilon_\Gamma \colon\mathcal{G}(E_{\Gamma}) \to \mathcal{G}_\Gamma$ is an evaluation functor.
\end{pro}
\begin{proof}
To prove the proposition, we need to show that $\epsilon_\Gamma$ is a $v$-isomorphism.

First, we show that $\epsilon_\Gamma$ is a functor. Suppose $\mathfrak{c},\mathfrak{c}' \in \mathcal{G}(E_{\Gamma})$ are such that $\mathfrak{c}.\mathfrak{c}'$ exists. Then
\begin{equation*}
\begin{split}
\epsilon_\Gamma(\mathfrak{c})\epsilon_\Gamma(\mathfrak{c}') =& ((\gamma_1\gamma_2\dotsc\gamma_n)(c_1),(\delta_1\delta_2\dotsc\delta_n)(d_1))((\gamma'_1\gamma'_2\dotsc\gamma'_n)(c'_1),(\delta'_1\delta'_2\dotsc\delta'_n)(d'_1))\\
=&(\gamma_1(c_1)\:\gamma_2(c_1)\:\dots\:\gamma_n(c_{n-1})\:\gamma'_1(c'_1)\:\gamma'_2(c'_1)\:\dots\:\gamma'_n(c'_{n-1}),\\
&\qquad\qquad\qquad\delta_1(d_1)\delta_2(d_1)\:\dots\:\delta_n(d_{n-1})\:\delta'_1(d'_1)\delta'_2(d'_1)\:\dots\:\delta'_n(d'_{n-1}))\\
=&(\gamma_1(c_1)\:\gamma_2(c_1)\:\dots\:\gamma_n(c_{n-1})\:\gamma'_1(c_n)\:\gamma'_2(c'_1)\:\dots\:\gamma'_n(c'_{n-1}),\\
&\qquad\qquad\qquad\delta_1(d_1)\delta_2(d_1)\:\dots\:\delta_n(d_{n-1})\:\delta'_1(d_n)\delta'_2(d'_1)\:\dots\:\delta'_n(d'_{n-1}))\\
=&((\gamma_1\gamma_2\dots\gamma_n\gamma'_1\gamma'_2\dots\gamma'_n)(c_{1}),(\delta_1\delta_2\dots\delta_n\delta'_1\delta'_2\dots\delta'_n)(d_{1}))\\
=&\epsilon_\Gamma(\mathfrak{c}.\mathfrak{c}').
\end{split}
\end{equation*}

Now, to show that $\epsilon_\Gamma$ is order preserving, it suffices to show that for an arbitrary $\mathfrak{c}=((c_1,d_1),(c_2,d_2),\dotsc,(c_n,d_n))$  in $\mathcal{G}(E_{\Gamma})$ and $(h,k) \leqslant (c_1,d_1)$,
$$\epsilon_\Gamma((h,k){\downharpoonleft}\mathfrak{c}) = (h,k){\downharpoonleft}\epsilon_\Gamma(\mathfrak{c}).$$ In the sequel, we shall denote the idempotent cones $\gamma(h,k)$, $\gamma(h_i,k_i)$ and $\delta(h_i,k_i)$ by $\theta$, $\theta_i$ and $\eta_i$ respectively. So according to our notations, $\theta_i=\gamma_i\theta_{i-1}\gamma_i$ and $\eta_i=\delta_i\eta_{i-1}\delta_i$.

First observe that since $\gamma_i\mathrel{\mathscr{L}}\gamma_{i+1}$ or $\gamma_i\mathrel{\mathscr{R}}\gamma_{i+1}$, either 
$\gamma_i\gamma_{i+1}= \gamma_i$ or $\gamma_{i+1}\gamma_i= \gamma_i$. So, $$\gamma_1\gamma_2\dots\gamma_{i-1}\gamma_i\gamma_{i-1}\dots\gamma_2\gamma_1= \gamma_1 \quad \text{ for all }i=1,\cdots,n.$$ Similarly $\delta_1\delta_2\dots\delta_{i-1}\delta_i\delta_{i-1}\dots\delta_2\delta_1= \delta_1$.

Also since $(h,k) \leqslant  (c_1,d_1)$, we have $\theta\gamma_1=\gamma_1\theta=\theta$. Now,
\begin{equation*}
\begin{split}
\theta\theta_1\theta_2\dots\theta_n&= \theta(\gamma_1\theta\gamma_1)(\gamma_2\gamma_1\theta\gamma_1\gamma_2)\dots(\gamma_n\gamma_{n-1}\dots\gamma_1\theta\gamma_1\gamma_2\dots\gamma_n)\\
&= \theta(\gamma_1)\theta(\gamma_1\gamma_2\gamma_1)\theta(\gamma_1\dots\gamma_i\gamma_{i-1}\dots\gamma_1)\theta(\gamma_1\dots\gamma_n\gamma_{n-1}\dots\gamma_1)\theta(\gamma_1\gamma_2\dots\gamma_n)\\
&= \theta(\gamma_1)\theta(\gamma_1)\theta(\gamma_1)\theta(\gamma_1)\theta(\gamma_1\gamma_2\dots\gamma_n)\\
&=\theta\gamma_1\gamma_2\dots\gamma_n.
\end{split}
\end{equation*}

Similarly
$\eta\eta_1\eta_2\dots\eta_n=\eta\delta_1\delta_2\dots\delta_n$.

So,
\begin{equation*}
\begin{split}
\epsilon_\Gamma((h,k){\downharpoonleft}\mathfrak{c}) =& \epsilon_\Gamma(\mathfrak{c}((h,k),(h_1,k_1),(h_2,k_2),\dots,(h_n,k_n)))\\
=& (\theta\theta_1\theta_2\dots\theta_n(h),\eta\eta_1\eta_2\dots\eta_n(k))\\
=&(\theta\gamma_1\gamma_2\dots\gamma_n(h),\eta\delta_1\delta_2\dots\delta_n(k))\\
=&(\theta(h)((\gamma_1\gamma_2\dots\gamma_n)(h))^\circ,\eta(k)((\delta_1\delta_2\dots\delta_n)(k))^\circ)\quad (\text{by cone multiplication: see (\ref{eqnsg1}}))\\
=&(1_h.((\gamma_1\gamma_2\dots\gamma_n)(h))^\circ,1_k((\delta_1\delta_2\dots\delta_n)(k))^\circ)\\
=&((j(h,c_1)(\gamma_1\gamma_2\dots\gamma_n)(c_1))^\circ,(j(k,d_1)(\delta_1\delta_2\dots\delta_n)(d_1))^\circ)\quad(\text{as }(h,k)\subseteq(c_1,d_1))\\
=& (h,k){\downharpoonleft}((\gamma_1\gamma_2\dotsc\gamma_n)(c_1),(\delta_1\delta_2\dotsc\delta_n)(d_1))\\
=&(h,k) {\downharpoonleft}\epsilon_\Gamma(\mathfrak{c}).
\end{split}
\end{equation*}
Hence $\epsilon_\Gamma \colon\mathcal{G}(E_{\Gamma}) \to \mathcal{G}_\Gamma$ is an evaluation functor.
\end{proof}

\begin{rmk}\label{rmkog}
Observe that in the proof of Theorem \ref{thmog} that $(\mathcal{G}_\Gamma,\leq_\Gamma)$ is an ordered groupoid, we have not used any `cross-connection' properties. Hence using the same proof, one could show that the isomorphisms in the normal categories $\mathcal{C}$ and $\mathcal{D}$, denoted by $\mathcal{G}_{\mathcal{C}}$ and $\mathcal{G}_{\mathcal{D}}$ respectively, form ordered groupoids. The one-sided subgroupoids of $\mathcal{G}_\Gamma$, i.e., the groupoids $\mathcal{G}_{\Gamma|\mathcal{C}}$ and $\mathcal{G}_{\Gamma|\mathcal{D}}$ (which are subgroupoids of $\mathcal{G}_{\mathcal{C}}$ and $\mathcal{G}_{\mathcal{D}}$ respectively) obtained by restricting $\mathcal{G}_\Gamma$ to the categories $\mathcal{C}$ and $\mathcal{D}$ respectively also form ordered groupoids. Observe that restricting all the above described groupoids to the image $\epsilon_{\Gamma}(\mathcal{G}(E_{\Gamma}))$ of the evaluation functor $\epsilon_{\Gamma}$ also give rise to ordered groupoids. Hence given a cross-connection $\Gamma$, we can associate with it several interesting ordered groupoids, which are all subgroupoids of $\mathcal{G}_{\Gamma}$.
\end{rmk}

\subsection{The inductive groupoid $\mathcal{G}_\Gamma$}
Before we proceed to prove that $(\mathcal{G}_\Gamma,\epsilon_\Gamma)$ is an inductive groupoid, we need the following important lemma concerning the ordered groupoid $(\mathcal{G}_\Gamma,\leq_\Gamma)$ which need not necessarily hold in its ordered subgroupoids. It describes the relationship between the retractions in the normal categories and the restrictions in the inductive groupoid.

\begin{lem}\label{lemog}
Let $(f,g)$ be a morphism in $\mathcal{G}_\Gamma$ from $(c,d)$ to $(c',d')$ and let $(c_1,d_1)\in E_\Gamma$ be such that $(c_1,d_1) \subseteq (c,d)$. If $(f_1,g_1)$ is the restriction $(c_1,d_1){\downharpoonleft}(f,g)$ with codomain $(c'_1,d'_1)$, then
$$(f,g)(\gamma'_1(c'),\delta'_1(d'))= (\gamma_1(c),\delta_1(d))(f_1,g_1).$$
\end{lem}
\begin{proof}
First, recall that, using \cite[Proposition IV.24]{cross},
$$(j(d'_1,d'))^*=\gamma'_1(c'),\: (j(c'_1,c'))^*=\delta'_1(d'),\:(j(d_1,d))^*=\gamma_1(c)\text{ and }(j(c_1,c))^*=\delta_1(d).$$
Also, all the morphisms on the right hand side are retractions.

Now, since $g_1= (j(d_1,d)g)^\circ$, we have
$$ g_1j(d'_1,d')= (j(d_1,d)g)^\circ j(d'_1,d')\implies g_1j(d'_1,d')= j(d_1,d)g.$$
Taking transposes we get
$$ (j(d'_1,d'))^*g_1^*= g^*(j(d_1,d))^* \implies \gamma'_1(c')(f_1)^{-1}=f^{-1}\gamma_1(c)\implies f\gamma'_1(c')=\gamma_1(c)f_1.$$

Similarly we can prove that $g\delta'_1(d')=\delta_1(d)g_1$. Hence the lemma.
\end{proof}
\begin{thm}
$(\mathcal{G}_\Gamma,\epsilon_\Gamma)$ is an inductive groupoid.
\end{thm}
\begin{proof}
	
We need to verify that $(\mathcal{G}_\Gamma,\epsilon_\Gamma)$ satisfies the axioms of Definition \ref{dfnig}. First, we need to verify (IG1) of Definition \ref{dfnig}. To this end, consider an ordered groupoid $\mathcal{G}$ and an evaluation functor $\epsilon\colon\mathcal{G}(E)\to\mathcal{G} $. Let $x\in \mathcal{G}$ and for $i=1,2$, let $e_i$, $f_i \in E$ be such that  $e_1\ler  e_2$, $\epsilon (e_i) \leq \mathbf{d}(x)$ and $\epsilon (f_i)= \mathbf{r}(\epsilon (e_i){\downharpoonleft} x)$. Then we need to verify that $f_1\ler  f_2$, and
$$\epsilon (e_1,e_1e_2)(\epsilon (e_1e_2){\downharpoonleft} x) = (\epsilon (e_1){\downharpoonleft} x)\epsilon (f_1,f_1f_2).$$
The above condition can be illustrated by the commutativity of the solid arrows in Figure~\ref{indg1}. Observe that Lemma \ref{lemog} concerns with the commutativity of the square consisting of the elements $e$, $f$, $f_1f_2$ and $e_1e_2$ in Figure~\ref{indg1}.		

\begin{figure}[h]
	\centering
	$\xymatrixcolsep{1pc}\xymatrixrowsep{1pc}\xymatrix
	{  &&&&&&&f\ar@{.}[dddr]^{\leqslant }\ar@{.}[ddddddll]_{\leqslant }\\
		&& e \ar@{.}[dddr]^{\leqslant }\ar@{.}[dddddddll]_{\leqslant }\ar@{-->}@/^.8pc/[rrrrru]^{x} &&\\
		\\
		&&&&&&&&f_2\ar@{.}[dddr]^{\leqslant }\\
		&&&e_2\ar@{.}[ddddr]^{\leqslant }\\
		\\
		&&&&&f_1\ar@{->}[rrrr]^{\epsilon (f_1,f_1f_2)}&&&&f_1f_2\\\\
		e_1\ar@{->}@/^.7pc/[rrrrruu]^{\epsilon (e_1){\downharpoonleft} x}\ar@{->}[rrrr]_{\epsilon (e_1,e_1e_2)}&&&&e_1e_2 \ar@{->}@/^.7pc/[rrrrruu]_{\epsilon (e_1e_2){\downharpoonleft} x} 	
		}
	$
	\caption{The axiom (IG1) of an inductive groupoid corresponds to the commutativity of the  solid arrows (which represent the isomorphisms). The dashed line represents an arbitrary morphism in the inductive groupoid and the dotted lines correspond to the natural partial order $\leqslant $ between the elements of the biordered set.}\label{indg1}
\end{figure}
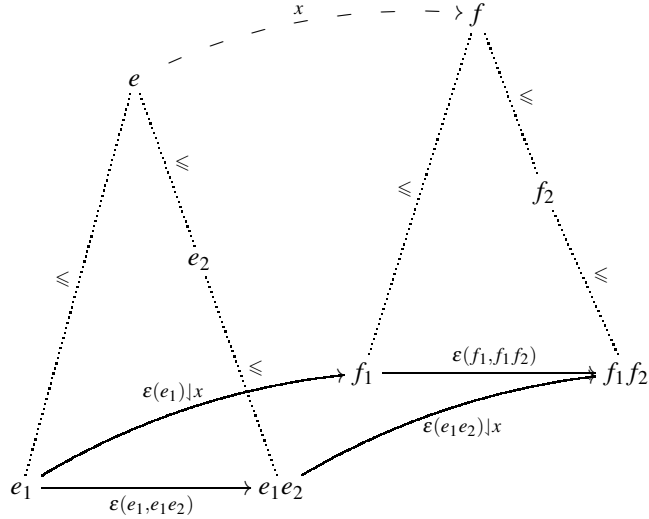

First, let $(f,g)$ be a morphism in $\mathcal{G}_{\Gamma}$ from $(c,d)$ to $(c',d')$ such that $(c_i,d_i) \subseteq (c,d)$  for $i=1,2$. Then the codomain
$\mathbf{r}((c_i,d_i){\downharpoonleft}(f,g))$ is $(c'_i,d'_i)= \im (j(c_i,c)f,j(d_i,d)g)$.

By Remark \ref{rmkog}, the groupoid $\mathcal{G}_{\Gamma|\mathcal{D}}$ is an ordered groupoid; so if $d_1\subseteq d_2$, then clearly $d'_1 \subseteq d'_2$. Also, by the partial binary composition of the biordered set,
$$(c_1,d_1)(c_2,d_2)=(\im\gamma(c_2,d_2)(c_1),d_1),\text{ denoted by }(c_3,d_3)\text{ in the sequel}.$$
As the reader sees, here we relabel $d_1$ as $d_3$. This is done for the sake of notational convenience as we want the indices of $\gamma_i$s and $\delta_i$s to match the indices of $(c_i,d_i)$ in the subsequent proof. Similarly, $$(c'_1,d'_1)(c'_2,d'_2)=(\im\gamma(c'_2,d'_2)(c'_1),d'_1)=:(c'_3,d'_3).$$

Now,
\begin{equation*}
\begin{split}
\epsilon_\Gamma((c_1,d_1),(c_3,d_3))(\epsilon_\Gamma((c_3,d_3)&{\downharpoonleft}(f,g))\\
&= (\gamma_1\gamma_3(c_1),\delta_1\delta_3(d_1))((c_3,d_3){\downharpoonleft}(f,g))\\
&= (\gamma_1\gamma_3(c_1),\delta_1\delta_3(d_1))((j(c_3,c)f)^\circ,(j(d_3,d)g)^\circ)\\
&= ((\gamma_3(c_1))^\circ(j(c_3,c)f)^\circ,(\delta_3(d_1))^\circ(j(d_3,d)g)^\circ)\\
&= (\gamma_3(c_1)(j(c_3,c)f)^\circ ,(1_{d_1})(j(d_3,d)g)^\circ)\\
&= (\gamma_3(c_1)(j(c_3,c)f)^\circ,(j(d_3,d)g)^\circ).\\
\end{split}
\end{equation*}

Also,
\begin{equation*}
\begin{split}
(\epsilon_\Gamma((c_1,d_1)){\downharpoonleft}(f,g))\epsilon_\Gamma(&(c'_1,d'_1),(c'_3,d'_3))\\
&= ((c_1,d_1){\downharpoonleft}(f,g))(\gamma'_1\gamma'_3(c'_1),\delta'_1\delta'_3(d'_1))\\
&= ((j(c_1,c)f)^\circ,(j(d_1,d)g)^\circ)(\gamma'_1\gamma'_3(c'_1),\delta'_1\delta'_3(d'_1))\\
&= ((j(c_1,c)f)^\circ(\gamma'_3(c'_1))^\circ,(j(d_1,d)g)^\circ(\delta'_3(d'_1))^\circ)\\
&= ((j(c_1,c)f)^\circ\gamma'_3(c'_1),(j(d_1,d)g)^\circ(1_{d'_1})^\circ)\\
&= ((j(c_1,c)f)^\circ j(c'_1,c')\gamma'_3(c'),(j(d_1,d)g)^\circ 1_{d'_1})\\
&=(j(c_1,c)f\gamma'_3(c'),(j(d_1,d)g)^\circ)\\
&=(j(c_1,c)\gamma_3(c)(j(c_3,c)f)^\circ,(j(d_1,d)g)^\circ)\quad(\text{By Lemma }\ref{lemog})\\
&=(\gamma_3(c_1)(j(c_3,c)f)^\circ,(j(d_1,d)g)^\circ).\\
\end{split}
\end{equation*}
Since $d_1=d_3$, the right hand sides coincide and thus we have verified (IG1). Similarly, we can verify its dual.

Now we need to verify (IG2), that is, every singular $E$-square is $\epsilon_\Gamma$-commutative. Let $\bigl[ \begin{smallmatrix} (h_1,k_1)&(h_1,k_1)(c,d)\\ (h_2,k_2)&(h_2,k_2)(c,d) \end{smallmatrix} \bigr]$
be a column-singular $E$-square such that $k_1,k_2 \subseteq d$ and $h_1=h_2$. Then,
$$\gamma\theta_1=\theta_1,\:\gamma\theta_2=\theta_2,\:\theta_1\theta_2=\theta_1,\:\theta_2\theta_1=\theta_2;\:\delta\eta_1=\eta_1,\:\delta\eta_2=\eta_2,\:\eta_1\eta_2=\eta_1\text{ and }\eta_2\eta_1=\eta_2.$$
So,
\begin{equation*}
\begin{split}
\epsilon_\Gamma((h_1,k_1),(h_2,k_2))\epsilon_\Gamma((h_2,k_2),(h_2,k_2)&(c,d))\\
&=(\theta_1\theta_2(h_1),\eta_1\eta_2(k_1))(\theta_2\theta_2\gamma(h_2),\eta_2\eta_2\delta(k_2))\\
&=(\theta_1\theta_2\theta_2\theta_2\gamma(h_1),\eta_1\eta_2\eta_2\eta_2\delta(k_1))\\
&=(\theta_1\gamma(h_1),\eta_1\delta(k_1)).\\
\end{split}
\end{equation*}
Also,
\begin{equation*}
\begin{split}
\epsilon_\Gamma((h_1,k_1),(h_1,k_1)(c,d))\epsilon_\Gamma((h_1,k_1)&(c,d),(h_2,k_2)(c,d))\\
&=(\theta_1\theta_1\gamma(h_1),\eta_1\eta_1\delta(k_1))(\theta_1\gamma\theta_2\gamma(h_1),\eta_1\delta\eta_2\delta(k_1))\\
&=(\theta_1\theta_1\gamma\theta_1\gamma\theta_2\gamma(h_1),\eta_1\eta_1\delta\eta_1\delta\eta_2\delta(k_1))\\
&=(\theta_1\gamma(h_1),\eta_1\delta(k_1)).\\
\end{split}
\end{equation*}
So, the column-singular $E$-square $\bigl[ \begin{smallmatrix} (h_1,k_1)&(h_1,k_1)(c,d)\\ (h_2,k_2)&(h_2,k_2)(c,d) \end{smallmatrix} \bigr]$ is $\epsilon_\Gamma$-commutative. Dually, we can show that every row-singular $E$-square is also $\epsilon_\Gamma$-commutative. So (IG2) also holds.

Hence $(\mathcal{G}_{\Gamma},\epsilon_\Gamma)$ is an inductive groupoid.
\end{proof}

\subsection{The functor $\mathbb{I}\colon \mathbf{CC}\to \mathbf{IG}$}

We have seen that given a cross-connection $(\mathcal{D},\mathcal{C};{\Gamma})$, it has a corresponding inductive groupoid $(\mathcal{G}_\Gamma,\epsilon_\Gamma)$. Now, we extend this correspondence to morphisms and also show that it is, in fact, functorial.

Given two cross-connections $(\mathcal{D},\mathcal{C};{\Gamma})$ and $(\mathcal{D}',\mathcal{C}';{\Gamma}')$, a \emph{morphism of cross-connections} $m:\Gamma\to \Gamma'$ is a pair $m=(F_m,G_m)$ of inclusion preserving functors $F_m:\mathcal{C}\to \mathcal{C}'$ and $G_m:\mathcal{D}\to \mathcal{D}'$ satisfying the  axioms of Definition \ref{morcxn}.

Clearly $m$ is a functor from $\mathcal{C}\times\mathcal{D}$ to $\mathcal{C}'\times\mathcal{D}'$ and $\mathcal{G}_\Gamma\subseteq \mathcal{C}\times\mathcal{D}$. So $m_{|\mathcal{G}_\Gamma}\colon \mathcal{G}_\Gamma\to\mathcal{G}_{\Gamma'}$ is also a functor.

\begin{pro}
$m_{|\mathcal{G}_\Gamma}$ is an inductive functor from $(\mathcal{G}_\Gamma,\epsilon_{\Gamma})$ to $(\mathcal{G}_{\Gamma'},\epsilon_{\Gamma'})$.
\end{pro}
\begin{proof}
First, by \cite[Lemma V.4]{cross}, $m_{|\mathcal{G}_\Gamma}$ is a regular bimorphism from $E_\Gamma$ to $E_{\Gamma'}$.

Recall from \cite[Section V.2]{cross} that any inclusion preserving functor between two normal categories preserves normal factorisation and in particular it preserves epimorphic components. Now, since $F_m$ and $G_m$ are inclusion preserving,
\begin{equation*}
\begin{split}
m_{|\mathcal{G}_\Gamma}((c_1,d_1){\downharpoonleft}(f,g))&= m_{|\mathcal{G}_\Gamma}((j(c_1,c)f)^\circ,(j(d_1,d)g)^\circ)\\
&= (F_m((j(c_1,c)f)^\circ),G_m((j(d_1,d)g)^\circ))\\
&= ((j(F_m(c_1),F_m(c))F_m(f))^\circ,(j(G_m(d_1),G_m(d))G_m(g))^\circ)\\
&= (F_m(c_1),G_m(d_1)){\downharpoonleft}(F_m(f),G_m(g))\\
&= m_{|\mathcal{G}_\Gamma}((c_1,d_1)){\downharpoonleft}m_{|\mathcal{G}_\Gamma}((f,g)).
\end{split}
\end{equation*}
So, $m_{|\mathcal{G}_\Gamma}$ is order preserving.

Given $((c_1,d_1),(c_2,d_2),\dots,(c_n,d_n)) \in  \mathcal{G}(E) $, as earlier, we denote the idempotent cones $\gamma(c_i,d_i)$ and $\delta(c_i,d_i)$ by $\gamma_i$ and $\delta_i$, respectively, and the cones $\gamma'(F_m(c_i),G_m(d_i))$ and $\delta'(F_m(c_i),G_m(d_i)$ by $\gamma'_i$ and $\delta'_i$, respectively. Then,
\begin{equation*}
\begin{split}
\mathcal{G}(vm_{|\mathcal{G}_\Gamma})\epsilon_{\Gamma'}&((c_1,d_1),(c_2,d_2),\dots,(c_n,d_n))\\
&= \epsilon_{\Gamma'}((F_m(c_1),G_m(d_1)),(F_m(c_2),G_m(d_2)),\dots,(F_m(c_n),G_m(d_n)))\\
&=(\gamma'_1\gamma'_2\dots\gamma'_n(F_m(c_1)),\delta'_1\delta'_2\dots\delta'_n(G_m(d_1)))\\
&=(F_m(\gamma_1\gamma_2\dots\gamma_n(c_1)),G_m(\delta_1\delta_2\dots\delta_n(d_1))) \quad\text{ (Using the axiom (M1)})\\
&=m_{|\mathcal{G}_\Gamma}(\gamma_1\gamma_2\dots\gamma_n(c_1),\delta_1\delta_2\dots\delta_n(d_1)) \\
&=\epsilon_{\Gamma}m_{|\mathcal{G}_\Gamma}((c_1,d_1),(c_2,d_2),\dots,(c_n,d_n)).
\end{split}
\end{equation*}
Hence the following diagram commutes:
\begin{equation*}\label{indfun}
\xymatrixcolsep{6pc}\xymatrixrowsep{4pc}\xymatrix
{
 \mathcal{G}(E) \ar[r]^{\mathcal{G}(vm_{|\mathcal{G}_\Gamma})} \ar[d]_{\epsilon_\Gamma}
 & \mathcal{G}(E') \ar[d]^{\epsilon_{\Gamma'}} \\
 \mathcal{G}_\Gamma \ar[r]^{m_{|\mathcal{G}_\Gamma}} & \mathcal{G}_{\Gamma'}
}
\end{equation*}
Thus $m_{|\mathcal{G}_\Gamma}$ is an inductive functor.
\end{proof}
Verification of the next result is routine.
\begin{thm}
The assignments
$$(\mathcal{D},\mathcal{C};{\Gamma})\mapsto (\mathcal{G}_\Gamma,\epsilon_{\Gamma}) \quad\quad m\mapsto m_{|\mathcal{G}_\Gamma}$$
is a functor $\mathbb{I}\colon \mathbf{CC}\to \mathbf{IG}$.
\end{thm}


\section{Cross-connection from an inductive groupoid}\label{cxnind}

Having constructed the inductive groupoid of a cross-connection, now we attempt the converse. Given the inductive groupoid $(\mathcal{G},\epsilon)$ with biordered set $E$ having preorders $\lel $ and $\ler $, we construct a cross-connection $(\mathcal{R}_G,\mathcal{L}_G;{\Gamma_G})$. Unlike the previous case where it sufficed to identify the category sitting inside the cross-connection, here we have to `split' the inductive groupoid and then `extend' each part to the required normal category.

\subsection{The normal category $\mathcal{L}_G$.}
We begin with building the `left' normal category $\mathcal{L}_G$ associated to the inductive groupoid $\mathcal{G}$. The crucial property of a normal category that will guide us in this construction is that every morphism has a normal factorisation into a retraction, an isomorphism and an inclusion. So, we shall build three separate categories: one category $\mathcal{P}_L$ `responsible' for inclusions, the other one $\mathcal{G}_L$ `responsible' for isomorphisms, and the last one $\mathcal{Q}_L$ `responsible' for retractions. Then we combine these categories to build our required category by extending the composition of the isomorphisms (which is inherited from  the given inductive groupoid).

The major obstacle in this procedure arises from the fact that normal factorisation of a morphism is not unique. But fortunately, the epimorphic component (retraction + isomorphism) of  a morphism is indeed unique. Exploiting this fact, we first build an intermediate category $\mathcal{E}_L$ from the categories $\mathcal{Q}_L$ and $\mathcal{G}_L$, and then finally realise $\mathcal{L}_G$ as a suitable product of $\mathcal{E}_L$ and $\mathcal{P}_L$.

Given the inductive groupoid $\mathcal{G}$ with regular biordered set $E$, let the object set $v\mathcal{L}_G$ be $E/\mathscr{L}$, where $\mathscr{L}\:=\:\lel \:\cap\:(\lel )^{-1}$. This gives a partially ordered set $E/\mathscr{L}$ with respect to the order $\leq_L\: := \: \lel /\mathscr{L}$. In fact, $E/\mathscr{L}$ forms a regular partially ordered set, in the sense of Grillet \cite{grilcross}. The proof of this statement may be found in \cite{bicxn}. Given $e\in E$, in the sequel, $\overleftarrow{e}$ shall denote the canonical image of $e$ in $E/\mathscr{L}$. This set $ E/\mathscr{L}$ shall act as the object set of all our three categories: $\mathcal{P}_L$, $\mathcal{G}_L$ and $\mathcal{Q}_L$. That is, $ v\mathcal{P}_L= v\mathcal{G}_L= v\mathcal{Q}_L:= E/\mathscr{L}$. 

Let us begin by completing our first category $\mathcal{P}_L$. Recall that a partially ordered set corresponds naturally to a strict preorder category. In fact, our first required category $\mathcal{P}_L$ is the preorder category associated with the partially ordered set $E/\mathscr{L}$. Hence, in $\mathcal{P}_L$, we introduce the formal symbol $j= j_\mathcal{L}(e,f)$ for a morphism from $\overleftarrow{e}$ to $\overleftarrow{f}$ whenever $\overleftarrow{e}\leq_L \overleftarrow{f}$. So, given two morphisms $j_\mathcal{L}(e,f)$ and $j_\mathcal{L}(g,h)$ in $\mathcal{P}_L$, they are equal if and only if $e\mathrel{\mathscr{L}}g$ and $f\mathrel{\mathscr{L}}h$. Given $j_\mathcal{L}(e,f)$ and $j_\mathcal{L}(f,g)$, we compose them using the composition induced by the partial binary composition of the biordered set $E$ as follows:
$$j_\mathcal{L}(e,f)\: j_\mathcal{L}(f,g) := j_\mathcal{L}(e,g).$$
Observe that since $e\lel  f$, we have $e\: f = e$ in $E$. Now, the verification of the following proposition is routine.

\begin{pro}\label{lempl}
$\mathcal{P}_L$ is a strict preorder category with the object set $v\mathcal{P}_L := E/\mathscr{L}$ and the morphisms in $\mathcal{P}_L$ as defined above.
\end{pro}

Now, we move onto our second required category, namely $\mathcal{G}_L$ which shall be responsible for the isomorphisms in $\mathcal{L}_G$. Recall from Definition \ref{dfnig} that an inductive groupoid $\mathcal{G}$ comes equipped with an evaluation functor $\epsilon$, which helps you to `evaluate' $E$-chains of the ordered groupoid $\mathcal{G}(E)$ in the groupoid $\mathcal{G}$. Also, recall that the object set $v\mathcal{G}_L:=E/\mathscr{L}$. Then, to define morphisms in the category $\mathcal{G}_L$, given any two morphisms $\alpha,\beta$ in the inductive groupoid $\mathcal{G}$, we first define a relation $\sim_L$ as follows:
\begin{equation}\label{eqgl}
\alpha \sim_L \beta \:\iff\: \mathbf{d}(\alpha) \mathrel{\mathscr{L}} \mathbf{d}(\beta),\: \mathbf{r}(\alpha) \mathrel{\mathscr{L}} \mathbf{r}(\beta) \:\text{ and }\: \alpha\:\epsilon(\mathbf{r}(\alpha),\mathbf{r}(\beta)) = \epsilon(\mathbf{d}(\alpha),\mathbf{d}(\beta))\:\beta.
\end{equation}

\begin{lem}
$\sim_L$ is an equivalence relation.
\end{lem}
\begin{proof}
Clearly $\sim_L$ is reflexive.

Now if $\alpha \sim_L \beta$, then
\begin{equation}\label{eqnsim}
\alpha\:\epsilon(\mathbf{r}(\alpha),\mathbf{r}(\beta)) = \epsilon(\mathbf{d}(\alpha),\mathbf{d}(\beta))\:\beta.
\end{equation}
Observe that $$\epsilon(\mathbf{r}(\alpha),\mathbf{r}(\beta))\epsilon(\mathbf{r}(\beta),\mathbf{r}(\alpha))=1_{\mathbf{r}(\alpha)}\text{ and }
\epsilon(\mathbf{d}(\beta),\mathbf{d}(\alpha))\epsilon(\mathbf{d}(\alpha),\mathbf{d}(\beta))=1_{\mathbf{d}(\beta)}.$$
Hence multiplying the equation (\ref{eqnsim}) by $\epsilon(\mathbf{r}(\beta),\mathbf{r}(\alpha))$ on the right and by $\epsilon(\mathbf{d}(\beta),\mathbf{d}(\alpha))$ on the left, we have $$ \epsilon(\mathbf{d}(\beta),\mathbf{d}(\alpha))\:\alpha = \beta\:\epsilon(\mathbf{r}(\beta),\mathbf{r}(\alpha)).$$ That is, $ \beta \sim_L \alpha$ and so the relation $\sim_L$ is symmetric.

If $\alpha \sim_L \beta$ and $\beta \sim_L \gamma$, then
\begin{equation}\label{eqnsim1}
\alpha\:\epsilon(\mathbf{r}(\alpha),\mathbf{r}(\beta)) = \epsilon(\mathbf{d}(\alpha),\mathbf{d}(\beta))\:\beta\text{ and }\beta\:\epsilon(\mathbf{r}(\beta),\mathbf{r}(\gamma)) = \epsilon(\mathbf{d}(\beta),\mathbf{d}(\gamma))\:\gamma.
\end{equation}
Observe that since $\mathbf{r}(\alpha)\mathrel{\mathscr{L}}\mathbf{r}(\beta)\mathrel{\mathscr{L}}\mathbf{r}(\gamma)$ and $\mathbf{d}(\alpha)\mathrel{\mathscr{L}}\mathbf{d}(\beta)\mathrel{\mathscr{L}}\mathbf{d}(\gamma)$, we have $$\epsilon(\mathbf{r}(\alpha),\mathbf{r}(\beta))\epsilon(\mathbf{r}(\beta),\mathbf{r}(\gamma))=\epsilon(\mathbf{r}(\alpha),\mathbf{r}(\gamma))\text{ and } \epsilon(\mathbf{d}(\alpha),\mathbf{d}(\beta))\epsilon(\mathbf{d}(\beta),\mathbf{d}(\gamma))=\epsilon(\mathbf{d}(\alpha),\mathbf{d}(\gamma)).$$
So,
\begin{align*}
\alpha\:\epsilon(\mathbf{r}(\alpha),\mathbf{r}(\gamma)) =&\alpha\:\epsilon(\mathbf{r}(\alpha),\mathbf{r}(\beta))\:\epsilon(\mathbf{r}(\beta),\mathbf{r}(\gamma))\\ =&\epsilon(\mathbf{d}(\alpha),\mathbf{d}(\beta))\:\beta\:\epsilon(\mathbf{r}(\beta),\mathbf{r}(\gamma))\\
=&\epsilon(\mathbf{d}(\alpha),\mathbf{d}(\beta))\epsilon(\mathbf{d}(\beta),\mathbf{d}(\gamma))\:\gamma\\
=&\epsilon(\mathbf{d}(\alpha),\mathbf{d}(\gamma))\:\gamma.
\end{align*}
So $ \alpha \sim_L \gamma$ and the relation $\sim_L$ is transitive. Hence the lemma.
\end{proof}
\begin{rmk}
Observe that the relation $\sim_L $ reduces to the $\mathrel{\mathscr{L}}$-relation on the identities of $\mathcal{G}$. Informally speaking, the relation (\ref{eqgl}) may be seen as a `left-sided' version of the crucial $p$-relation of \cite[Section 4]{mem}, which is defined later in this article: see equation (\ref{eqnp}). 
\end{rmk}
We shall require the following observation in the sequel.
\begin{pro}\label{progl}
Let $(\mathcal{D},\mathcal{C};\Gamma)$ be a cross-connection with its inductive groupoid $\mathcal{G}_\Gamma$. For $(f_1,g_1),(f_2,g_2)\in \mathcal{G}_\Gamma$,
$$(f_1,g_1)\sim_L(f_2,g_2) \iff f_1=f_2.$$
\end{pro}
\begin{proof}
Suppose that $(f_i,g_i)\colon(c_i,d_i)\to (c_i',d_i')$ for $i=1,2$. Let $(f_1,g_1)\sim_L(f_2,g_2)$ so that $\mathbf{d}(f_1,g_1)=(c_1,d_1) \mathrel{\mathscr{L}} (c_2,d_2)=\mathbf{d}(f_2,g_2)$. Then, by (\ref{eqb2}), we have $c_1=c_2$. Similarly we have $c_1'=c_2'$. Further the last condition in (\ref{eqgl}) implies that:
\begin{align*}
 &\qquad (f_1,g_1)\:\epsilon((c_1',d_1'),(c_2',d_2')) = \epsilon((c_1,d_1),(c_2,d_2))\:(f_2,g_2)\\
\iff& \qquad(f_1,g_1)\:(\gamma_1'\gamma_2'(c_1'),\delta_1'\delta_2'(d_1')) = (\gamma_1\gamma_2(c_1),\delta_1\delta_2(d_1))\:(f_2,g_2)\\
\iff& \qquad(f_1\gamma_2'(c_1'),g_1\delta_2'(d_1')) = (\gamma_2(c_1)f_2,\delta_2(d_1)g_2).
\end{align*}
Since $c_1=c_2$ and $c_1'=c_2'$, we have $\gamma_2'(c_1')=1_{c_1'}$ and $\gamma_2(c_1)=1_{c_2}$. Thus, equating the first components gives us $f_1=f_2$.

Conversely, suppose that $f_1=f_2$. Since $c_1=c_2$ and $c_1'=c_2'$, by  (\ref{eqb2}), we have that $\mathbf{d}(f_1,g_1) \mathrel{\mathscr{L}} \mathbf{d}(f_2,g_2)$ and $\mathbf{r}(f_1,g_1) \mathrel{\mathscr{L}} \mathbf{r}(f_2,g_2)$. Also, since $\gamma_2'(c_1')=1_{c_1'}$ and $\gamma_2(c_1)=1_{c_2}$, we have that $f_1\gamma_2'(c_1') = \gamma_2(c_1)f_2$. Now observe that in the inductive groupoid $\mathcal{G}_\Gamma$, the morphisms $g_1\delta_2'(d_1')$ and $\delta_2(d_1)g_2$ are transposes of $(f_1\gamma_2'(c_1'))^{-1}$ and $ (\gamma_2(c_1)f_2)^{-1}$, respectively. But since we know that $(f_1\gamma_2'(c_1'))^{-1}= (\gamma_2(c_1)f_2)^{-1}$, and also that both the morphisms $g_1\delta_2'(d_1')$ and $\delta_2(d_1)g_2$ are from $d_1$ to $d_2'$, by the uniqueness of transposes (for a given domain and codomain), we have that $g_1\delta_2'(d_1')=\delta_2(d_1)g_2$. Hence we have $(f_1,g_1)\sim_L(f_2,g_2)$.
\end{proof}

Given a morphism $\alpha$ in $\mathcal{G}$ from $e$ to $f$, we shall denote the $\sim_L$-class of $\mathcal{G}$ containing the morphism $\alpha$ by $\overleftarrow{\alpha}=\overleftarrow{\alpha}(e,f)$. We shall define $\overleftarrow{\alpha}$ as a morphism in $\mathcal{G}_L$ from $\overleftarrow{e}$ to $\overleftarrow{f}$. Further, for $\overleftarrow{\alpha},\overleftarrow{\beta}\in \mathcal{G}_L$ such that $\mathbf{r}(\alpha)\mathrel{\mathscr{L}}\mathbf{d}(\beta)$, we define a composition in $\mathcal{G}_L$ as
$$\overleftarrow{\alpha}\overleftarrow{\beta}:= \overleftarrow{\alpha\:\epsilon(\mathbf{r}(\alpha),\mathbf{d}(\beta))\:\beta}.$$

\begin{pro}
$\mathcal{G}_L$ is a groupoid.
\end{pro}
\begin{proof}
First, we verify that the composition is well-defined. Suppose $\alpha_1\sim_L\alpha_2$ and $\beta_1\sim_L\beta_2$ are such that $\overleftarrow{\alpha_1}\overleftarrow{\beta_1}$ and $\overleftarrow{\alpha_2}\overleftarrow{\beta_2}$ exist.
Let $\gamma_1:=\alpha_1\:\epsilon(\mathbf{r}(\alpha_1),\mathbf{d}(\beta_1))\:\beta_1$ and $\gamma_2:=\alpha_2\:\epsilon(\mathbf{r}(\alpha_2),\mathbf{d}(\beta_2))\:\beta_2$; so we need to show that  $\gamma_1\sim_L\gamma_2$.

Since  $\mathbf{d}(\alpha_1)\mathrel{\mathscr{L}}\mathbf{d}(\alpha_2)$,  we have that $\mathbf{d}(\gamma_1)\mathrel{\mathscr{L}}\mathbf{d}(\gamma_2)$ and since $\mathbf{r}(\beta_1)\mathrel{\mathscr{L}}\mathbf{r}(\beta_2)$, we have $\mathbf{r}(\gamma_1)\mathrel{\mathscr{L}}\mathbf{r}(\gamma_2)$. Also since  $\mathbf{r}(\alpha_1)\mathrel{\mathscr{L}}\mathbf{r}(\alpha_2)\mathrel{\mathscr{L}}\mathbf{d}(\beta_1)\mathrel{\mathscr{L}}\mathbf{d}(\beta_2)$, we get
\begin{align}
\label{eqnproid}
\epsilon(\mathbf{r}(\alpha_1),\mathbf{d}(\beta_1))=\epsilon(\mathbf{r}(\alpha_1),\mathbf{r}(\alpha_2))\:&\epsilon(\mathbf{r}(\alpha_2),\mathbf{d}(\beta_1))\quad \text{ and }\\
\epsilon(\mathbf{r}(\alpha_2),&\mathbf{d}(\beta_2))=\epsilon(\mathbf{r}(\alpha_2),\mathbf{d}(\beta_1))\:\epsilon(\mathbf{d}(\beta_1),\mathbf{d}(\beta_2)).\notag	
\end{align}
So,
\begin{align*}
\gamma_1&\:\epsilon(\mathbf{r}(\gamma_1),\mathbf{r}(\gamma_2)) \\
&=\alpha_1\:\epsilon(\mathbf{r}(\alpha_1),\mathbf{d}(\beta_1))\:\beta_1\:\epsilon(\mathbf{r}(\beta_1),\mathbf{r}(\beta_2))\quad (\text{by definition of }\gamma_1)\\
&=\alpha_1\:\epsilon(\mathbf{r}(\alpha_1),\mathbf{r}(\alpha_2))\:\epsilon(\mathbf{r}(\alpha_2),\mathbf{d}(\beta_1))\:\beta_1\:\epsilon(\mathbf{r}(\beta_1),\mathbf{r}(\beta_2)) \quad (\text{using } (\ref{eqnproid}))\\
&=\epsilon(\mathbf{d}(\alpha_1),\mathbf{d}(\alpha_2))\:\alpha_2\:\epsilon(\mathbf{r}(\alpha_2),\mathbf{d}(\beta_1))\:\epsilon(\mathbf{d}(\beta_1),\mathbf{d}(\beta_2))\:\beta_2 \quad (\text{since } \alpha_1\sim_L\alpha_2\text{ and }\beta_1\sim_L\beta_2)\\
&=\epsilon(\mathbf{d}(\alpha_1),\mathbf{d}(\alpha_2))\:\alpha_2\:\epsilon(\mathbf{r}(\alpha_2),\mathbf{d}(\beta_2))\:\beta_2\quad (\text{using } (\ref{eqnproid}))\\
&=\epsilon(\mathbf{d}(\gamma_1),\mathbf{d}(\gamma_2))\:\gamma_2 \quad (\text{by definition of }\gamma_2).
\end{align*}
Hence $\gamma_1\sim_L\gamma_2$ and so the composition is well-defined. The associativity of the composition follows from the associativity of the composition in $\mathcal{G}$.

Now given $\overleftarrow{\alpha} \in \mathcal{G}_L$ from $\overleftarrow{e}$ to $\overleftarrow{f}$, since
$$\overleftarrow{1_e}\overleftarrow{\alpha}=\overleftarrow{1_e\:\epsilon(e,e)\:\alpha}=\overleftarrow{\alpha}\text{ and }\overleftarrow{\alpha}\overleftarrow{1_f}=\overleftarrow{\alpha\:\epsilon(f,f)\:1_f}=\overleftarrow{\alpha},$$
$\overleftarrow{1_e}$ is the identity morphism at the apex $\overleftarrow{e}$.
Also,
$$\overleftarrow{\alpha^{-1}}\overleftarrow{\alpha}=\overleftarrow{\alpha^{-1}\:\epsilon(e,e)\:\alpha}=\overleftarrow{1_f}\text{ and }\overleftarrow{\alpha}\overleftarrow{\alpha^{-1}}=\overleftarrow{\alpha\:\epsilon(f,f)\:\alpha^{-1}}=\overleftarrow{1_e}.$$
So $(\overleftarrow{\alpha})^{-1}= \overleftarrow{\alpha^{-1}}$ and hence $\mathcal{G}_L$ is a groupoid.
\end{proof}
\begin{rmk}
It can be shown that, in fact, $\mathcal{G}_L$ is an ordered groupoid with respect to the order induced from the inductive groupoid.
\end{rmk}

Having built our required two categories, we move onto our third category $\mathcal{Q}_L$.  Recall again that $v\mathcal{Q}_L:=E/\mathscr{L}$. Now, if $f\lel  e$, for each $u$ in the biordered set $E$ such that $u\leqslant e$ and $u\mathrel{\mathscr{L}}f$, we define a morphism $q=q_\mathcal{L}(e,u)$ in $\mathcal{Q}_L$ from $\overleftarrow{e}$ to $\overleftarrow{u}=\overleftarrow{f}$. Then two such morphisms $q_\mathcal{L}(e,u)$ and $q_\mathcal{L}(g,v)$ are equal if and only if $e\mathrel{\mathscr{L}}g$ and $v=gu$ in $E$.
In that case, since $u\leqslant  e\mathrel{\mathscr{L}}g$, observe that $gu$ is a basic product. In particular, if $e\mathrel{\mathscr{L}}g$, then $g=ge$ and so $q_\mathcal{L}(e,e)=q_\mathcal{L}(g,g)$.

Further, if we have two morphisms $q_\mathcal{L}(e,u)$ and $q_\mathcal{L}(f,v)$ in $\mathcal{Q}_L$ such that $u\mathrel{\mathscr{L}}f$, then we define a composition on $\mathcal{Q}_L$ as
$$q_\mathcal{L}(e,u)\:q_\mathcal{L}(f,v) := q_\mathcal{L}(e,uv).$$
Since $v\lel u$, so $uv$ is a basic product in $E$ and by axiom (B2), $uv\mathrel{\mathscr{L}}v$. Since $v\lel  u \lel  u$, using axiom (B4), we have $uv=u(uv)$. That implies $uv \ler  u$. Combining this with the fact that $uv \lel  u$, we have $uv \leqslant  u$. Further, since $u \leqslant  e$, by transitivity, we have $uv \leqslant  e$. Hence $q_\mathcal{L}(e,uv)$ is a morphism in $\mathcal{Q}_L$ and the above composition is well-defined.

\begin{pro}
$\mathcal{Q}_L$ is a category.
\end{pro}
\begin{proof}
We need to verify associativity and identity.
If $q_\mathcal{L}(e,u), q_\mathcal{L}(f,v)$, and $q_\mathcal{L}(g,w)$ are composable morphisms in $\mathcal{Q}_L$, then observe that $w\lel v\lel u$. Then,
\begin{align*}
(q_\mathcal{L}(e,u)\:q_\mathcal{L}(f,v))\:q_\mathcal{L}(g,w)
&= q_\mathcal{L}(e,uv)\:q_\mathcal{L}(g,w)\\
&= q_\mathcal{L}(e,(uv)w)\\
&= q_\mathcal{L}(e,u(vw)) \quad\text{(Using the dual of \cite[Proposition 2.3]{mem})}\\\
&=q_\mathcal{L}(e,u)\:q_\mathcal{L}(f,vw)\\
&=q_\mathcal{L}(e,u)\:(q_\mathcal{L}(f,v)\:q_\mathcal{L}(g,w)).
\end{align*}
So it is associative. Also, since
$$q_\mathcal{L}(e,u)q_\mathcal{L}(u,u)=q_\mathcal{L}(e,u)\text{ and }q_\mathcal{L}(e,e)q_\mathcal{L}(e,u)=q_\mathcal{L}(e,u),$$
the identity morphism at $\overleftarrow{e}$ in $\mathcal{Q}_L$ is $q_\mathcal{L}(e,e)$. Hence $\mathcal{Q}_L$ is a category.
\end{proof}

We have successfully built all our three ingredient categories, namely  $\mathcal{P}_L$, $\mathcal{G}_L$, and $\mathcal{Q}_L$. Now we need to synthesise their composition to construct our required category  $\mathcal{L}_G$ from the inductive groupoid composition. As mentioned earlier, in this process, we shall rely heavily on the uniqueness of epimorphic components of morphisms in a normal category. To that end, we first build the category $\mathcal{E}_L$ (that would be responsible for epimorphisms) using the categories $\mathcal{Q}_L$ and $\mathcal{G}_L$. For this, we need the following concept of a quiver.

\begin{dfn}
A \emph{quiver} $\mathcal{K}$ consists of a set of objects (denoted as $v\mathcal{K}$) together with a set of morphisms (denoted by $\mathcal{K}$ itself) and two functions $\mathbf{d},\mathbf{r}\colon\mathcal{K}\rightrightarrows v\mathcal{K}$ giving the \emph{domain} and \emph{codomain} of each morphism. 
\end{dfn}

Now given categories $\mathcal{Q}_L$ and $\mathcal{G}_L$, consider an intermediary quiver $\mathcal{E}$ with object set $v\mathcal{E}:= E/\mathscr{L}$ and morphisms as follows:
$$\mathcal{E}:=\{(q,\overleftarrow{\alpha})\in \mathcal{Q}_L \times \mathcal{G}_L :  \mathbf{r}(q)=\mathbf{d}(\overleftarrow{\alpha}) \}.$$

Then, consider the following relation $\sim_\mathcal{E}$ on the morphisms of the quiver $\mathcal{E}$: for any $\xi_1:=(q_1(e,u),\overleftarrow{\alpha})$ and $\xi_2:=(q_2(f,v),\overleftarrow{\beta})$,
$$\xi_1\sim_\mathcal{E}\xi_2\iff e\mathrel{\mathscr{L}}f,\:u\mathrel{\mathscr{R}}v\ \text{ and }\overleftarrow{\alpha}=\overleftarrow{\epsilon(u,v)}\:\overleftarrow{\beta}.$$
The following lemma can be easily verified.
\begin{lem}
$\sim_\mathcal{E}$ is an equivalence relation.
\end{lem}

Now, we define the category $\mathcal{E}_L$ whose object set is $v\mathcal{E}_L := E/\mathscr{L}$ and whose morphisms are $\sim_\mathcal{E}$-classes of the quiver $\mathcal{E}$. We take the liberty of referring to $\mathcal{E}_L$ as a category even though we have not yet defined any composition in $\mathcal{E}_L$; if fact, heading to such a definition, we first need to specify a `good' representative in each $\sim_\mathcal{E}$-class. 

It is easy to see that for an arbitrary morphism $(q,\overleftarrow{\alpha}) = (q_\mathcal{L}(e,u),\overleftarrow{\alpha(f,g)}) \in\mathcal{E}$, the morphism $\theta:=\epsilon(u,f)\:\alpha$ from the inductive groupoid  $\mathcal{G}$ satisfies $\mathbf{d}(\theta)=u$ and $\overleftarrow{\alpha}=\overleftarrow{\theta}$. From the latter property, we conclude that $(q,\overleftarrow{\theta})$ lies in the  $\sim_\mathcal{E}$-class of $(q,\overleftarrow{\alpha})$. This allows us to represent each $\sim_\mathcal{E}$-class by a morphism of the form $(q_\mathcal{L}(e,u),\overleftarrow{\alpha(u,f)})$ which we refer to as a \emph{right epi} in the category $\mathcal{E}_L$. In the sequel, unless otherwise stated, a morphism in the category $\mathcal{E}_L$ shall always be represented by its right epi. For brevity, whenever there is no scope of confusion, we shall denote the above right epi by just $[e,\alpha\rangle$. 

Observe that a right epi $[e,\alpha\rangle$ represents the following $\sim_\mathcal{E}$-class of morphisms in the quiver $\mathcal{E}$:
$$ [  e , \alpha \rangle_{\sim_\mathcal{E}} = \{\:( q_\mathcal{L}(f,v),\overleftarrow{\theta})\in \mathcal{E}:\:f\mathrel{\mathscr{L}}e,\:v\mathrel{\mathscr{R}}\mathbf{d}(\alpha)\text{ and }\overleftarrow{\theta}=\overleftarrow{\epsilon(v,\mathbf{d}(\alpha))\:\alpha} \:\}. $$
We allow ourselves the notation $[e,\alpha\rangle\in\mathcal{E}_L$, meaning, of course, the $\sim_\mathcal{E}$-class just shown. 

So, given two right epis $ [  e, \alpha \rangle $ and $ [  f, \beta \rangle $  in the category $\mathcal{E}_L$, they are $\sim_\mathcal{E}$ related if and only if $e\mathrel{\mathscr{L}}f$, $\mathbf{d}(\alpha)\mathrel{\mathscr{R}}\mathbf{d}(\beta)$ and $\overleftarrow{\alpha}=\overleftarrow{\epsilon(\mathbf{d}(\alpha),\mathbf{d}(\beta))\:\beta}$. So, if $ [  e, \alpha \rangle , [  f,\beta \rangle$ are such that $e\mathrel{\mathscr{L}}\mathbf{d}(\alpha)$ and $f\mathrel{\mathscr{L}}\mathbf{d}(\beta)$, then $ [  e, \alpha \rangle =  [  f,\beta \rangle $ if and only if $\overleftarrow{\alpha}=\overleftarrow{\beta}$. Also $ [  e,1_u \rangle = [  f,1_v \rangle $ if and only if $e\mathrel{\mathscr{L}}f$ and $v=fu$. In particular, if $e\mathrel{\mathscr{L}}f$, we have $ [  e , 1_e \rangle = [  f , 1_f \rangle $.

As in \cite[Section 4]{mem}, consider the following relation $p$ on $\mathcal{G}$: for $\alpha,\beta\in\mathcal{G}$,
\begin{equation} \label{eqnp}
    \alpha\:p\:\beta \iff \mathbf{d}(\alpha)\mathrel{\mathscr{R}}\mathbf{d}(\beta),\:\mathbf{r}(\alpha)\mathrel{\mathscr{L}}\mathbf{r}(\beta)\:\text{ and } \alpha\:\epsilon(\mathbf{r}(\alpha),\mathbf{r}(\beta)) = \epsilon(\mathbf{d}(\alpha),\mathbf{d}(\beta))\:\beta.
\end{equation}
In contrast to \eqref{eqgl}, the definition of $p$ is `bilateral'; in \cite[Section 4]{mem}, it is verified that $p$ is an equivalence relation. As in \cite{mem}, we shall denote the $p$-class of $\mathcal{G}$ containing $\alpha$ by $\overline{\alpha}$. 

The following lemma which is crucial in the further considerations gives the relationship between the morphisms of the required normal category $\mathcal{L}_G$ and the given inductive groupoid $\mathcal{G}$.
\begin{lem}\label{lemp}
Let $ [  e,{\alpha} \rangle $, $ [  f, {\beta} \rangle $ be right epis in the category $\mathcal{E}_L$. Then
$ [  e, {\alpha} \rangle  =  [  f, {\beta} \rangle $ if and only if $e\mathrel{\mathscr{L}}f$ and $\overline{\alpha}=\overline{\beta}$.
\end{lem}
\begin{proof}
Let $ [  e, {\alpha} \rangle  =  [  f, {\beta} \rangle $, then $e\mathrel{\mathscr{L}}f$, $\mathbf{d}(\alpha)\mathrel{\mathscr{R}}\mathbf{d}(\beta) $ and $\overleftarrow{\alpha}=\overleftarrow{\epsilon(\mathbf{d}(\alpha),\mathbf{d}(\beta))\:\beta}$. That is $\alpha\:\epsilon(\mathbf{r}(\alpha),\mathbf{r}(\beta)) = \epsilon(\mathbf{d}(\alpha),\mathbf{d}(\beta))\:\beta$ and $\mathbf{r}(\alpha)\mathrel{\mathscr{L}}\mathbf{r}(\beta) $. So $\overline{\alpha}=\overline{\beta}$.

Conversely, if $e\mathrel{\mathscr{L}}f$ and $\overline{\alpha}=\overline{\beta}$, then $\mathbf{d}(\alpha)\mathrel{\mathscr{R}}\mathbf{d}(\beta)$, $\mathbf{r}(\alpha)\mathrel{\mathscr{L}}\mathbf{r}(\beta)$ and $\alpha\:\epsilon(\mathbf{r}(\alpha),\mathbf{r}(\beta)) = \epsilon(\mathbf{d}(\alpha),\mathbf{d}(\beta))\:\beta$.  Now since $\overleftarrow{\epsilon(\mathbf{r}(\alpha),\mathbf{r}(\beta))}=1_{\mathbf{r}(\alpha)}$, we have $\overleftarrow{\alpha}=\overleftarrow{\epsilon(\mathbf{d}(\alpha),\mathbf{d}(\beta))\:\beta}$. So, $ [  e, {\alpha} \rangle  =  [  f, {\beta} \rangle $. Hence the lemma.
\end{proof}

Now we proceed to define a partial composition in the category $\mathcal{E}_L$ as follows. Let ${\varepsilon}_1=  [  e_1, {\alpha} \rangle $ and ${\varepsilon}_2=  [  e_2, {\beta} \rangle $ be right epis in the category $\mathcal{E}_L$ such that $\alpha\in \mathcal{G}(u,f_1)$ and $\beta\in \mathcal{G}(v,f_2)$. If $\overleftarrow{f_1}\leq_L\overleftarrow{e_2}$, then
\begin{equation}\label{eqncomp}
{\varepsilon}_1 {\varepsilon}_2 :=  [  e_1, {\theta} \rangle
\end{equation}
where for $h\in \mathcal{S}(\mathbf{r}(\alpha),\mathbf{d}(\beta))=\mathcal{S}(f_1,v)$,
$$\theta := (\alpha\circ\beta)_h=(\alpha{\downharpoonright} f_1 h)\:\epsilon(f_1 h,h)\:\epsilon(h,hv)\: (hv{\downharpoonleft}\beta).$$

Observe that the sandwich set $\mathcal{S}(f_1,v)$ is well-defined as it depends only on the $\mathrel{\mathscr{L}}$-class of $f_1$ and $\mathrel{\mathscr{R}}$-class of $v$. To verify that the composition is well-defined, we need to prove the following lemmas. The first lemma shows that the composition in (\ref{eqncomp}) is independent of the representing element of the morphism $ [  e, {\alpha} \rangle $ in $\mathcal{E}_L$.

\begin{lem}\label{lemcmp1}
Let $ [  e_1, {\alpha} \rangle = [  e'_1, {\alpha'} \rangle $ and  $  [  e_2, {\beta} \rangle = [  e'_2, {\beta'} \rangle $ be morphisms in the category $\mathcal{E}_L$ such that $\mathbf{r}(\overleftarrow{\alpha})\leq_L\overleftarrow{e_2}$ and $\mathbf{r}(\overleftarrow{\alpha'})\leq_L\overleftarrow{e'_2}$. Then for any fixed $h\in \mathcal{S}(\mathbf{r}(\alpha),\mathbf{d}(\beta))=\mathcal{S}(\mathbf{r}(\alpha'),\mathbf{d}(\beta'))$,
$$ [  e_1, (\alpha\circ\beta)_h \rangle = [  e'_1, (\alpha'\circ\beta')_h \rangle .$$
\end{lem}
\begin{proof}
Clearly $e_1\mathrel{\mathscr{L}}e'_1$ and using \cite[Lemma 4.7]{mem}, we have $\overline{(\alpha\circ\beta)_h}=\overline{(\alpha'\circ\beta')_h}$. Hence the lemma follows from Lemma \ref{lemp}.
\end{proof}
We also need to show that (\ref{eqncomp}) is independent of the choice of the sandwich element in the sandwich set.
\begin{lem}\label{lemcmp2}
Let $ [  e_1, {\alpha} \rangle ,  [  e_2, {\beta} \rangle  \in \mathcal{E}_L$ such that $\mathbf{r}(\overleftarrow{\alpha})\leq_L\overleftarrow{e_2}$. For $h,h'\in \mathcal{S}(\mathbf{r}(\alpha),\mathbf{d}(\beta))$,
$$ [  e_1, (\alpha\circ\beta)_h \rangle = [  e_1, (\alpha\circ\beta)_{h'} \rangle .$$
\end{lem}
\begin{proof}
The lemma follows from \cite[Lemma 4.8]{mem} and Lemma \ref{lemp}.
\end{proof}

\begin{figure}[h]
\xymatrixcolsep{1pc}\xymatrixrowsep{1pc}
\xymatrix{
\\
e_1\ar@{->}[ddddrrr]^{\leqslant  }\ar@{->}[dddddddddrrrr]_{\leqslant  }&&&&&&&&g_1=e_2\ar@{->}[ddddrr]^{\leqslant  }&&&&&&&&g_2\ar@{.}[ddddlll]_{\lel}\ar@{.}[ddddddddlll]^{\lel}\\
\\
\\
\\
&&&u\ar@{->}@/^1pc/[rrr]^{\alpha}&&&f_1\ar@{.}[uuuurr]^{\lel }\ar@{.}[dddddrr]_{\leqslant }&&&&v\ar@{.}[ddddl]^{\leqslant }\ar@{->}@/^1pc/[rrr]^{\beta}&&&f_2\\
\\\\\\
&&&&&&&&h\ar@{->}[r]&hv\ar@{->}@/^1pc/[rrrr]^{hv{\downharpoonleft}\beta}&&&&k_2\\
&&&&k_1\ar@{->}@/^1pc/[rrrr]^{\alpha{\downharpoonright} f_1 h}&&&&f_1h\ar@{-}[u]
}
\caption{Composition of morphisms in $\mathcal{E}_L$}\label{comp}
\end{figure}
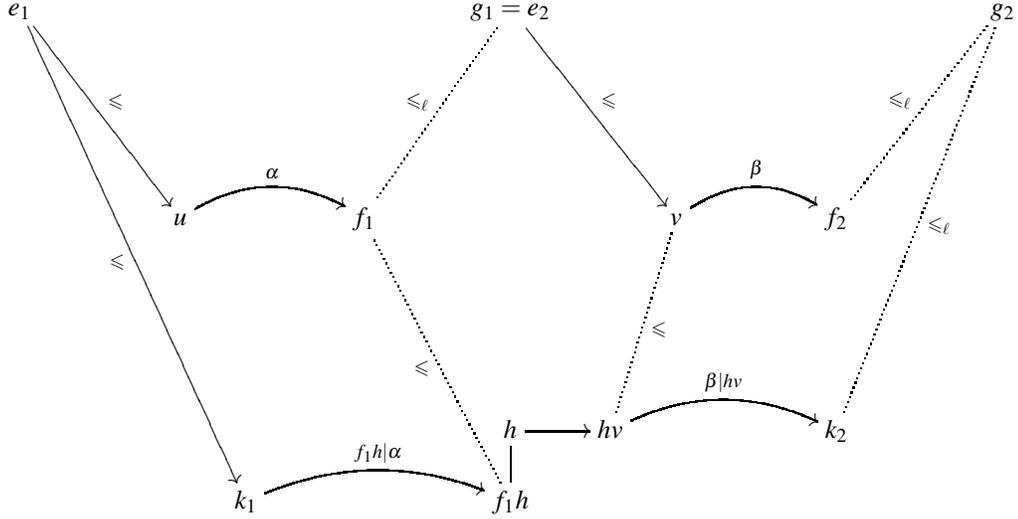

The partial composition defined in the category $\mathcal{E}_L$ may be illustrated using Figure \ref{comp} in the inductive groupoid $\mathcal{G}$. The solid arrows correspond to the relevant morphisms in the category $\mathcal{E}_L$.
\begin{rmk}\label{rmkepi}
Observe that the partial composition of $ [  e_1, {\alpha} \rangle $ and $ [  e_2, {\beta} \rangle $ in the category $\mathcal{E}_L$ does not depend on the condition that $\mathbf{r}(\overleftarrow{\alpha})\leq_L\overleftarrow{e_2}$. Hence this composition may be defined between any two morphisms in $\mathcal{E}_L$. 
\end{rmk}

\begin{pro}\label{proel}
$\mathcal{E}_L$ is a category.
\end{pro}
\begin{proof}
Here, given two morphisms $ [  e_1, {\alpha} \rangle $ and $ [  e_2, {\beta} \rangle $ in the category $\mathcal{E}_L$, we shall prove well-definedness and associativity of the composition when  $\mathbf{r}(\overleftarrow{\alpha})\leq_L\overleftarrow{e_2}$. Clearly, the result also holds, in particular, when $\mathbf{r}([  e_1, {\alpha} \rangle)=\mathbf{d}([  e_2, {\beta} \rangle)$, i.e., when $\mathbf{r}(\overleftarrow{\alpha}) = \overleftarrow{e_2}$.

Let $ [  e_1, {\alpha} \rangle = [  e'_1,{\alpha'} \rangle $ and  $  [   e_2, {\beta} \rangle = [  e'_2, {\beta'} \rangle $ be morphisms in the category $\mathcal{E}_L$ such that $\mathbf{r}(\overleftarrow{\alpha})\leq_L\overleftarrow{e_2}$ and $\mathbf{r}(\overleftarrow{\alpha'})\leq_L\overleftarrow{e'_2}$. Then for $h,h'\in \mathcal{S}(\mathbf{r}(\alpha),\mathbf{d}(\beta))$, by Lemma \ref{lemcmp1} and Lemma \ref{lemcmp2},
$$ [  e_1,(\alpha\circ\beta)_h \rangle = [  e'_1,(\alpha'\circ\beta')_h \rangle = [  e'_1,(\alpha'\circ\beta')_{h'} \rangle .$$
Hence the composition is well-defined.

Now we need to verify identity and associativity. Given a morphism $ [  e,\alpha \rangle $ in $\mathcal{E}_L$ from $\overleftarrow{e}$ to $\overleftarrow{f}$,
$$ [  e, 1_e \rangle   [  e,\alpha \rangle  =  [  e, (1_e \circ \alpha)_{\mathbf{d}(\alpha)}  \rangle =  [  e, \alpha  \rangle  \text{ and }  [  e,\alpha \rangle  [  f, 1_f \rangle  =  [  e, (\alpha \circ 1_f)_f  \rangle =  [  e, \alpha  \rangle .$$
So $ [  e, 1_e \rangle $ is the identity element at $\overleftarrow{e}\in v\mathcal{E}_L$.

Also, if $ [  e,\alpha \rangle $, $ [  f,\beta \rangle $ and $ [  g,\gamma \rangle $ are composable morphisms in the category $\mathcal{E}_L$, then for $h_1\in \mathcal{S}(\mathbf{r}(\alpha),\mathbf{d}(\beta))$ and $h_2\in \mathcal{S}(\mathbf{r}(\beta),\mathbf{d}(\gamma))$, by \cite[Lemma 4.4 ]{mem}, there exist $h \in \mathcal{S}(\mathbf{r}(\alpha),\mathbf{d}((\beta\circ\gamma)_{h_2}))$ and $h' \in \mathcal{S}(\mathbf{r}((\alpha\circ\beta)_{h_1}),\mathbf{d}(\gamma))$ such that
\begin{equation}\label{eqnasso}
((\alpha\circ\beta)_{h_1})\circ\gamma)_{h'} =(\alpha\circ(\beta\circ\gamma)_{h_2})_{h}.
\end{equation}
So,
\begin{align*}
(\: [  e,\alpha \rangle \: [  f,\beta \rangle \:)\: [  g,\gamma \rangle =&(\: [  e,(\alpha\circ\beta)_{h_1} \rangle \:)\: [  g,\gamma \rangle \\
=& [  e,((\alpha\circ\beta)_{h_1}\circ\gamma)_{h'} \rangle \\
=& [  e,(\alpha\circ(\beta\circ\gamma)_{h_2})_{h} \rangle \\
=& [  e,\alpha \rangle \:(\: [  f,(\beta\circ \gamma)_{h_2} \rangle \:)\\
=& [  e,\alpha \rangle \:(\: [  f,\beta \rangle \: [  g,\gamma \rangle \:).\\
\end{align*}
Thus associativity also holds. Hence $\mathcal{E}_L$ is a category.
\end{proof}

Having prepared all the necessary ingredients, now we are in a position to define our category $\mathcal{L}_G$. Recall that $\mathcal{L}_G:=E/\mathscr{L}$ and we define the morphisms in $\mathcal{L}_G= \mathcal{E}_L\otimes \mathcal{P}_L$ as follows:
\begin{equation*}
\mathcal{L}_G:=\{({\varepsilon},{j})\in \mathcal{E}_L \times \mathcal{P}_L : \mathbf{r}(\varepsilon) =\mathbf{d}(j) \}.
\end{equation*}
In the sequel, a morphism $(\varepsilon,j)$ in $\mathcal{L}_G$ shall be denoted as $ [  e,\alpha,f \rangle $ where $\varepsilon= [  e,\alpha \rangle $ and $\overleftarrow{f}=\mathbf{r}(j)$. Hence, by Lemma \ref{lemp}, the morphisms $ [  e,\alpha,f \rangle  =  [  g,\beta,h \rangle $ if and only if $e\mathrel{\mathscr{L}}g$, $\overline{\alpha}=\overline{\beta}$ and $f\mathrel{\mathscr{L}}h$. So, if $e\mathrel{\mathscr{L}}\mathbf{d}(\alpha)$ and $g\mathrel{\mathscr{L}}\mathbf{d}(\beta)$ in the biordered set $E$, we have that $ [  e,\alpha,f \rangle  =  [  g,\beta,h \rangle $ if and only if $\overleftarrow{\alpha}=\overleftarrow{\beta}$ and $f\mathrel{\mathscr{L}}h$. In particular, $ [  e,1_e,e \rangle = [  f,1_f,f \rangle $ in $\mathcal{L}_G$ if and only if $e\mathscr{L}f$.

Further, given two morphisms $ [  e_1,\alpha,g_1 \rangle  ,  [  e_2,\beta,g_2 \rangle $ in the category $\mathcal{L}_G$ such that $\overleftarrow{g_1}=\overleftarrow{e_2}$, we define a partial composition in $\mathcal{L}_G$ as follows. For $h\in \mathcal{S}(\mathbf{r}(\alpha),\mathbf{d}(\beta))$,
\begin{equation}\label{eqncomplg}
 [  e_1,\alpha,g_1 \rangle  \:  [  e_2,\beta,g_2 \rangle  =  [  e_1,(\alpha\circ\beta)_h,g_2 \rangle .
\end{equation}
\begin{pro}
$\mathcal{L}_G$ is a category.
\end{pro}
\begin{proof}
Well-definedness and associativity of the composition in  $\mathcal{L}_G$ easily follow from well-definedness and associativity of the composition in the category $\mathcal{E}_L$. (See Proposition~\ref{proel} and Figure \ref{comp}.)

Also, given a morphism $ [  e,\alpha,f \rangle $ in $\mathcal{L}_G$ from $\overleftarrow{e}$ to $\overleftarrow{f}$, we can see that
$$ [  e, 1_e, e \rangle   [  e,\alpha,f \rangle  =  [  e, (1_e \circ \alpha)_{\mathbf{d}(\alpha)},f  \rangle =  [  e, \alpha,f  \rangle $$ { and } $$ [  e,\alpha,f \rangle  [  f, 1_f,f \rangle  =  [  e, (\alpha \circ 1_f)_f,f  \rangle =  [  e, \alpha,f  \rangle .$$
So $ [  e, 1_e,e \rangle $ is the identity element at $\overleftarrow{e}\in v\mathcal{L}_G$. Hence $\mathcal{L}_G$ is a category.
\end{proof}
\begin{rmk}\label{rmksubcat}
We know that $v\mathcal{P}_L=v\mathcal{G}_L=v\mathcal{Q}_L=v\mathcal{E}_L=v\mathcal{L}_G=E/\mathscr{L}$. We can see that in fact all these categories are subcategories of $\mathcal{L}_G$ by the following identification.
\begin{center}
\vspace{.5cm}
\begin{tabular}{ccccc}
\hline
Category&&Typical morphism&& Corresponding morphism in $\mathcal{L}_G$\\
\hline\\
$\mathcal{P}_L$&&$j_{\mathcal{L}}(e,f)$&&$ [  e,1_e,f \rangle $\\[3pt]
$\mathcal{G}_L$&&$\overleftarrow{\alpha}(e,f)$&&$ [  e,\alpha,f \rangle  $\\[3pt]
$\mathcal{Q}_L$&&$q_\mathcal{L}(e,u)$&&$ [  e,1_u,u \rangle  $\\[3pt]
$\mathcal{E}_L$&&$ [  e, \alpha \rangle $&&$ [  e,\alpha,\mathbf{r}(\alpha) \rangle  $\\[3pt]
\hline
\end{tabular}
\vspace{1cm}
\end{center}
\end{rmk}

\begin{lem}\label{lemcso}
$(\mathcal{L}_G,\mathcal{P}_L)$ is a category with subobjects.
\end{lem}
\begin{proof}
Clearly, $\mathcal{L}_G$ is a small category and by Proposition \ref{lempl}, the category $\mathcal{P}_L$ is a strict preorder subcategory of $\mathcal{L}_G$ such that $v\mathcal{L}_G = v\mathcal{P}_L$. Let $ [  e,1_e,f \rangle $ be a morphism in $\mathcal{P}_L$. If $ [  g,\alpha,e \rangle \: [  e,1_e,f \rangle = [  h,\beta,e \rangle \: [  e,1_e,f \rangle $, then $ [  g,(\alpha\circ1_e)_e,f \rangle = [  h,(\beta\circ1_e)_e,f \rangle $. That is, $ [  g,\alpha,f \rangle = [  h,\beta,f \rangle $.
So by Lemma \ref{lemp}, we have $g\mathrel{\mathscr{L}}h$ and  $\overline{\alpha}=\overline{\beta}$ where $\overline{\alpha}$ is the $p$-class of $\alpha$ in $\mathcal{G}$ (see (\ref{eqnp})). Hence $ [  g,\alpha,e \rangle = [  h,\beta,e \rangle $ and so $ [  e,1_e,f \rangle  \in \mathcal{P}_L$ is a monomorphism.

Now if $ [  e,1_e,f \rangle = [  g,\alpha,h \rangle  [  h,1_h,k \rangle $ for $ [  e,1_e,f \rangle , [  h,1_h,k \rangle \in \mathcal{P}_L$ and $ [  g,\alpha,h \rangle  \in \mathcal{L}_G$, then
since $ [  g,\alpha,h \rangle  [  h,1_h,k \rangle = [  g,\alpha,k \rangle $, we have $ [  g,\alpha,k \rangle  =  [  e,1_e,f \rangle $. Then $e\mathrel{\mathscr{L}}g$, $f\mathrel{\mathscr{L}}k$, and  $\overleftarrow{\alpha}=\overleftarrow{1_e}=\overleftarrow{1_g}$. So
$ [  g,\alpha,h \rangle  =  [  g,1_g,h \rangle $, i.e., $ [  g,\alpha,h \rangle \in \mathcal{P}_L$.
Hence $(\mathcal{L}_G,\mathcal{P}_L)$ is a category with subobjects.
\end{proof}

So, if $e\lel f$, a morphism $ [  e,1_e,f \rangle  \in \mathcal{P}_L\subseteq\mathcal{L}_G$ shall be an \emph{inclusion} in the category $\mathcal{L}_G$. So, two inclusions $ [  e,1_e,f \rangle$ and $[  g,1_g,h \rangle $ are equal if and only if $e\mathrel{\mathscr{L}}g$ and $f\mathrel{\mathscr{L}}h$  if and only if  $j_\mathcal{L}(e,f)= j_\mathcal{L}(g,h)$.

\begin{lem}\label{lemsplit}
Every inclusion in $\mathcal{L}_G$ {splits}.
\end{lem}
\begin{proof}
Let $ [  e,1_e,f \rangle $ be an inclusion in $\mathcal{L}_G$. Then since $fe\mathrel{\mathscr{L}}e$, we have a retraction $ [  f,1_{ef},{ef} \rangle $ in $\mathcal{L}_G$  such that
$$ [  e,1_e,f \rangle \: [  f,1_{fe},{fe} \rangle = [  e,(1_e\circ1_{fe})_{fe},{fe} \rangle = [  e,1_{e},{fe} \rangle = [  e,1_{e},{e} \rangle .$$
Hence the lemma.
\end{proof}

\begin{rmk}\label{rmkog1}
If $e\lel f$, a morphism $ [  f,1_{fe},fe \rangle  \in \mathcal{Q}_L\subseteq\mathcal{L}_G$ is a right inverse to the inclusion $ [  e,1_{e},f \rangle$. Hence it will be a \emph{retraction} in the category $\mathcal{L}_G$. Observe that  two retractions $ [  e,1_u,u \rangle$ and $[  f,1_v,v \rangle $ are equal if and only if $e\mathcal{L}f$ and $v=fu$ if and only if $q_\mathcal{L}(e,u)=q_\mathcal{L}(f,v)$.

For $\alpha\in \mathcal{G}(e,f)$, we can easily verify that the morphism $ [  e,\alpha,f \rangle  \in \mathcal{G}_L\subseteq\mathcal{L}_G$ is an \emph{isomorphism} in $\mathcal{L}_G$. Observe that two isomorphisms $ [  e,\alpha,f \rangle$ and $ [  g,\beta,h \rangle $ are equal if and only if  $e \mathrel{\mathscr{L}} g,$ $f \mathrel{\mathscr{L}} h,$  and  $\alpha\:\epsilon(f,h) = \epsilon(e,g)\:\beta$ if and only if $\overleftarrow{\alpha}=\overleftarrow{\beta}$.

Also, given an arbitrary morphism $ [  e, \alpha,f \rangle $ in $\mathcal{L}_G$, we can easily see that the morphism $ [  e,\alpha,\mathbf{r}(\alpha) \rangle  \in \mathcal{E}_L\subseteq \mathcal{L}_G$ is the epimorphic component of the morphism $ [  e, \alpha,f \rangle $. The epimorphic component shall be denoted as just $ [  e, \alpha,f \rangle ^\circ$ in the sequel.
\end{rmk}

Now we proceed to construct certain distinguished cones in the category $\mathcal{L}_G$. Given a morphism $\alpha\colon e\to f$ in the inductive groupoid $\mathcal{G}$, recall that $ [  e, \alpha, f \rangle $ is an isomorphism in $\mathcal{L}_G$ from $\overleftarrow{e}$ to $\overleftarrow{f}$. Also, for an arbitrary $\overleftarrow{g}\in v\mathcal{L}_G$, observe that $[  g,1_g,g \rangle $ is the identity morphism at $\overleftarrow{g}$.

Recall from Remark \ref{rmkepi} that we can compose any two epimorphisms in $\mathcal{E}_L$ and similarly in the category $\mathcal{L}_G$. In particular, we can compose $[  g,1_g,g \rangle $ with $ [  e, \alpha, f \rangle $, using the rule \eqref{eqncomp} for the corresponding right epis $[g,1_g\rangle$ and $[e,\alpha\rangle$, respectively. We get the following morphism in $\mathcal{L}_G$ from $\overleftarrow{g}$ to $\overleftarrow{f}$:  
$$ [  g,1_g,g \rangle   [  e, \alpha, f \rangle  =  [  g,(1_g\circ\alpha)_h,f \rangle, $$ 
where $h\in \mathcal{S}(g,e)$. 

Fixing $\alpha\in\mathcal{G}$ and making $\overleftarrow{g}$ run over $v\mathcal{L}_G$, we define a map $r^\alpha\colon v\mathcal{L}_G \to \mathcal{L}_G $ as $$r^\alpha\colon\overleftarrow{g}\mapsto [  g,(1_g\circ\alpha)_h,f \rangle,$$ 
where $h\in \mathcal{S}(g,e)$ and $f\mathrel{\mathscr{L}}\mathbf{r}(\alpha)$.
\begin{rmk}\label{rmkcone}
Observe that for $\alpha,\beta \in \mathcal{G}$ such that $\overline{\alpha}=\overline{\beta}$, we have  $\overline{(1_g\circ\alpha)_h}=\overline{(1_g\circ\beta)_h}$ by \cite[Lemma 4.7]{mem}. So using Lemma \ref{lemp}, we have $[  g,(1_g\circ\alpha)_h,f \rangle = [  g,(1_g\circ\beta)_h,f \rangle $. That is, $r^\alpha=r^\beta$.
\end{rmk}
\begin{lem}
The map $r^\alpha$ is a cone in $\mathcal{L}_G$.
\end{lem}
\begin{proof}
Since $r^\alpha(g)$ is independent of $g\in \overleftarrow{g}$ and $h\in\mathcal{S}(g,\mathbf{d}(\alpha))$, we can see that $r^\alpha$ is a well-defined map. Now if $g'\leq_Lg$, then $g'\lel  g $ and so we have an inclusion $ [  g',1_{g'},g \rangle $ in $\mathcal{L}_G$.
Then,
\begin{align*}
 [  g',1_{g'},g \rangle  r^\alpha(\overleftarrow{g})
&= [  g',1_{g'},g \rangle  [  g,(1_g\circ\alpha)_h,f \rangle \\
&= [  g',(1_{g'}\circ (1_g\circ\alpha)_h)_{h_1},f \rangle \quad \quad (\text{where } h_1\in \mathcal{S}(g',\mathbf{d}((1_g\circ\alpha)_h)) \\
&= [  g',((1_{g'}\circ 1_g)_{g'}\circ\alpha)_{h_2},f \rangle \quad \quad (\text{where } h_2\in \mathcal{S}(g',\mathbf{d}(\alpha)) \text{ using (\ref{eqnasso})})\\
&= [  g',(1_{g'}\circ\alpha)_{h_2},f \rangle \\
&=r^\alpha(\overleftarrow{g'}).
\end{align*}

Also since $r^\alpha(\overleftarrow{e})= [  e,\alpha,f \rangle $ is an isomorphism in the category $\mathcal{L}_G$, we see that $r^\alpha$ is a cone in $\mathcal{L}_G$.
\end{proof}

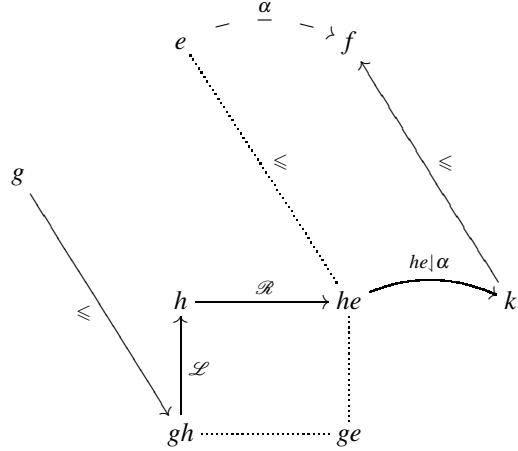
\begin{figure}[h]
\centering
$\xymatrixcolsep{4pc}\xymatrixrowsep{3pc}\xymatrix
{
 &e \ar@{-->}@/^.7pc/[r]^{\alpha} \ar@{.}[rdd]_{}^{\leqslant }&f\\
 g\ar@{->}[ddr]_{\leqslant }^{}\\
 &h \ar[r]^{\mathrel{\mathscr{R}}}&he\ar@{.}[d]\ar@/^.7pc/[r]^{he{\downharpoonleft}\alpha} &k\ar[luu]_{\leqslant }\\
 &gh \ar@{->}[u]_{\mathrel{\mathscr{L}}}\ar@{.}[r] & ge
}
$
\caption{Cone $r^\alpha$ in the category $\mathcal{L}_G$}\label{figpc}
\end{figure}

Figure \ref{figpc} illustrates the cone $r^\alpha$ in the category $\mathcal{L}_G$. The dashed arrow represents the morphism $\alpha\in\mathcal{G}$. The solid arrows give rise to the relevant morphisms in $\mathcal{L}_G$ arising from the inductive groupoid. If $h\in \mathcal{S}(g,e)$, then $ [  g,1_{gh},gh \rangle $ is a retraction in $\mathcal{L}_G$ from $\overleftarrow{g}$ to $\overleftarrow{gh}$. The morphism $\epsilon(gh,h)\:\epsilon(h,he)\:(he{\downharpoonleft}\alpha)$ is an isomorphism in $\mathcal{G}$ from $gh$ to $k=\mathbf{r}(he{\downharpoonleft}\alpha)$. So, $ [  gh, {\epsilon(gh,h)\:\epsilon(h,he)\:(he{\downharpoonleft}\alpha)}, k \rangle $ is an isomorphism in $\mathcal{L}_{G}$ from $\overleftarrow{gh}$ to $\overleftarrow{k}$. Also, $ [  k,1_k,f \rangle $ is an inclusion from $\overleftarrow{k}$ to $\overleftarrow{f}$. Composing these morphisms, we get a morphism $r^\alpha(\overleftarrow{g})=  [  g, {\epsilon(gh,h)\:\epsilon(h,he)\:(he{\downharpoonleft}\alpha)},f \rangle $ from $\overleftarrow{g}$ to $\overleftarrow{f}$ in the category $\mathcal{L}_G$.

Similarly, since $1_e\in \mathcal{G}$, we can build cones $r^e(\overleftarrow{g})= [  g,(1_g\circ1_e)_h,e \rangle $ by taking $\alpha=1_e$ and $h\in \mathcal{S}(g,e)$.

We shall need the following lemma in the sequel.

\begin{lem}\label{lemr}
Let $r^\alpha$ and $r^\beta$ be cones in the category $\mathcal{L}_G$ as defined above. Then $$r^\alpha\:r^\beta =r^{(\alpha\circ\beta)_h}\ \text{ where } h\in \mathcal{S}(\mathbf{r}(\alpha),\mathbf{d}(\beta)).$$
\end{lem}
\begin{proof}
Let $\alpha,\beta\in \mathcal{G}$ be such that $\mathbf{r}(\alpha)=f$ and $\mathbf{r}(\beta)=g$,  then for an arbitrary $e\in E$, using the composition (\ref{eqnsg1}) and for $h_1\in\mathcal{S}(e,\mathbf{d}(\alpha))$ and $ h_2\in \mathcal{S}(f,\mathbf{d}(\beta))$, we have \begin{align*}
r^\alpha\:r^\beta(\overleftarrow{e})&=[e,(1_e\circ\alpha)_{h_1}, f\rangle\:[f,(1_f\circ \beta)_{h_2},g\rangle^\circ \\
&=[e,(1_e\circ\alpha)_{h_1}, f\rangle\:[f,(1_f\circ \beta)_{h_2},g'\rangle\ 
\text{ where }g':=\mathbf{r}((1_f\circ \beta)_{h_2})\\
&=[e,((1_e\circ\alpha)_{h_1} \circ (1_f\circ \beta)_{h_2})_{h_3},g'\rangle\ \ \text{ where }h_3\in\mathcal{S}(\mathbf{r}((1_e\circ\alpha)_{h_1}),\mathbf{d}((1_f\circ \beta)_{h_2})).
\end{align*}
Then, using \cite[Lemma 4.4]{mem} with $h_4\in\mathcal{S}(f,\mathbf{d}((1_f\circ \beta)_{h_2}))$, $h_5\in \mathcal{S}(e,\mathbf{d}(\alpha \circ (1_f\circ \beta)_{h_2})_{h_4}))$, $h_6\in\mathcal{S}(\mathbf{d}(\alpha),f)$, and $h\in \mathcal{S}(\mathbf{r}((\alpha \circ 1_f)_{h_6}),\mathbf{d}(\beta))$, we have
\begin{align*}
((1_e\circ\alpha)_{h_1} \circ (1_f\circ \beta)_{h_2})_{h_3}&=(1_e\circ(\alpha \circ (1_f\circ \beta)_{h_2})_{h_4})_{h_5} \\
&=(1_e\circ((\alpha \circ 1_f)_{h_6}\circ \beta)_{h})_{h_5}\\
&=(1_e\circ(\alpha \circ\beta)_{h})_{h_5} \end{align*}
The last equality holds as $\mathbf{r}(\alpha)=f \text{ implies } (\alpha\circ1_f)_{h_6}=\alpha$ and $g'\mathrel{\mathscr{L}}(\alpha \circ\beta)_{h}$.

Combining the above arguments, for an arbitrary $e\in E$, we have
\begin{align*}
r^\alpha\:r^\beta(\overleftarrow{e})&= [e,((1_e\circ\alpha)_{h_1} \circ (1_f\circ \beta)_{h_2})_{h_3},g'\rangle \\
&= [e,(1_e\circ(\alpha \circ\beta)_{h})_{h_5},g'\rangle\\ 
&=r^{(\alpha\circ\beta)_h}(\overleftarrow{e}).
\end{align*}
Hence, $r^\alpha\:r^\beta =r^{(\alpha\circ\beta)_h}$. \end{proof}

The above proof essentially shows how the role played by sandwich sets in inductive groupoid theory is captured by the cone multiplication (\ref{eqnsg1}) in cross-connection theory.

\begin{rmk}\label{rmkic}
Observe that if $ef$ is a basic product in $E$, then in the inductive groupoid $\mathcal{G}$, $\overline{1_{ef}}=\overline{(1_e\circ1_f)_h}$ where $h\in \mathcal{S}(e,f)$. Hence using Remark \ref{rmkcone} and Lemma \ref{lemr}, we have $$r^{ef}=r^{(1_e\circ1_f)_h}=r^er^f.$$
Hence for an element $e\in E$, the cone $r^e$ is an idempotent cone in $\mathcal{L}_G$.
\end{rmk}

\begin{thm}
$(\mathcal{L}_G,\mathcal{P}_L)$\/ forms a normal category.
\end{thm}
\begin{proof}
By Lemma \ref{lemcso}, $(\mathcal{L}_G,\mathcal{P}_L)$ is a category with subobjects. Given an arbitrary morphism in $ [  e,\alpha,f \rangle  \in \mathcal{L}_G$,
$$ [  e,\alpha,f \rangle = [  e,1_{\mathbf{d}(\alpha)},\mathbf{d}(\alpha) \rangle \: [  \mathbf{d}(\alpha),\alpha,\mathbf{r}(\alpha) \rangle \: [  \mathbf{r}(\alpha),1_{\mathbf{r}(\alpha)},f \rangle $$
is a normal factorisation such that $ [  e,1_{\mathbf{d}(\alpha)},\mathbf{d}(\alpha) \rangle $ is a retraction, $ [  \mathbf{d}(\alpha),\alpha,\mathbf{r}(\alpha) \rangle $ is an isomorphism and $ [  \mathbf{r}(\alpha),1_{\mathbf{r}(\alpha)},f \rangle $ is an inclusion. By Lemma \ref{lemsplit}, every inclusion in $\mathcal{L}_G$ splits. Also, given an object $\overleftarrow{e}\in v\mathcal{L}_G$, the map $r^e$ is an idempotent cone with apex $\overleftarrow{e}$. Hence the theorem.
\end{proof}

\begin{rmk}
By Remarks \ref{rmkog} and \ref{rmkog1}, we have several associated `one-sided' ordered groupoids with a given cross-connection. The above theorem describes the normal category (a `one-sided' category with a regular partially ordered set \cite{grilcross} as its object set) constructed from an inductive groupoid. So, it may be worthwhile to investigate if we can associate a normal category from a suitable ordered groupoid by assuming its object set to form a regular partially ordered set.
\end{rmk}

\subsection{The normal category $\mathcal{R}_G$}
Dually, given an inductive groupoid $\mathcal{G}$, we proceed to build a `right-hand side' normal category via intermediary categories $\mathcal{P}_R$, $\mathcal{G}_R$ and $\mathcal{Q}_R$, as follows. One major difference from the construction of the category $\mathcal{L}_G$ in the preceding subsection is that the morphisms in the ordered groupoid $\mathcal{G}_R$ are induced from the inductive groupoid in the opposite direction (see below). We shall omit most of the details as the dual arguments of the construction of $\mathcal{L}_G$ will suffice. 

First, given the inductive groupoid $\mathcal{G}$ with regular biordered set $E$, consider the quotient set $v\mathcal{R}_G:= E/\mathscr{R}$ where $\mathscr{R}\:=\:\ler \: \cap \:(\ler )^{-1}$. This gives a regular partially ordered set with respect to $\leq_R\: := \: \ler /\mathscr{R} $. We shall denote the $\mathscr{R}$-class of $e$ in $E/\mathscr{R}$ by $\overrightarrow{e}$.

Now we aim to build three categories --- $\mathcal{P}_R$, $\mathcal{G}_R$ and $\mathcal{Q}_R$ --- that all have $E/\mathscr{R}$ as the object set, that is, $v\mathcal{P}_R=v\mathcal{G}_R=v\mathcal{Q}_R:= E/\mathscr{R}$. It remains to define the morphisms in each of these three categories.

We begin with $\mathcal{P}_R$ which has a unique morphism, denoted $j_\mathcal{R}(e,f)$, from $\overrightarrow{e}$ to $\overrightarrow{f}$ for each pair $(\overrightarrow{e}, \overrightarrow{f})$ such that $\overrightarrow{e}\leq_R \overrightarrow{f}$. Clearly, $\mathcal{P}_R$ is a strict preorder category under the following composition:
$$j_\mathcal{R}(e,f)\: j_\mathcal{R}(f,g) := j_\mathcal{R}(e,g).$$

We proceed with defining the morphisms in $\mathcal{G}_R$. Given any two morphisms $\alpha,\beta$ in the inductive groupoid $\mathcal{G}$, we first define an equivalence  relation $\sim_R$ as follows:
\begin{equation}\label{eqgr}
\alpha \sim_R \beta \:\iff\: \mathbf{d}(\alpha) \mathrel{\mathscr{R}} \mathbf{d}(\beta),\ \mathbf{r}(\alpha) \mathrel{\mathscr{R}} \mathbf{r}(\beta) \:\text{ and }\: \alpha\:\epsilon(\mathbf{r}(\alpha),\mathbf{r}(\beta)) = \epsilon(\mathbf{d}(\alpha),\mathbf{d}(\beta))\:\beta.
\end{equation}
Given a morphism $\alpha\in\mathcal{G}(e,f)$, we denote the $\sim_R$-class of $\mathcal{G}$ containing $\alpha$ by $\overrightarrow{\alpha}$ and we treat $\overrightarrow{\alpha}$ as a morphism in $\mathcal{G}_R$ from $\overrightarrow{f}$ to $\overrightarrow{e}$; observe the direction change! By the definition of $\overrightarrow{\alpha}$, we have  $\mathbf{d(\overrightarrow{\alpha})}= \overrightarrow{f}=\overrightarrow{\mathbf{r}(\alpha)}$ so that $\mathbf{d(\overrightarrow{\alpha})}= \overrightarrow{\mathbf{r}(\alpha)}$ for every  $\alpha\in\mathcal{G}$.

Further for $\overrightarrow{\alpha},\overrightarrow{\beta}\in \mathcal{G}_R$ such that $\mathbf{d}(\alpha)\mathrel{\mathscr{R}}\mathbf{r}(\beta)$, we define a composition in $\mathcal{G}_R$ as
$$\overrightarrow{\alpha}\overrightarrow{\beta}:= \overrightarrow{\beta\:\epsilon(\mathbf{r}(\beta),\mathbf{d}(\alpha))\:\alpha}$$
so that $\overrightarrow{\alpha}\overrightarrow{\beta}$ will be a morphism in $\mathcal{G}_R$ from $\mathbf{d}(\overrightarrow{\alpha})=\overrightarrow{\mathbf{r}(\alpha)}$ to $\mathbf{r}(\overrightarrow{\beta})=\overrightarrow{\mathbf{d}(\beta)}$. It can be shown that $\mathcal{G}_R$ is a groupoid under the above composition.

Finally, for the morphisms in $\mathcal{Q}_R$, if $f\ler e$, then for each $u\in E$ such that $u\leqslant e$ and $u \mathrel{\mathscr{R}}f$, we define a morphism $q_\mathcal{R}(e,u)$ in $\mathcal{Q}_R$ from $\overrightarrow{e}$ to $\overrightarrow{f}$. If we have two morphisms $q_\mathcal{R}(e,u)$ and $q_\mathcal{R}(f,v)$ in $\mathcal{Q}_R$ such that $u\mathrel{\mathscr{R}}f$, then we compose them as follows:
$$q_\mathcal{R}(e,u)\:q_\mathcal{R}(f,v) = q_\mathcal{R}(e,vu)$$
so that $\mathcal{Q}_R$ forms a category.

Now, to build the category $\mathcal{E}_R$ using the categories $\mathcal{Q}_R$ and $\mathcal{G}_R$, we consider an intermediary quiver $\mathcal{E}$ with object set $v\mathcal{E}:= E/\mathscr{R}$ and morphisms as follows:
$$\mathcal{E}:=\{(q,\overrightarrow{\alpha})\in \mathcal{Q}_R \times \mathcal{G}_R :  \mathbf{r}(q)=\mathbf{d}(\overrightarrow{\alpha}) \}.$$
Then, we define an equivalence relation $\sim_\mathcal{E}$ on the set $\mathcal{E}$ as follows:  for $\xi_1=(q_1(e,u),\overrightarrow{\alpha})$, $\xi_2=(q_2(f,v),\overrightarrow{\beta})$, 
$$\xi_1\sim_\mathcal{E}\xi_2\iff e\mathrel{\mathscr{R}}f,\:u\mathrel{\mathscr{L}}v\:\text{ and }\overrightarrow{\alpha}=\overrightarrow{\epsilon(u,v)}\:\overrightarrow{\beta}.$$
Now define the category $\mathcal{E}_R$ with the object set $v\mathcal{E}_R := E/\mathscr{R}$ and with morphisms being $\sim_\mathcal{E}$-classes of morphisms in $\mathcal{E}$.
A morphism $(q_\mathcal{R}(e,u),\overrightarrow{\alpha(f,u)})$ in  $\mathcal{E}$ such that  $\mathbf{r}(\alpha)=u$ is called a \emph{left epi} in the category $\mathcal{E}_R$ and shall be denoted by $ \langle  e,\alpha  ]$ in the sequel. We use left epis as `good' representatives of $\sim_\mathcal{E}$-classes: the left epi $ \langle  e,\alpha  ]$ represents the $\sim_\mathcal{E}$-class
$$ \{\:( q_\mathcal{R}(f,v),\overrightarrow{\theta})\in \mathcal{E}:\:f\mathrel{\mathscr{R}}e,\:v\mathrel{\mathscr{L}}\mathbf{r}(\alpha)\text{ and }\overrightarrow{\theta}=\overrightarrow{\alpha \:\epsilon(\mathbf{r}(\alpha),v)}\:\}. $$
Then we define a composition in $ \mathcal{E}_R$ as follows. Let ${\varepsilon}_1=  \langle  e_1, {\alpha}(f_1,u)  ]  $ and ${\varepsilon}_2=  \langle  e_2, {\beta}(f_2,v)  ]  $ be left epis in the category $\mathcal{E}_R$. If $\overrightarrow{f_1}\leq_R\overrightarrow{e_2}$, then we define
\begin{equation}\label{eqncomp1}
{\varepsilon}_1 {\varepsilon}_2 :=  \langle   e_1, {\theta}   ]
\end{equation}
where for $h\in \mathcal{S}(\mathbf{r}(\beta),\mathbf{d}(\alpha))= \mathcal{S}(v,f_1)$,
$$\theta := (\beta\circ\alpha){_h}=(\beta{\downharpoonright}vh)\:\epsilon(vh,h)\: \epsilon(h,hf_1)\:(hf_1{\downharpoonleft}\alpha ).$$

Finally, we define the morphisms in $\mathcal{R}_G= \mathcal{E}_R\otimes \mathcal{P}_R$ as follows:
\begin{equation*}
\mathcal{R}_G:=\{({\varepsilon},{j})\in \mathcal{E}_R \times \mathcal{P}_R : \mathbf{r}(\varepsilon) =\mathbf{d}(j) \}.
\end{equation*}
We shall denote a morphism $(\varepsilon,j)$ in $\mathcal{R}_G$ by $ \langle   e,\alpha,f  ]  $ where $\varepsilon= \langle   e,\alpha  ]  $ and $\overrightarrow{f}=\mathbf{r}(j)$. Hence $ \langle  e,\alpha,f  ]   =  \langle  g,\beta,h  ]  $ if and only if $e\mathrel{\mathscr{R}}g$, $\overline{\alpha}=\overline{\beta}$ and $f\mathrel{\mathscr{R}}h$. So, if $e\mathrel{\mathscr{R}}\mathbf{r}(\alpha)$ and $g\mathrel{\mathscr{R}}\mathbf{r}(\beta)$ in the biordered set $E$, we have that $ \langle  e,\alpha,f  ]   =  \langle  g,\beta,h  ]  $ if and only if $\overrightarrow{\alpha}=\overrightarrow{\beta}$ and $f\mathrel{\mathscr{R}}h$. Recall that here $\overline{\alpha}$ represents the $p$-class of $\alpha$ in $\mathcal{G}$ (see (\ref{eqnp})). In particular, $ \langle  e,1_e,e  ]  = \langle  f,1_f,f  ]  $ in $\mathcal{R}_G$ if and only if $e\mathscr{R}f$.

Further, given two morphisms $ \langle  e_1,\alpha,g_1  ]   ,  \langle  e_2,\beta,g_2  ]  $ in $\mathcal{R}_G$ such that $\overrightarrow{g_1}=\overrightarrow{e_2}$, we define a partial composition in $\mathcal{R}_G$ as follows. For $h\in \mathcal{S}(\mathbf{r}(\beta),\mathbf{d}(\alpha))$,
\begin{equation}
 \langle  e_1,\alpha,g_1  ]   \:  \langle  e_2,\beta,g_2  ]   :=  \langle  e_1,(\beta\circ\alpha){_h},g_2  ]  .
\end{equation}

We can easily verify that $\mathcal{R}_G$ is a category such that $\mathcal{P}_R$, $\mathcal{G}_R$, $\mathcal{Q}_R$ and $\mathcal{E}_R$ are all subcategories of $\mathcal{R}_G$ by the following identification.
	
\begin{center}
	\vspace{.5cm}
	\begin{tabular}{ccccc}
		\hline
		Category&&Typical morphism&& Corresponding morphism in $\mathcal{R}_G$\\
		\hline\\
		$\mathcal{P}_R$&&$j_{\mathcal{R}}(e,f)$&&$ \langle  e,1_e,f  ]  $\\[3pt]
		$\mathcal{G}_R$&&$\overrightarrow{\alpha}$&&$ \langle  \mathbf{r}({\alpha}),\alpha,\mathbf{d}(\alpha)  ]   $\\[3pt]
		$\mathcal{Q}_R$&&$q_\mathcal{R}(e,u)$&&$ \langle  e,1_u,u  ]   $\\[3pt]
		$\mathcal{E}_R$&&$ \langle  e, \alpha  ]  $&&$ \langle  e,\alpha,\mathbf{d}(\alpha)  ]   $\\[3pt]
		\hline
	\end{tabular}
	\vspace{1cm}
\end{center}
	
As in the case of $\mathcal{L}_G$, we can easily see that given a   morphism $ \langle  e,1_e,f  ]   \in \mathcal{P}_R\subseteq\mathcal{R}_G$, it is an {inclusion}, a morphism $ \langle  e,\alpha,f  ]   \in \mathcal{G}_R\subseteq\mathcal{R}_G$ is an {isomorphism} and a morphism $ \langle  f,1_{fe},fe  ]   \in \mathcal{Q}_R\subseteq\mathcal{R}_G$ is a {retraction} in the category $\mathcal{R}_G$. Also, we can verify that $(\mathcal{R}_G,\mathcal{P}_R)$ forms a normal category with distinguished principal cones $l^\alpha$ defined as follows. For a morphism $\alpha$ in the inductive groupoid $\mathcal{G}$ and for every $\overrightarrow{g}\in v\mathcal{R}_G$,
$$l^\alpha(\overrightarrow{g}):=  \langle  g,(\alpha\circ1_g){_h},f  ]  $$
where $h\in \mathcal{S}(\mathbf{r}(\alpha),g)$ and $f \mathrel{\mathscr{R}} \mathbf{d}(\alpha)$.

\subsection{The cross-connection $\Gamma_G$ of an inductive groupoid $\mathcal{G}$}
Now, we proceed to construct the required cross-connection $\Gamma_G$. Recall that the principal cone $r^e$ is defined as $r^e(\overleftarrow{g})= [  g,(1_g\circ1_e)_h,e \rangle $ for each $g\in E$. The cone $r^e$ determines an $H$-functor $H(r^e;-)\colon$ $\mathcal{L}_G\to \mathbf{Set}$ so that $H(r^e;-)\in vN^*\mathcal{L}_G$. Also recall that $\eta_{r^e}$ is the natural isomorphism between the $H$-functor $H(r^e;-)$ and the covariant hom-functor  $\mathcal{L}_G(\overleftarrow{e},-)\colon \mathcal{L}_G\to \mathbf{Set}$ determined by the object $\overleftarrow{e}\in v\mathcal{L}_G$. Now, define a functor $\Gamma_G\colon  \mathcal{R}_G \to N^*\mathcal{L}_G$ as follows:
\begin{equation}\label{eqncxn}
v\Gamma_G(\overrightarrow{e}):=H(r^e;-) \text{ and } \Gamma_G( \langle   e, {\alpha},f  ]  ):=\eta_{r^e} \mathcal{L}_G(  [    f,  {\alpha},e \rangle ,-) \eta_{r^f}^{-1}
\end{equation}
for each $\overrightarrow{e} \in v\mathcal{R}_G$ and for each morphism $\langle   e, {\alpha},f  ] \in \mathcal{R}_G(\overrightarrow{e}, \overrightarrow{f})$. The natural transformation $\Gamma_G(\langle   e, {\alpha},f  ] )$ above may also be described by the following commutative diagram:
\begin{equation*}\label{eta}
\xymatrixcolsep{2pc}\xymatrixrowsep{4pc}\xymatrix
{
	H(r^e;-) \ar[rr]^{\eta_{r^e}} \ar[d]_{\Gamma_G(\langle   e, {\alpha},f  ] )}
	&& \mathcal{L}_G(\overleftarrow{e},-) \ar[d]_{\mathcal{L}([    f,  {\alpha},e \rangle,-)} &\overleftarrow{e}\\
	H(r^f;-) \ar[rr]^{\eta_{r^f}} && \mathcal{L}_G(\overleftarrow{f},-)&\overleftarrow{f}\ar[u]_{[    f,  {\alpha},e \rangle}
}
\end{equation*}

Now, we proceed to show that $\Gamma_G$ is a cross-connection and for that, first we need to prove that $\Gamma_G$ is a local isomorphism. One can prove directly that the functor $\Gamma_G$ is a local isomorphism by working with the morphisms in the functor category $N^*\mathcal{L}_G$, but it would involve a rather cumbersome argument and use of several undefined ideas. So, we take an alternate easier route. We shall realise the functor $\Gamma_G$ as a composition of two functors, namely a local isomorphism $\bar{F}\colon \mathcal{R}_{G} \to \mathcal{R}_{T\mathcal{L}_G}$ and an isomorphism $\bar{G} \colon \mathcal{R}_{T\mathcal{L}_G}\to N^\ast\mathcal{L}_G$; thereby proving that $\Gamma_G$ is a local isomorphism. It is worth noting that a similar technique has been employed in \cite[Section IV.1]{cross}. To this end, we shall require the following discussion about the structure arising from a given normal category. As the reader shall see, this is the only place where regular semigroups surface in our discussion.

Given a normal category $\mathcal{C}$, it can be seen that \cite{cross} the set of all the {cones} in $\mathcal{C}$ with the special composition as defined in (\ref{eqnsg1}) forms a regular semigroup known as the \emph{semigroup of cones} in $\mathcal{C}$ and denoted by $T\mathcal{C}$.

Also, given a regular semigroup $S$, we can  naturally associate with it two normal categories, each arising from the principal left and right ideals respectively (denoted as $\mathcal{L}_S$ and $\mathcal{R}_S$ respectively). An object of the category $\mathcal{L}_S$ is a principal left ideal $Se$ for $e\in E$, and a morphism from $Se$ to $Sf$ is a partial right translation $\rho(e,u,f)\colon  u \in eSf$. 
Two morphisms $\rho(e,u,f)$ and $\rho(g,v,h)$ are equal in $\mathcal{L}_S$ if and only if $e \mathrel{\mathscr{L}} g$, $f\mathrel{\mathscr{L}} h$, $u \in eSf$, $v\in gSh$, $u=ev$, and, consequently, $v=gu$. Dually, we can define $\mathcal{R}_S$.

Given an abstractly defined normal category $\mathcal{C}$ with an associated regular semigroup $T\mathcal{C}$, the relationship of the normal categories arising from the semigroup $T\mathcal{C}$ is described in the following theorem.

\begin{thm}[\!{\cite[Theorem III.19, Theorem III.25]{cross}}]\label{nd}
	Let $\mathcal{C}$ be a normal category with normal dual $N^\ast\mathcal{C}$. The category $\mathcal{L}_{T\mathcal{C}}$ is normal category isomorphic to the category $\mathcal{C}$  and the category $\mathcal{R}_{T\mathcal{C}}$ is normal category isomorphic to the normal dual $N^\ast\mathcal{C}$.
\end{thm}

So now, we first proceed to prove that $\bar{F}\colon \mathcal{R}_{G} \to \mathcal{R}_{T\mathcal{L}_G}$ defined as follows is a local isomorphism. For each $\overrightarrow{e}\in v\mathcal{R}_G$ and for each morphism $\langle   e, {\alpha},f  ] \in \mathcal{R}_G(\overrightarrow{e}, \overrightarrow{f})$,
\begin{equation}\label{eqnf}
v\bar{F}(\overrightarrow{e}):=r^eT\mathcal{L}_G \text{ and } \bar{F}( \langle   e, {\alpha},f  ]  ):=\lambda({r^e}, r^{\alpha},{r^f}).
\end{equation}
To show that $\bar{F}$ is a local isomorphism, we begin by verifying the following lemma.
\begin{lem}\label{lemF}
$\bar{F}$ is a covariant functor.
\end{lem}	
\begin{proof}
	We first need to verify that the $\bar{F}$ is well-defined. If $\overrightarrow{e}	=\overrightarrow{f}$, then $e\mathrel{\mathscr{R}}f$. Then for an arbitrary $k\in E$,
	\begin{align*}
	r^er^f (\overleftarrow{k})&=r^{(1_e\circ1_f)_h} (\overleftarrow{k})&&\text{ where }h\in\mathcal{S}(e,f), \text{ using Lemma \ref{lemr}}\\
	&=r^{e} (\overleftarrow{k})&&\text{ since } e\mathrel{\mathscr{R}}f\text{ implies } ef=e.	\end{align*}
Hence, in the semigroup $T\mathcal{L}_G$, we have $r^e\cdot r^f =r^f$. Similarly, we can show that $r^f\cdot r^e =r^e$. So, $r^e\mathrel{\mathscr{R}} r^f$ and so, the right ideal $r^eT\mathcal{L}_G=r^fT\mathcal{L}_G$. Hence $v\bar{F}$ is well-defined.

Now, suppose that $\langle   e, {\alpha},f  ] =\langle   g, {\beta},h  ] $. That is, $e\mathrel{\mathscr{R}}g$, $\overline{\alpha} = \overline{\beta}$ and $f\mathrel{\mathscr{R}}h	$.
Then, as shown above $r^e \mathrel{\mathscr{R}} r^g$ and $r^f \mathrel{\mathscr{R}}r^h$.
Also, for an arbitrary $k\in E$,
\begin{align*}
r^\beta\:r^e(\overleftarrow{k})&=r^{(\beta\circ1_e)_{h_1}}(\overleftarrow{k})\quad\text{where }h_1\in\mathcal{S}(\mathbf{r}(\beta),e), \text{ using Lemma \ref{lemr}}\\
&=r^{(\alpha\circ1_e)_{h_1}}(\overleftarrow{k}) \quad\text{ since }\overline{\alpha}=\overline{\beta}\\
&=r^\alpha(\overleftarrow{k})\quad\text{ since }\mathbf{r}(\alpha)\lel e \text{ implies }\mathbf{r}(\alpha)e=\mathbf{r}(\alpha)\text{ implies }(\alpha\circ1_e)_{h_1}=\alpha.
\end{align*}
That is, in the semigroup $T\mathcal{L}_G$, we have $r^\beta \:r^e=r^\alpha$ and so $\lambda({r^e}, r^{\alpha},{r^f})=\lambda({r^g}, r^{\beta},{r^h})$. Hence $\bar{F}$ is well-defined. Also, if $\langle   e, {\alpha},f  ] $ and $\langle   g, {\beta},h  ] $ are composable morphisms in the category $\mathcal{R}_G$, then
\begin{align*}
\bar{F}(\langle   e, {\alpha},f  ])\:\bar{F} (\langle   g, {\beta},h  ]  )&=\lambda(r^e,r^\alpha,r^f)\:\lambda(r^g,r^\beta,r^h)\\
&=\lambda(r^e,r^\beta r^\alpha,r^h)\quad\text{ since composition flips in $\mathcal{R}_{T\mathcal{L}_G}$}\\
&=\lambda(r^e,r^{(\beta\circ\alpha)_{h_1}},r^h)\ \text{ for }h_1\in \mathcal{S}(\mathbf{r}(\beta),\mathbf{d}(\alpha)), \text{ using Lemma \ref{lemr}}.
\end{align*}
Also,
\begin{align*}
\bar{F}(\langle   e, {\alpha},f  ] \langle   g, {\beta},h  ]  )&=\bar{F}(\langle   e, {(\beta\circ\alpha)_{h_1}	},h  ]) \quad\text{ for }h_1\in \mathcal{S}(\mathbf{r}(\beta),\mathbf{d}(\alpha))\\
&=\lambda(r^e,r^{(\beta\circ\alpha)_{h_1}},r^h).
\end{align*}
Hence, $$\bar{F}(\langle   e, {\alpha},f  ])\:\bar{F} (\langle   g, {\beta},h  ]  )= \bar{F}(\langle   e, {\alpha},f  ] \langle   g, {\beta},h  ]  ).$$
Moreover, $$\bar{F}(1_{\overrightarrow{e}})= \bar{F}(\langle   e, {1_e},e  ] )=\lambda(r^e,r^e,r^e)= 1_{r^eT\mathcal{L}_G}.$$
Hence $\bar{F}$ is a well-defined covariant functor.
\end{proof}

Before we proceed further, we prove the following lemma which relates the biordered set $E$ of an inductive groupoid $\mathcal{G}$ with the biordered set  of the semigroup of cones $T\mathcal{L}_G$ in the associated normal category $\mathcal{L}_G$.
\begin{lem}\label{lemwr}
The map $\theta\colon E \to E(T\mathcal{L}_G)$ given by $e\mapsto r^e$ is a regular bimorphism of biordered sets. Further, if $e,f\in E$ with $r^e\ler r^f$, then there exists $h\in E$ such that $h\theta=r^e$ and $h\ler f$.
\end{lem}
\begin{proof}
First, suppose that $(e,f) \in D_E$. Suppose that $e\lel  f$, i.e., $e\:f=e$. Then, $r^e\:r^f=r^{(1_e\circ1_f)_e}=r^e$. So $r^e\lel  r^f$ and $(r^e,r^f)\in D_{E(T\mathcal{L}_G)}$. Similarly we can verify for the preorders $\ler $, $(\lel )^{-1}$ and $(\ler )^{-1}$. Hence (BM1) is satisfied.

Also, using Remark \ref{rmkic}, we can see that
$$(e\:f)\theta=r^{e\:f}=r^e\:r^f=(e\theta)\:(f\theta).$$
Hence the condition (BM2) is also satisfied and $\theta$ is a bimorphism.

Now, if $h\in \mathcal{S}(e,f)$, then as above, we can easily verify that $r^h\in \mathcal{S}(r^e,r^f)$; thus (RBM) is satisfied and so $\theta$ is a regular bimorphism from $E$ to $E(T\mathcal{L}_G)$.

The rest of the statement of the lemma directly follows from \cite[Proposition 2.14]{mem}.
\end{proof}

\begin{pro}
The functor	$\bar{F}\colon \mathcal{R}_{G} \to \mathcal{R}_{T\mathcal{L}_G}$ is a local isomorphism.
\end{pro}
\begin{proof}
Lemma \ref{lemF} shows that $\bar{F}$ is a well-defined functor. To show that $\bar{F}$ is a local isomorphism, we need to show that $\bar{F}$ is inclusion preserving, fully faithful and for each $c \in v\mathcal{C}$, $F_{|\langle c \rangle}$ is an isomorphism of the ideal $\langle c \rangle$ onto $\langle F(c) \rangle$. First observe that for an inclusion $\langle e,1_e,f]\in \mathcal{R}_G$, we have $\bar{F}(\langle e,1_e,f])=\lambda(r^e,r^e,r^f)$. Then by dual of \cite[Proposition IV.13(d)]{cross}, the morphism $\lambda(r^e,r^e,r^f)$ is an inclusion in $\mathcal{R}_{T\mathcal{L}_G}$ and so $\bar{F}$ is inclusion preserving.

Now we proceed to show that $\bar{F}$ is fully-faithful. Suppose that $\langle e,\alpha,f]$ and $\langle e,\beta,f]$ are two morphisms in $\mathcal{R}_{G}$ from $\overrightarrow{e}$ to $\overrightarrow{f}$ such that $\bar{F}(\langle e,\alpha,f])=\bar{F}(\langle e,\beta,f])$. Then, $\lambda(r^e,r^\alpha,r^f)=\lambda(r^e,r^\alpha,r^f)$, that is, $r^\alpha=r^\beta$. In particular,  for $\overleftarrow{f}\in v\mathcal{L}_G$, we have $r^\alpha(\overleftarrow{f})=r^\beta(\overleftarrow{f})$. That is, $[f,(1_f\circ\alpha)_{h_1},\mathbf{r}(\alpha) \rangle = [f,(1_f\circ\beta)_{h_2},\mathbf{r}(\beta) \rangle $. But since $\mathbf{d}(\alpha)\ler  f$, we have $f\mathbf{d}(\alpha)=\mathbf{d}(\alpha)$, i.e., $(1_f\circ\alpha)_{h_1}=\alpha$. Also since $\mathbf{d}(\beta)\ler  f$, we have $(1_f\circ\beta)_{h_2}=\beta$. Thus we see that $[f,\alpha,\mathbf{r}(\alpha) \rangle = [f,\beta,\mathbf{r}(\beta) \rangle $. This implies that, in particular, $\overline{\alpha}=\overline{\beta}$. So, we have  $\langle e,\alpha,f] = \langle e,\beta,f]$ and hence $\bar{F}$ is a faithful functor.

To show that $\bar{F}$ is full, let $\lambda(r^e,\gamma,r^f)$ be an arbitrary morphism in $\mathcal{R}_{T\mathcal{L}_G}$ such that $\gamma\in r^f \:{T\mathcal{L}_G}\:r^e$. Then $r^f\gamma=\gamma=\gamma r^e$ and so, $\gamma=r^f\ast \gamma(\overleftarrow{f})$ and $c_\gamma\subseteq c_{r^e}= \overleftarrow{e}$. Observe that $\gamma(\overleftarrow{f})$ is an epimorphism in $\mathcal{L}_G$ and so $\gamma(\overleftarrow{f})=[f,\alpha,\mathbf{r}(\alpha)\rangle$ for some $\alpha\in\mathcal{G}$ such that $\mathbf{d}(\alpha)\ler f$. So,
$$r^\alpha=r^{(1_f\circ\alpha)_h}=r^fr^\alpha=r^f\ast [f,\alpha,\mathbf{r}(\alpha)\rangle=r^f\ast\gamma(\overleftarrow{f})=\gamma.$$
So, $\bar{F}(\langle e,\alpha,f])=\lambda(r^e,r^\alpha,r^f) = \lambda(r^e,\gamma,r^f)$ and $\bar{F}$ is full. Hence $\bar{F}$ is fully-faithful.

Now, to complete the proposition, we need to show that for any $g\in E$, the functor $\bar{F}_{|\langle\overrightarrow{g}\rangle}$  is an isomorphism. Since $\bar{F}$ is already shown to be fully-faithful, it suffices to show that the vertex map $v\bar{F}_{|\langle\overrightarrow{g}\rangle}$ is an order isomorphism. To this end, suppose that $\overrightarrow{e},\overrightarrow{f} \leq_R \overrightarrow{g}$ such that $\bar{F}(\overrightarrow{e})=\bar{F}(\overrightarrow{f})$. So, we have  $r^eT\mathcal{L}_G=r^fT\mathcal{L}_G$, i.e., $r^e\mathrel{\mathscr{R}}r^f\iff r^er^f=r^f$, in the biordered set $E(T\mathcal{L}_G)$. Then in the biordered set $E$, we have $e,f\in E$ such that $r^e\ler r^f$. So by Lemma \ref{lemwr}, there exists $h\in E$ such that $r^h=r^e$ and $h\ler f$. Then for $k\in \mathcal{S}(h,f)$ and using Remark \ref{rmkic}, we have
$$r^{hf}=r^hr^f=r^er^f=r^f.$$
Evaluating the apices of the cones, we get $hf\mathrel{\mathscr{L}}f$. Also since $h\ler f\text{ implies } hf\leqslant f$, we have that $hf=f$, i.e., $f\ler h$. But since $h\ler f$ also, we have $h\mathrel{\mathscr{R}}f$. Then,
\begin{align*}
[e,1_e,e\rangle&=[e,1_{ge},e\rangle &&\text{ since }e\ler g\text{ implies } e=ge\\
&=[e,(1_g\circ1_e)_{k_1},e\rangle &&\text{ where }k_1\in \mathcal{S}(g,e)\\
&=[e,1_g,g\rangle\:[g,1_e,e\rangle &&\text{ using }(\ref{eqncomplg})\\
&=[e,1_g,g\rangle\:[g,1_h,h\rangle &&\text{ since } [g,1_e,e\rangle=r^e(\overleftarrow{g})=r^h(\overleftarrow{g})=[g,1_h,h\rangle\\
&=[e,(1_g\circ1_h)_{k_2},h\rangle &&\text{ using }(\ref{eqncomplg}) \text{ where } k_2\in \mathcal{S}(g,h) \\
&=[e,1_{gh},h\rangle &&\text{ since $gh$ is a basic product in }E\\
&=[e,1_h,h\rangle &&\text{ since }h\ler g\text{ implies } h=gh.
\end{align*}
So by Lemma \ref{lemp}, we see that $\overline{1_e}=\overline{1_h}$. That is, $e=h$ and since $h\mathrel{\mathscr{R}}f$, we have $e\mathrel{\mathscr{R}}f$. Thus, $\overrightarrow{e}=\overrightarrow{f}$ and so $v\bar{F}$ is injective on $\langle\overrightarrow{g}\rangle$.

Now, to prove $v\bar{F}$ is surjective on $\langle\overrightarrow{g}\rangle$, let us suppose that $\epsilon\in E(T\mathcal{L}_G)$ such that $\epsilon \:T\mathcal{L}_G \subseteq r^g\:T\mathcal{L}_G$. Then $\epsilon\ler r^g$ and so $\epsilon\mathrel{\mathscr{R}}\epsilon r^g\leqslant r^g$. Without any loss in generality, we can assume that $\epsilon=\epsilon r^g$ and then we have $\epsilon \leqslant r^g$. If $c_{\epsilon}=\overleftarrow{e}$, then $\overleftarrow{e}=c_\epsilon\subseteq c_{r^g}=\overleftarrow{g}$ and so $e\lel g$. Since $e\mathrel{\mathscr{L}}ge\leqslant g$, we can assume that $e\leqslant g$. Now since  $\epsilon\ler r^g$ implies $r^g\epsilon=\epsilon$, using \cite[Lemma III.3]{cross}, there is a unique epimorphism $\mu\in \mathcal{L}_G(\overleftarrow{g},\overleftarrow{e})$ such that $\epsilon=r^g\ast\mu$. Then,
$$1_{\overleftarrow{e}}= [e,1_e,e\rangle= \epsilon(\overleftarrow{e})=r^g\epsilon(\overleftarrow{e})=r^g(\overleftarrow{e})\ast\mu=[e,1_{eg},g\rangle\mu=[e,1_{e},g\rangle\mu.$$
Observe that $[e,1_{e},g\rangle$ is an inclusion and so $\mu$ is a retraction in $\mathcal{L}_G$ of the form $[g,1_{f},f\rangle$ for some $f\leqslant g$, i.e., $gf=fg=f$. Then, since $\mu$ is an epimorphism and using (\ref{eqnsg1}),
$$\epsilon=r^g\ast [g,1_{f},f\rangle=r^g\ast (r^f(\overleftarrow{g}))^\circ=r^g\:r^f=r^{gf}=r^f.$$
Thus $\bar{F}(\overrightarrow{f})= r^f\:T\mathcal{L}_G= \epsilon \:T\mathcal{L}_G$ with $\overrightarrow{f} \leq_R  \overrightarrow{g}$ and so $\bar{F}$ is surjective on $\langle\overrightarrow{g}\rangle$. This also implies that if $\overrightarrow{f} \leq_R  \overrightarrow{g}\iff r^f \:T\mathcal{L}_G \subseteq r^g\:T\mathcal{L}_G$. Hence the object map $v\bar{F}$ is an order isomorphism on $\langle\overrightarrow{g}\rangle$. Thus, $\bar{F}$ is a local isomorphism.
\end{proof}

\begin{thm}
The triple $(\mathcal{R}_G,\mathcal{L}_G;\Gamma_G)$ is a cross-connection.
\end{thm}
\begin{proof}
First, recall from Theorem \ref{nd} that for a normal category $\mathcal{C}$, the normal category $\mathcal{R}_{T\mathcal{C}}$ is isomorphic to the normal dual $N^\ast\mathcal{C}$. For the case when $\mathcal{C}=\mathcal{L}_G$, an isomorphism  $\bar{G}\colon\mathcal{R}_{T\mathcal{L}_G}\to N^*\mathcal{L}_G$ can be defined as follows. For each $\epsilon T\mathcal{L}_G\in v\mathcal{R}_{T\mathcal{L}_G}$ and for each morphism $\lambda(\epsilon,\gamma,\epsilon') \colon \epsilon T\mathcal{L}_G \to \epsilon'  T\mathcal{L}_G$, we set
\begin{equation}\label{eqng}
v\bar{G}(\epsilon T\mathcal{L}_G):= H(\epsilon;-)\ \text{ and }\ \bar{G}( \lambda(\epsilon,\gamma,\epsilon')  ):=\eta_{\epsilon}\mathcal{L}_G(\widetilde{\gamma},-)\eta_{\epsilon'}^{-1}.
\end{equation}
Here $\widetilde{\gamma}=\gamma(c_{\epsilon'})j(c_\gamma,c_\epsilon)$, as defined in \cite[Lemma III.22]{cross}.

Now we make use of the local isomorphism $\bar{F}\colon \mathcal{R}_{G} \to \mathcal{R}_{T\mathcal{L}_G}$ as defined in (\ref{eqnf}). For a morphism $\langle e, \alpha, f] \in \mathcal{R}_G$, we have $\bar{F} (\langle e, \alpha, f])= \lambda(r^e,r^\alpha,r^f)$ by definition. Since  $\widetilde{r^\alpha}= [f,\alpha,e \rangle$,
we have
$$\bar{F} \bar{G} (\langle e, \alpha, f])=\eta_{r^e} \mathcal{L}_G(  [    f,  {\alpha},e \rangle ,-) \eta_{r^f}^{-1}.$$
We see that the functor $\Gamma_G\colon  \mathcal{R}_G \to N^*\mathcal{L}_G$ defined in (\ref{eqncxn}) is equal to the composition of the local isomorphism $\bar{F}$ and the isomorphism $\bar{G}$. Hence the functor $\Gamma_G$ is a local isomorphism from $\mathcal{R}_G$ to $N^*\mathcal{L}_G$.

Further, for every $\overleftarrow{e}\in v\mathcal{L}_G$, we have $M\Gamma_G(\overrightarrow{e})=MH(r^e;-)=Mr^e$. Then the component $r^e(\overleftarrow{e})=[e,1_e,e\rangle$ is an isomorphism, so that $\overleftarrow{e}\in M\Gamma_G(\overrightarrow{e})$. Thus $\Gamma_G$ is a cross-connection from $\mathcal{R}_G$ to $N^*\mathcal{L}_G$.
\end{proof}

Dually, we can show that $(\mathcal{L}_G,\mathcal{R}_G;\Delta_G)$ defined by the functor $\Delta_G\colon  \mathcal{L}_G \to N^*\mathcal{R}_G$ as follows 
\begin{equation}
v\Delta_G(\overleftarrow{e}):=H(l^e;-) \text{ and } \Delta_G(  [    e,{\alpha},f \rangle ):=\eta_{l^e} \mathcal{R}_G ( \langle  f,{\alpha},e  ]  ,-)\eta_{l^f}^{-1}
\end{equation}
is a cross-connection.

We can now see that the biordered set $E_{\Gamma_G}$ of the cross-connection $(\mathcal{R}_G,\mathcal{L}_G;\Gamma_G)$  is given by the set
$$E_{\Gamma_G}:=\{(\overleftarrow{e},\overrightarrow{e}) : e\in E\}.$$
Here the element $(\overleftarrow{e},\overrightarrow{e})$ corresponds to the pair of cones $(\gamma(\overleftarrow{e},\overrightarrow{e}),\delta(\overleftarrow{e},\overrightarrow{e}))= (r^e,l^e)$. 
Further if we define the preorders $\lel $ and $\ler $ on $E_{\Gamma_G}$ as follows:
\begin{equation}\label{eqpog}
(\overleftarrow{e},\overrightarrow{e})\lel (\overleftarrow{f},\overrightarrow{f}) \iff \overleftarrow{e}\leq_L\overleftarrow{f}\text{ and } (\overleftarrow{e},\overrightarrow{e})\ler (\overleftarrow{f},\overrightarrow{f}) \iff \overrightarrow{e}\leq_R\overrightarrow{f},
\end{equation}
then $E_{\Gamma_G}$ forms a regular biordered set and it is biorder isomorphic to the biordered set $E$ of the inductive groupoid $\mathcal{G}$ (also see \cite{bicxn}). Moreover, it can be easily verified that given $(\overleftarrow{e},\overrightarrow{e}), (\overleftarrow{f},\overrightarrow{f})\in E_{\Gamma_G}$, the morphism $[e,\alpha,f\rangle \in \mathcal{L}_G(\overleftarrow{e},\overleftarrow{f})$ will be the transpose of the morphism $\langle f,\alpha,e] \in \mathcal{R}_G(\overrightarrow{f},\overrightarrow{e})$, relative to the cross-connection $(\mathcal{R}_G,\mathcal{L}_G;\Gamma_G)$.

\subsection{The functor $\mathbb{C}\colon\mathbf{IG}\to\mathbf{CC}$}
Having built the cross-connection $(\mathcal{R}_G,\mathcal{L}_G;\Gamma_G)$ associated with an inductive groupoid $(\mathcal{G},\epsilon)$, now we proceed to show that this correspondence is, in fact, functorial.

Suppose that $(\mathcal{G},\epsilon)$ and $(\mathcal{G}',\epsilon')$ are inductive groupoids with biordered sets $E$ and $E'$, respectively. Let their associated cross-connections be $(\mathcal{R}_G,\mathcal{L}_G;\Gamma_G)$ and $(\mathcal{R}_{G'},\mathcal{L}_{G'};\Gamma_{G'})$, respectively. If $F\colon \mathcal{G}\to \mathcal{G}'$ is an inductive functor between the inductive groupoids $\mathcal{G}$ and $\mathcal{G}'$, then define two functors $F_1:\mathcal{L}_{G}\to\mathcal{L}_{G'}$ and $F_2:\mathcal{R}_{G}\to\mathcal{R}_{G'}$ as follows:
$$vF_1(\overleftarrow{e}):=\overleftarrow{F(e)}\text{ and } F_1([e,\alpha,f\rangle):= [F(e),F(\alpha),F(f)\rangle, \text{ and }$$
$$vF_2(\overrightarrow{e}):=\overrightarrow{F(e)}\text{ and } F_2(\langle e,\alpha,f]):= \langle F(e),F(\alpha),F(f)]. $$

\begin{pro}
The pair $(F_1,F_2)$ is a cross-connection morphism from $(\mathcal{R}_G,\mathcal{L}_G;\Gamma_G)$ to $(\mathcal{R}_{G'},\mathcal{L}_{G'};\Gamma_{G'})$.
\end{pro}
\begin{proof}
Since $F$ is an inductive functor, we can easily verify that $F_1$ and $F_2$ are  well-defined functors. Since $F$ is order preserving, $F_1$ and $F_2$ are inclusion preserving.	

For $e\in E$, if $(\overleftarrow{e},\overrightarrow{e})\in E_{\Gamma_G}$, then since $vF\colon E \to E'$ is a bimorphism, we have $F(e)\in E'$ and so $(\overleftarrow{F(e)},\overrightarrow{F(e)})\in E_{\Gamma_{G'}}$. Then, for an arbitrary $g\in E$,
\begin{align*}
F_1({\gamma(\overleftarrow{e},\overrightarrow{e})(\overleftarrow{g})})&=F_1([g,(1_g\circ1_e)_h,e\rangle)\quad\text{ where }h\in \mathcal{S}(g,e)\\
&= [F(g),F((1_g\circ1_e)_h),F(e)\rangle \\
&= [F(g),(F(1_g)\circ F(1_e))_{h'},F(e)\rangle\quad\text{ where } h':=F(h)\in F(\mathcal{S}(g,e))\\
&\phantom{= [F(g),(F(1_g)\circ F(1_e))_{h'},F(e)\rangle}\quad  \subseteq \mathcal{S}(F(g),F(e)) \text{ as $F$ is a functor}\\ &\phantom{= [F(g),(F(1_g)\circ F(1_e))_{h'},F(e)\rangle}\quad \text{ and $vF$ is a regular bimorphism}\\
&= r^{F(e)}(\overleftarrow{F(g)})\\
&= \gamma(\overleftarrow{F(e)},\overrightarrow{F(e)})(\overleftarrow{F(g)})\\
&= \gamma(F_1(\overleftarrow{e}),F_2(\overrightarrow{e}))(F_1(\overleftarrow{g}))
\end{align*}
Hence the pair of functors $(F_1,F_2)$ satisfies the condition (M1).

Suppose that $(\overleftarrow{e},\overrightarrow{e}), (\overleftarrow{f},\overrightarrow{f})\in E_{\Gamma_G}$ such that the morphism $\langle f,\alpha,e] \in \mathcal{R}_G(\overrightarrow{f},\overrightarrow{e})$ is the transpose of the morphism $[e,\alpha,f\rangle \in \mathcal{L}_G(\overleftarrow{e},\overleftarrow{f})$. Then,
\begin{align*}
F_2(\langle f,\alpha,e])&=\langle F(f),F(\alpha),F(e)]\\
&=[ F(e),F(\alpha),F(f)\rangle^*\\
&= (F_1([e,\alpha,f\rangle))^*
\end{align*}
So, the pair of functors $(F_1,F_2)$ satisfies the condition (M2) also and hence is a morphism of cross-connections.
\end{proof}
The proof of the following theorem is routine with all the preparations we have made so far.
\begin{thm}
The assignments
$$(\mathcal{G},\epsilon)\mapsto (\mathcal{R}_G,\mathcal{L}_G;\Gamma_G) \text{ and }F\mapsto(F_1,F_2)$$	
is a functor $\mathbb{C}\colon\mathbf{IG}\to\mathbf{CC}$.
\end{thm}

\section{The category equivalence}
\label{sec:equiv}

Having built the functors $\mathbb{I}\colon\mathbf{CC} \to\mathbf{IG}$ and the functor $\mathbb{C}:\mathbf{IG}\to\mathbf{CC}$ in the previous sections, now we proceed to prove the category eqivalence between the categories $\mathbf{IG}$ and $\mathbf{CC}$. To this end, we need to show that the functor $\mathbb{I}\mathbb{C}$ is naturally isomorphic to the functor $1_\mathbf{CC}$ and the functor $\mathbb{C}\mathbb{I}$ is naturally isomorphic to the functor $1_\mathbf{IG}$.

We first show that $\mathbb{I}\mathbb{C} \cong 1_\mathbf{CC}$. Suppose that $(\mathcal{D},\mathcal{C};\Gamma)$ is a cross-connection  with an associated inductive groupoid $\mathcal{G}_\Gamma$. Recall that for any two elements $(c,d)$ and $(c',d')$ in the biordered set $E_\Gamma$,
\begin{equation}\label{eqbi2}
(c,d)\lel (c',d')\iff c\subseteq c'\text{  and  }(c,d)\ler (c',d')\iff d\subseteq d'.
\end{equation}
That is, $(c,d)\mathrel{\mathscr{L}}(c',d')\iff c= c'$ and so for an arbitrary $(c,d)\in E_\Gamma$, the canonical image $\overleftarrow{(c,d)}=c$. Similarly, we have $\overrightarrow{(c,d)}=d$.

Also, for a morphism $(f,g)\colon(c,d)\to(c',d')$ in the groupoid $\mathcal{G}_\Gamma$, by Proposition \ref{progl}, we have $\overleftarrow{(f,g)}=f\colon c \to c'$. So, the corresponding morphism $\overleftarrow{(f,g)}$ in the category $(\mathcal{G}_\Gamma)_L$ can be represented by just $[c,f,c'\rangle$. Similarly, we have  $\overrightarrow{(f,g)}=\langle d,g,d']$.

So, if $(\mathcal{D},\mathcal{C};\Gamma)$ is a a cross-connection with an associated inductive groupoid $\mathcal{G}_\Gamma$ and the cross-connection $\mathbb{C}(\mathcal{G}_\Gamma)$ associated with the inductive groupoid $\mathcal{G}_\Gamma$ is given by $(\mathcal{R}_{\mathcal{G}_\Gamma},\mathcal{L}_{\mathcal{G}_\Gamma}; \Gamma_{\mathcal{G}_\Gamma})$, then the functors $(F_{\Gamma1},F_{\Gamma2})$ defined as follows constitute a cross-connection isomorphism. The functors $F_{\Gamma1}\colon\mathcal{C} \to \mathcal{L}_{\mathcal{G}_\Gamma}$ and $F_{\Gamma2}\colon\mathcal{D} \to \mathcal{R}_{\mathcal{G}_\Gamma}$ are given by
$$vF_{\Gamma1} (c):= c\text{ and }F_{\Gamma1}(f):=[c,f,c'\rangle;$$
$$vF_{\Gamma2} (d):= d\text{ and }F_{\Gamma2}(g):=\langle d,g,d'].$$
Hence, for each cross-connection  $(\mathcal{D},\mathcal{C};\Gamma)$, we can easily see that the assignment
 $$(\mathcal{D},\mathcal{C};\Gamma)\mapsto (F_{\Gamma1},F_{\Gamma2})$$
will be a natural isomorphism between the functor $1_{\mathbf{CC}}$ and the functor $\mathbb{I}\mathbb{C}$. That is, for an arbitrary cross-connection morphism $(F_1,F_2)\colon (\mathcal{D},\mathcal{C};\Gamma)\to (\mathcal{D}',\mathcal{C}';\Gamma')$, the following diagram commutes:
\begin{equation*}\label{}
\xymatrixcolsep{3.5pc}\xymatrixrowsep{5pc}\xymatrix
{
	(\mathcal{D},\mathcal{C};\Gamma) \ar[rr]^{(F_{\Gamma1},F_{\Gamma2})} \ar[d]_{(F_1,F_2)}
	&& (\mathcal{R}_{\mathcal{G}_\Gamma},\mathcal{L}_{\mathcal{G}_\Gamma}; \Gamma_{\mathcal{G}_\Gamma}) \ar[d]^{\mathbb{I}\mathbb{C}(F_1,F_2)} \\
	(\mathcal{D}',\mathcal{C}';\Gamma')\ar[rr]^{(F_{\Gamma1'},F_{\Gamma2'})} && (\mathcal{R}_{\mathcal{G}_{\Gamma'}},\mathcal{L}_{\mathcal{G}_{\Gamma'}}; \Gamma_{\mathcal{G}_{\Gamma'}})
}
\end{equation*}

Now, conversely suppose $(\mathcal{G},\epsilon)$ is an inductive groupoid such that $\mathbb{C}(\mathcal{G})=(\mathcal{R}_G,\mathcal{L}_G,\Gamma_G)$, then define the functor $F_\mathcal{G}\colon\mathcal{G}\to \mathcal{G}_{\mathbb{C}(\mathcal{G})} \subseteq \mathcal{L}_G\times\mathcal{R}_G$ as follows:
$$vF_\mathcal{G}(e):=(\overleftarrow{e},\overrightarrow{e})\text{ and }F_\mathcal{G}(\alpha):=([\mathbf{d}(\alpha),\alpha,\mathbf{r}(\alpha)\rangle,\langle\mathbf{d}(\alpha),\alpha^{-1},\mathbf{r}(\alpha)]).$$
We can easily verify that $F_\mathcal{G}$ is an inductive isomorphism. Further, for an inductive functor $F\colon \mathcal{G}\to\mathcal{G}'$, the assignment
$$\mathcal{G} \mapsto F_\mathcal{G}$$
makes the following diagram commute:
\begin{equation*}\label{}
\xymatrixcolsep{3pc}\xymatrixrowsep{4pc}\xymatrix
{
	\mathcal{G} \ar[rr]^{F_\mathcal{G}} \ar[d]_{F}
	&& \mathcal{G}_{\mathbb{C}(\mathcal{G})} \ar[d]^{\mathbb{C}\mathbb{I}(F)} \\
	\mathcal{G}'\ar[rr]^{F_{\mathcal{G}'}} && \mathcal{G}_{\mathbb{C}(\mathcal{G}')}
}
\end{equation*}
Hence, the functor $\mathbb{C}\mathbb{I}$ is naturally isomorphic with the functor $1_{\mathbf{IG}}$.

Summarising the discussion in this section, we have the following theorem which is the main result of this paper.
\begin{thm}
\label{thm:main}
The category $\mathbf{CC}$ is equivalent to the category $\mathbf{IG}$.	
\end{thm}
\section{Conclusion}
We have described the inductive groupoid associated with a cross-connection, and conversely, we have built the cross-connection associated with an inductive groupoid. We emphasise once again that both constructions have been accomplished within a purely category theoretic framework, free of any semigroup theoretic assumptions, under which the concept of an inductive groupoid and that of a cross-connection were initially conceived almost half a century ago.

Because of certain historical circumstances\footnote{See \cite[Section~1]{indcxn1} for a discussion of these circumstances.}, the category $\mathbf{CC}$ was much less known than the category $\mathbf{IG}$. The recent studies by the first author and Rajan \cite{tx,tlx,var,css} have shown that $\mathbf{CC}$ admits worthwhile applications within semigroup theory. We anticipate that our present result may serve as a starting point for looking for a wider spectrum of applications. As a concrete (though, most probably, difficult) task, one can consider developing an abelian version of the theory of cross-connections, aiming at a new categorical framework for the class of von Neumann regular rings. In a sense, this would bring the theory back to its initial origin --- see our discussion in the introduction.

Another possible development consists in involving categories enhanced with an appropriate topology. Many natural groupoid-based structures like pseudogroups of transformations and \'etale groupoids~\cite{resende} come with inbuilt topology by the very definition, and a topological version of the theory of inductive groupoids has already been considered by Rajan \cite{arr}. This suggests that topological variants of the theory of cross-connections should be possible and might be relevant.

A further direction is associated with the generalisations of the categorical structures involved. Several generalisations of inductive groupoids have been proposed and studied in the recent years (see \cite{armstrong,lawsonordered0,wangureg,swang2018}, for instance). On the other side, the cross-connections of \emph{consistent categories} (which are generalisations of normal categories) have been considered in \cite{conc}. It may be worthwhile to explore the inter-connections between the categorical structures arising from these generalisations. In particular, our results can guide in the characterisation of the cross-connection constructions corresponding to the aforementioned generalisations of inductive groupoids. In this regard, as the anonymous referee has pointed out, ``the cross-connections might prove a useful tool for non-regular semigroups and rings, where inductive groupoids and their generalisations fail, since the latter rely on idempotents whereas the former (in theory at
least) do not''.

\section*{Acknowledgements}
The authors thank the referee for their very meticulous reading of the manuscript, generous comments and several excellent suggestions which have helped improve the readability of the paper.

We acknowledge the financial support by the Ministry of Science and Higher Education of the Russian Federation (Ural Mathematical Center project No. 075-02-2021-1387).

\bibliographystyle{plain}

\end{document}